\documentclass[12pt, reqno]{amsart}

\usepackage[english]{babel}
\usepackage{amsthm,verbatim} 
\usepackage{amssymb}
\usepackage{amsmath}  
\usepackage{hyperref}
\usepackage{amsfonts}
\usepackage{enumitem}
\usepackage[nocompress]{cite}
\usepackage{stmaryrd}
\usepackage{fancyhdr}
\usepackage{mathtools}
\usepackage{caption}
\usepackage{subcaption}
\usepackage{tikz-cd}
\usepackage{stackrel}
\usepackage{amsthm}
\usepackage{float}
\usepackage{mathtools, nccmath}
\DeclareFontFamily{U}{mathx}{\hyphenchar\font45}
\DeclareFontShape{U}{mathx}{m}{n}{
	<5> <6> <7> <8> <9> <10>
	<10.95> <12> <14.4> <17.28> <20.74> <24.88>
	mathx10
}{}
\DeclareSymbolFont{mathx}{U}{mathx}{m}{n}
\DeclareFontSubstitution{U}{mathx}{m}{n}
\DeclareMathAccent{\widebar}{0}{mathx}{"73}

\usepackage{graphicx,color}
\usepackage{epstopdf}
\usepackage[left=1.1in, right=1.1in, top=1.1in, bottom=1.1in]{geometry}
\usepackage{mathrsfs}  
\usepackage{etoolbox}
\patchcmd{\endproof}
  {\Qed}
  {\hfill \Qed}

\newtheorem{theorem}{Theorem}[section]

\newtheorem{lemma}[theorem]{Lemma}
\newtheorem{proposition}[theorem]{Proposition}

\newtheorem{assumption}{Assumption}
\theoremstyle{definition}

\theoremstyle{remark}

\newtheorem{examplex}{Example}[section]
\newenvironment{example}
{\pushQED{\qed}\examplex}
{\hfill\popQED\endexamplex}

\newtheorem{remarkx}[theorem]{Note}

\newtheorem{notex}[theorem]{Note}
\newenvironment{Note}
{\pushQED{\qed}\remarkx}
{\hfill\popQED\endnotex}

\newcommand{\RNum}[1]{\uppercase\expandafter{\romannumeral #1\relax}}

\newcommand{\C}{C}
\newcommand{\N}{\mathbb{N}}
\newcommand{\E}{\mathbb{E}}
\newcommand{\cL}{\mathcal{L}}

\newcommand{\cP}{\mathcal{P}}

\newcommand{\pa}{\partial}
\newcommand{\lv}{\left\vert}
\newcommand{\rv}{\right\vert}
\newcommand{\be}{\begin{equation}}
	\newcommand{\ee}{\end{equation}}

\newcommand{\drift}{g}
\newcommand{\diffusion}{A}

\newcommand{\cb}{\kappa}

\newcommand{\yfourder}{m_1}
\newcommand{\yfourdery}{m_2}
\newcommand{\ytwoderyonex}{m_3}
\newcommand{\ytwoderyoney}{m_4}
\newcommand{\ytwoderx}{m_5}
\newcommand{\ytwoderxy}{m_6}
\newcommand{\ytwodery}{m_7}
\newcommand{\ytwoderyy}{m_8}
\newcommand{\onederxg}{m_9}

\newcommand{\intPos}{\int_0^\infty}
\newcommand{\R}{\mathbb{R}}

\newcommand{\fastExp}[1]{\mathbb{E}_{}}

\newcommand{\eps}{^{\epsilon}}

\newcommand{\cPbar}{\bar \cP}
\newcommand{\Xbar}{\bar X}

\newcommand{\SGfastsx}[2]{{P}_{#1}^{#2}}
\newcommand{\SGfastsxn}[2]{P_{#1}^{#2}}

\newcommand{\Lfastx}[1]{\cL^{#1}}

\newcommand{\LfastxDeri}{\frac{\partial \cL^{x}}{\partial x_i}}
\newcommand{\LfastDeri}[1]{\frac{\partial \cL^{#1}}{\partial x_i}}
\newcommand{\LfastDerj}[1]{\frac{\partial \cL^{#1}}{\partial x_j}}
\newcommand{\LfastxDeriDerj}{\frac{\partial^2 \cL^{x}}{\partial {x_i x_j}}}

\newcommand{\slowDer}[1]{\partial_{x_{#1}}}
\newcommand{\slowDerDer}[2]{\partial_{x_{#1}x_{#2}}}
\newcommand{\slowDerDerDer}[3]{\partial_{x_{#1}x_{#2}x_{#3}}}
\newcommand{\slowDerDerDerDer}[4]{\partial_{x_{#1}x_{#2}x_{#3}x_{#4}}}

\newcommand{\pow}[2]{m^{#1}_{#2}}

\newcommand{\fastDer}[1]{\partial_{y_{#1}}}
\newcommand{\fastDerDer}[2]{\partial_{y_{#1}y_{#2}}}

\newcommand{\fastFixed}{Y^{x,y}}

\newcommand{\fastFixedI}{Y^{x,y}}

\newcommand{\funcSpace}[2]{\mathrm{Poly}_{#1,#2}}

\newcommand{\Vseminorm}[2]{\lvert #1 \rvert_{#2}}
\newcommand{\Vnorm}[2]{\lVert #1 \rVert_{#2}}

\begin{document}
	\title[Poisson Equations and Uniform in Time Averaging]{Poisson Equations with locally-Lipschitz coefficients and Uniform in Time Averaging for Stochastic Differential Equations via Strong Exponential Stability}

	\author{D. Crisan$^{(1)}$, P. Dobson$^{(2)}$, B. Goddard$^{(3)}$, M. Ottobre$^{(4)}$, I. Souttar$^{(5)}$}
	
	\address{(1) Department of Mathematics, Imperial College London, London, SW7 2AZ, UK. d.crisan at imperial.ac.uk}
\address{(2) Maxwell Institute for Mathematical Sciences and Mathematics Department, Heriot-Watt University, Edinburgh EH14 4AS, UK, p.dobson\_1 at hw.ac.uk}
\address{(3) School of Mathematics and Maxwell Institute for Mathematical Sciences, University of
Edinburgh, Edinburgh EH9 3FD, UK. bgoddard at ed.ac.uk}
\address{(4) Maxwell Institute for Mathematical Sciences and Mathematics Department, Heriot-Watt University, Edinburgh EH14 4AS, UK. m.ottobre at hw.ac.uk}

\address{(5) Warwick Mathematics Institute, University of Warwick, Coventry, CV4 7AL, UK. iain.souttar at warwick.ac.uk}

	\maketitle
\begin{abstract}
We study averaging for Stochastic Differential Equations (SDEs) and Poisson equations. We succeed in obtaining a  {\em uniform in time} (UiT) averaging result,  with a rate,  for fully coupled SDE models with super-linearly growing coefficients. This is the main result of this paper and it is,  to the best of our knowledge,  the first UiT multiscale result with a rate. More precisely, the main feature of our averaging theorem is that it holds uniformly in time; the technique of proof we use  gives, as a biproduct, a rate of convergence as well. Very few UiT averaging results exist in the literature, and they  almost exclusively apply to  multiscale systems of Ordinary Differential Equations.  Among these few,  none of those we are aware of comes with a rate of convergence.  
The UiT nature of this result (which is its main feature) and the fact that the main theorem comes with a rate of convergence as well, make it important as theoretical underpinning for a range of applications, such as applications to statistical methodology, molecular dynamics etc.  Key to obtaining  both our  UiT averaging  result and to enable dealing with the super-linear growth of the coefficients (of the slow-fast system and of the associated Poisson equation) is conquering   exponential decay in time of the space-derivatives of appropriate Markov semigroups. We refer to semigroups which enjoy this property as being {\em Strongly Exponentially Stable}. 

There are various approaches in the literature to proving averaging results. The analytic approach we take here requires studying a family of Poisson problems associated with the generator of the (fast component of the) SDE dynamics. The study of Poisson equations in non-compact state space is notoriously difficult, with current literature  mostly covering the case when the coefficients of the Partial Differential Equation (PDE) are either bounded or satisfy linear growth assumptions (with the latter case having  been achieved only recently). In this paper we  treat Poisson equations  on non-compact  state spaces for coefficients that can  grow super-linearly. In particular,  we demonstrate how Strong Exponential Stability can be employed not only to prove the UiT result for the slow-fast system  but also to overcome some of the technical hurdles in the analysis of Poisson problems.  Poisson equations are essential tools in both probability theory and PDE theory. Their vast range of applications includes the study of  the asymptotic behaviour of solutions of parabolic PDEs, the treatment of multi-scale and homogenization problems as well as the theoretical analysis of approximations of solutions of Stochastic Differential Equations (SDEs). So our result on Poisson equations is of independent interest as well.   
	\vspace{5pt}
    {\sc Keywords.} Averaging methods for Stochastic differential equations, Poisson equations, Uniform in time approximations, Derivative estimates for Markov Semigroups.
    
\vspace{5pt}
{\sc AMS Classification (MSC 2020).} 60J60, 
60H10, 35B30, 34K33, 34D20, 47D07, 65M75.
\end{abstract}

\tableofcontents

\section{Introduction}
	This paper is concerned with the problem of obtaining {\em uniform in time} averaging results, with a rate of convergence,  for fully coupled systems of Stochastic Differential Equations (SDEs). An important preliminary step to obtain such results is the study of so called `Poisson equations with a parameter'. The study of Poisson equations is also of independent interest as such equations play a pivotal role  both in PDE theory and in probability theory -- to obtain Functional Central Limit theorems \cite{landim,pavliotis2008multiscale}, in the study of  large deviations \cite{spiliopoulos2014fluctuation, morse2017moderate,dupuis2012large},  or in approximation theory, for example as a cornerstone in Stein's method \cite{barbour1990stein} --  so  we comment on these two topics, averaging and Poisson equations, separately, starting from the former.  
	
	{\bf Averaging.}  Consider the following slow-fast system	
	\begin{align}
		dX_{t}^{\epsilon,x,y} & =  b(X_{t}^{\epsilon,x,y}, Y_{t}^{\epsilon,x,y}) dt + \sqrt{2}\sigma(X_{t}^{\epsilon,x,y}, Y_{t}^{\epsilon,x,y}) \, dW_t \label{slowfullycoupled} \\ 
		dY_{t}^{\epsilon,x,y} & = \frac{1}{\epsilon}g(X_{t}^{\epsilon,x,y},Y_{t}^{\epsilon,x,y} ) dt + \sqrt{\frac{{2}}{{\epsilon}}} a(X_{t}^{\epsilon,x,y}, Y_{t}^{\epsilon,x,y})\, dB_t  \label{fastfully coupled} 
	\end{align}
	with initial datum $(X_0, Y_0)=(x,y)\in \mathbb{R}^n \times \mathbb{R}^d$. Here $0<\epsilon \ll 1$ is a small parameter,  $(X_{t}^{\epsilon,x,y}, Y_{t}^{\epsilon,x,y})$ takes values in $ \R^n \times\R^d$, $b: \R^n \times \R^d \rightarrow \R^n$, $\sigma: \R^n\times \R^d \rightarrow \R^{n\times n}$,  $a:\R^n \times \R^d \rightarrow \R^{d\times d}$, $g:\R^n\times\R^d \rightarrow \R^d$ and, finally, $W_t$ and $B_t$, respectively, are  $n-$dimensional and  $d-$dimensional standard Brownian motions, respectively, assumed to be independent of each other.
	
	The intuitive description of the  classical averaging paradigm proceeds as follows. First,  we consider the dynamics in \eqref{fastfully coupled} with $X_{t}^{\epsilon,x,y}=x$ fixed, i.e. we consider the SDE
	\begin{equation}\label{fasttfully coupled}
		dY_t^{\epsilon, x,y}  = \frac{1}{\epsilon}g(x,Y_t^{\epsilon, x,y}) dt + \frac{\sqrt{2}}{\sqrt{\epsilon}} a(x, Y_t^{\epsilon, x,y})\, dB_t \,, \quad Y_0^{\epsilon, x,y}=y\,.
	\end{equation}
	Assuming that for every $x$ the above evolution is ergodic,  with invariant measure $\mu^{x}(dy)$, under appropriate assumptions on the coefficients,  when  $\epsilon\rightarrow 0$ one has
	$$
	\E f(Y_t^{\epsilon, x,y}) \rightarrow \int_{\R^d} f(y)d\mu^{x}(y)\, ,
	$$
	for every $f \in C_b(\R^d)$ (throughout $C_b(\R^d)$ is the set of continuous and bounded real valued functions on $\R^d$), details in Section \ref{sec:sec3}. We emphasize that the invariant measure $\mu^{x}$ does depend on the parameter $x$.  Because $Y_t^{\epsilon, x,y}$ in \eqref{fastfully coupled} equilibrates much faster than $X_t^{\epsilon,x,y}$, one expects that, as $\epsilon \rightarrow 0$, the dynamics \eqref{slowfullycoupled} should be approximated, at least over finite time horizons $[0,T]$, by the so-called {\em averaged dynamics}, i.e. by the following $\R^n$-valued SDE:
	\begin{equation}\label{eqn:averagedfully coupled}
		d\bar{X}_{t}^{x} = \bar b(\bar{X}_{t}^{x}) dt +  \sqrt{2}\hat{\sigma}(\bar{X}_{t}^{x}) dW_t, \quad \bar X_0 = x,    
	\end{equation}
	where 
	\begin{equation}\label{avgDef}
		\widebar{b}(x):= \int_{\R^d} b(x,y) d\mu^{x}(y),
	\end{equation}
	and $\hat{\sigma}$ is a square root\footnote{Clearly $\hat{\sigma}$ is not univocally determined. This is compatible with the fact that the process associated to a generator is not unique and with the fact that we will study weak-type convergence, see \cite[Remark 11.2]{pavliotis2008multiscale} on this point.} of the matrix
	$\bar{\Sigma}(x):= \int_{\R^d} \Sigma(x, y) d\mu^{x}(y)
	:=  \int_{\R^d} \sigma(x, y)\sigma^T(x,y) d\mu^{x}(y)$, where $\sigma^T$ denotes the transpose of the matrix $\sigma$.

	Averaging (and, more generally, homogenization) results can be obtained in a number of ways but, to the best of our knowledge,  the existing techniques can be traced back to a variation of  either one of two main approaches, namely either  a {\em functional approach}, which is the one we adopt in this paper, or a   more  probabilistic one, as introduced by  Khas'minskii  \cite{khas1968averaging} in his seminal paper, see also  \cite{ethierKurtz, liu2020averaging}. We refer to the former as being a functional approach because it hinges on obtaining preliminary results on so-called Poisson equations with parameter, which we will come to later. 
	Irrespective of the approach taken, the convergence of the slow-fast system to the averaged dynamics (whether weak or strong), has ever only been proved to take place over finite time horizons.  That is, broadly speaking, one typically establishes results of the type
	\begin{equation}\label{resultwitht}
	\lv \mathbb E f(X_{t}^{\epsilon,x,y}) - \mathbb E f(\bar{X}_{t}^{x}) \rv \leq \epsilon \, C(t)\, 
	\end{equation}
	for every $f$ in a suitable class of functions, where $C=C(t)$ is a constant dependent  on time (and on $f$ as well as $x$ and $y$). \footnote{In \eqref{resultwitht} we are referring to weak convergence, but clearly the literature has plenty of pathwise or other types of results. The mode of convergence is not what we are trying to emphasize here, we take weak convergence only to fix ideas.} Because of the use of Gronwall's inequality (or similar arguments), this constant is an increasing function of time $t$ \cite{del2019backward}.  
	
	Averaging methods are extremely effective and routinely used in applications to engineering, biology, statistics, molecular dynamics to mention just a few application fields, see e.g. \cite{Carsten,may1979population, coti2020homogenization,gomes2019dynamics,engquist2005multiscale}, with no claim to completeness of references; yet,  a long-standing criticism of such
	techniques is the following: while one can typically only prove that the averaged dynamics is a good approximation of the 
	original slow-fast system for finite-time windows (with estimates that deteriorate in time), the averaged dynamics is often used in practice as an approximation of the long-time behaviour 
	of the slow-fast system. 
	The fact that the  slow-fast system should, under appropriate assumptions on the coefficients, converge to the averaged dynamics uniformly in time (i.e. that one should be able to prove that the constant $C$ in \eqref{resultwitht} is independent of time), has so far only been conjectured on the basis of numerical evidence, see e.g. Section \ref{sec:numerics} and \cite{pavliotis2008multiscale}, where the need in applications for multiscale results which hold uniformly in time (and on non-compact state-space) has been explicitly  advocated. 
	
	In this paper we make it possible to fill this theoretical gap and identify rather general assumptions on the coefficients of the SDEs, under which convergence  (as $\epsilon \rightarrow 0$) of the slow-fast system to the limiting dynamics is actually uniform in time, thereby providing the many works which implicitly use this fact with a firm theoretical basis. 
 We will show in future work (which is at the moment in progress) how leveraging on the theretical results that we produce in this paper can be powerful from an applications' perspective, for example to provide guarantees on statistical methodologies. 
	
	This UiT result (with a rate of convergence, in $\epsilon$)   is novel to this paper and to its inspiring work \cite{barr2020fast}. The difference between the present paper and \cite{barr2020fast} is in the nature of the system at hand, as in \cite{barr2020fast} the fast system is a (relatively straightforward) finite state space jump process and the slow system is an SDE with a specific structure, while here we treat general systems of SDEs (i.e. both the fast and slow processes are SDEs).
	More precisely, besides the fact that our result is uniform in time (UiT), this paper deals with slow-fast systems of SDEs which are fully coupled (i.e. all the coefficients can depend on both $X_{t}^{\epsilon,x,y}$ and $Y_{t}^{\epsilon,x,y}$), evolve in non-compact state space, and whose coefficients are allowed  super-linear growth; in particular, the drift coefficients of both  the slow and the fast system are locally Lipschitz. Furthermore, our convergence result comes with an explicit rate in  $\epsilon$ - feature which is particularly relevant for use in  applications. 

 To the best of our knowledge, there exist truly very few UiT averaging  theorems in the literature, even when searching beyond the case of slow-fast systems of SDEs and looking  in the multiscale literature for PDEs, SPDEs and ODEs.  The only other UiT results we are aware of are  those in \cite{ilyin1998global} (and references therein),  for deterministic systems;  for stochastic dynamics, aside from the work \cite{barr2020fast}, we are only aware of \cite{cheng2023second}, which deals with  Stochastic Partial Differential equations, and of \cite{stoltz2018longtime}, which is inspired by problems in molecular dynamics and treats specifically the case of Langevin dynamics on compact state space, hence the drifts of the SDEs are in gradient form and the diffusion coefficients are constant. However none of these works comes with an explicit convergence rate.   The work \cite{stoltz2018longtime} takes a 
very nice approach, producing a perturbative expansion in $\epsilon$ of the invariant measure of the slow-fast system, and therefore proving convergence of the invariant measure of the slow-fast system to the invariant measure of the averaged equation.
The analysis in \cite{stoltz2018longtime} is made possible by the fact that, due to the specific form of the Langevin dynamics, an explicit expression for the invariant measure of the slow-fast system is a priori known.  In our approach, there is no need to have such knowledge and indeed, while our assumptions do imply that the slow-fast system has an invariant measure, this fact is never used in the analysis.  So the class of SDEs we consider here is truly general.  
 
 \smallskip
	
	Results on multiscale methods for stochastic dynamics typically deal with the case in which the slow-fast system evolves in compact state space, see \cite{pavliotis2008multiscale,xuemei,GlynnMeyn, ethierKurtz} and references therein. The seminal papers \cite{pardoux2001,pardoux2003,Pardoux_2005} (and, later, \cite{rockner2020diffusion})  then paved the way for the non-compact setting, and produced results for slow-fast evolutions  of the type \eqref{slowfullycoupled}-\eqref{fastfully coupled} in $\R^n \times \R^d$, using the functional analytic approach,  as in this paper. The results of \cite{pardoux2001,pardoux2003,Pardoux_2005, rockner2020diffusion} refer to the case in which the coefficients of \eqref{slowfullycoupled}-\eqref{fastfully coupled} are bounded. The difficulty in going from compact  to non-compact  state space and then to unbounded coefficients comes from the fact that the associated Poisson equation inherits such features, i.e.   one needs to study a Poisson problem  which is posed on non-compact state space and associated with a differential operator with unbounded coefficients. This is a notoriously difficult problem, see e.g. \cite[Section 10.2 and 18.4]{pavliotis2008multiscale}. We will come back to this when we comment on our results on Poisson equations. For the time being we mention that  one way of bypassing  this issue is to adopt the diffusion approximation approach instead of the functional analytic one. This  has been done recently in \cite{liu2020averaging}.  Adopting such an approach has allowed the authors of \cite{liu2020averaging} to produce averaging results for the system \eqref{slowfullycoupled}-\eqref{fastfully coupled} when the coefficients are locally Lipschitz (and indeed our assumptions on the coefficients are analogous to those in \cite{liu2020averaging}), but the scheme of proof of \cite{liu2020averaging} does not give an explicit convergence rate in $\epsilon$ and it does not allow to consider the fully coupled regime (in \cite{liu2020averaging} the diffusion coefficient of the slow variable, which we here denote by $\sigma$,  does not depend on the fast variable); on the other hand in \cite{liu2020averaging} all the coefficients are allowed to be time-dependent, a case which we don't treat here. 
	
	The literature on multiscale methods is extremely vast so in the above we have  mentioned only the works which are most relevant to our discussion. Other related works (without any claim to completeness of references) are \cite{XU2018116,veretennikovAvging, liumsstochastic,stochAvgingXu,hairer2004periodic}, some of which do cover the case of linearly growing coefficients. We also flag up the very recent \cite{spiliopolus}, which  treats the case of  SDEs which are non-linear in the sense of McKean.

	From a technical standpoint, the idea that allows us to tackle both the problem of obtaining a UiT averaging result (with a rate) and to study Poisson equations with non globally Lipschitz coefficients hinges on using  exponentially fast decay (in time) of the derivatives of appropriate  Markov semigroups associated with the system \eqref{slowfullycoupled}-\eqref{fastfully coupled} and with \eqref{eqn:averagedfully coupled}, i.e. what we will refer to as {\em strong exponential stability}.  We will explain this further in the last part of this introduction.

	\medskip

	{\bf Poisson equations with parameter.}	The Poisson equations  to be studied in connection with the multiscale system \eqref{slowfullycoupled}-\eqref{fastfully coupled} are of the form 
	\begin{equation}
	    \label{poisprobgeneral}
	(\cL^x u)(x,y)= \phi(x,y)-\bar{\phi}(x) \, , \qquad {x \in \R^n,y\in \R^d, }
	\end{equation}
	where $\phi$ is a regular enough, polynomially growing given function (precise assumptions in Section \ref{sec:mainresults}), $\bar{\phi}:= \int_{\R^d} \phi(x,y) \mu^x(dy)$  (where $\mu^x$ is the measure defined in \eqref{starstar} below) and $\cL^x$ is a differential operator in the $y$ variable only, namely 
	\begin{equation}
	    \label{heart}
	(\cL^x u)(x,y) =( g (x,y) , \nabla_y u(x,y)) +  a(x,y) a(x,y)^T : \text{Hess}_y u(x,y) \,.
	\end{equation}
	Observe that $x \in \R^n$  appears in  \eqref{poisprobgeneral} and in \eqref{heart} purely as a parameter, so that \eqref{poisprobgeneral} can be regarded as a family of Poisson equations dependent on the parameter $x \in \R^n$. In the above and throughout, we use $(\cdot,\cdot )$ to denote the Euclidean scalar product, $\nabla_x$ and $\text{Hess}_x$ are used to denote the gradient and Hessian operators respectively. In Section \ref{subsec:heur}  we will explain how this Poisson problem emerges in the context of averaging. For the time being it suffices to say that the operator $\cL^x$ is the generator of the process $\fastFixed_t:= Y^{1,x,y}_t$ which solves the SDE
	\begin{equation}\label{introfastprocessepsilon=1}
		d\fastFixed_t  = g({x},\fastFixed_t) dt + \sqrt{2} a(x,\fastFixed_t)\, dB_t \,, 
	\end{equation}
	and is obtained from \eqref{fasttfully coupled} by setting $\epsilon=1$; clearly, for each $x$ fixed,  the asymptotic behaviour as $t \rightarrow \infty$ of \eqref{introfastprocessepsilon=1} is the same as the asymptotic behaviour as $\epsilon \rightarrow 0$ of \eqref{fasttfully coupled} and the invariant measure $\mu^x(dy)$ (of either processes)  is precisely the only probability measure such that  
	\begin{equation}
	    \label{starstar}
	\int_{\R^d}(\cL^x f)(x,y) \mu^x(dy)= 0 \,,
	\end{equation}
	for every $f$ in the domain of $\cL^x$. 
	Under our assumptions (see Section \ref{sec:notation}) such an invariant measure exists and, since we will take  $\cL^x$ to be  elliptic,  it is unique and it has a smooth density (for every $x \in \R^n$); with abuse of notation, we still denote the density by $\mu^x(y)$.

	As we have already mentioned, the study of Poisson problems in non-compact spaces is far from trivial  and was first tackled in the  seminal  papers \cite{pardoux2001,pardoux2003,Pardoux_2005}.  To be more precise, when using Poisson equations to obtain averaging results, one needs to establish, in turn:  well-posedness  of  the Poisson problem \eqref{poisprobgeneral} and a representation formula for the solution;  smoothness of the solution $u(x,y)$ in both the $x$ and $y$ variable; and, finally, one needs to quantify how the solution  $u(x,y)$ changes as the parameter $x\in \R^n$ varies.  Related to this problem is the issue of studying how the invariant measure $\mu^x(dy)$ of the frozen process varies as the parameter $x$ varies. {\footnote{For example,  in \cite{hairerPav} $x$ is the inverse temperature of the system so  tracing the dependence in $x$ of $\mu^x$ boils down to tracing the dependence on the inverse temperature.}}  Let us comment on these points in turn, starting from well-posedness of \eqref{poisprobgeneral}. The solution to \eqref{poisprobgeneral} is in general not unique, since constants belong to the kernel of $\cL^x$. However, once we restrict to the set of mean-zero functions, i.e. to the set of functions $f$ such that $\int f(x,y) \mu^x(dy)=0$, then the solution,  if it exists,  is unique (see Section \ref{sec:mainresults} and Section \ref{sec:psneqn}).  Moreover,  one can obtain a useful probabilistic representation formula for such a solution in terms of the process $Y_t^{x,y}$. Indeed,   when the equation is posed on a  bounded domain $D \subset \R^d$ the solution can be represented as 
	$$
	u(x,y) = - \int _0^{\tau} \left(\mathbb E [\phi(x, Y_t^{x,y})]-\bar{\phi}(x)\right)   dt, \quad \tau:= \inf\{t>0 : Y_t^{x,y} \notin D\}. 
	$$
	It is then reasonable to expect that when the equation is posed on $\R^d$ one should  have 
	\be\label{repformulaintro}
	u(x,y) = - \int _0^{\infty} \left( \mathbb E [\phi(x, Y_t^{x,y})]-\bar{\phi}(x)
	\right)   dt \,. 
	\ee
	The above representation formula has been proven in \cite{pardoux2001}, under the assumption that the coefficient $g$ is bounded and $\cL^x$ is uniformly elliptic. Recently, the authors of \cite{cattiaux2011central} have both streamlined the proof of this representation formula and dropped  the boundedness assumption on $g$ (as well as the ellipticity assumption on $\cL^x$), however \cite{cattiaux2011central} does not address the issue of the smoothness of the solution $u(x,y)$.  Regarding smoothness of the solution,  in view of the ellipticity in $y$ of $\cL^x$ the smoothness of $u(x,y)$ in the variable $y$ is straightforward. The difficulty is in obtaining smoothness in the $x$ variable as a priori there is no obvious reason why the solution should change continuously in $x$. Nonetheless it has been shown in \cite[Section 2]{pardoux2003} and then more systematically in \cite{rockner2020diffusion} that, under some assumptions on the coefficients, smoothness in $x$ is inherited from smoothness in $y$; this is done by using so called `transfer formulas', which allow one to express $x$ derivatives in terms of $y$ derivatives, see \cite{rockner2020diffusion} and Section 3. Moreover \cite{pardoux2001,pardoux2003,Pardoux_2005,rockner2020diffusion} obtain a formula for the derivatives in $x$ of $u(x,y)$, which quantifies the way in which the solution changes as $x$ changes. All such results, i.e. smoothness in $x$ and formulas for the $x$ derivatives of $u$, have been obtained so far under the assumptions that the coefficients of $\cL^x$ are bounded e.g \cite{rockner2020diffusion, pardoux2001,pardoux2003,Pardoux_2005}. In  \cite{pardoux2003} the author  observed that it should be possible to relax such an assumption  and the second contribution of this paper consists precisely in removing such an assumption and considering Poisson equations of type \eqref{poisprobgeneral} with unbounded, locally-Lipschitz coefficients. 
	
	\medskip
	
	As we have already pointed out, the technical tool which allows us to obtain UiT averaging results and to study Poisson equations with locally Lipschitz coefficients is the validity of certain Derivative Estimates (DE) for Markov Semigroups.  In particular,  denote by $P_t^x$ the semigroup associated to the process \eqref{introfastprocessepsilon=1} (acting, for each $x$ fixed, on $p \in C(\R^d)$)  and by $\bar{\mathcal P}_t$ the semigroup associated with the process \eqref{eqn:averagedfully coupled} (acting on $f \in C(\R^n)$) - precise definitions in the next section.  Central to our analysis will be estimates on the time-behaviour of the space derivatives of such semigroups, which can be informally  written as follows:\footnote{Precise statements in Section \ref{sec:mainresults}.}
	\be\label{derestgeneral}
	\lv\pa_x (\bar{\mathcal P}_t f)(x) \rv \leq M(x) e^{- Ct}, \quad \lv\pa_y (P_t^x p)(y) \rv \leq M(x,y) e^{- Ct} \,.
	\ee
	In the above, for simplicity, we only wrote first order derivatives (and we used a rather informal one-dimensional notation) but in reality we will need higher order derivatives. We will come back to this later. If a Markov semigroup enjoys the above property (i.e. the space derivatives of some order decay exponentially fast in time) then we say that {\em Strong Exponential Stability} (SES) holds for that semigroup.\footnote{In principle we should be more precise and say e.g. that SES of order $k$ holds if the derivatives up to order $k$ decay exponentially fast; we refrain from doing so, since precise requirements will be spelled out in theorems statements.}  The name is justified by the fact that, assuming the semigroup at hand admits  an invariant measure, the above estimates imply exponential convergence to such a measure (see Lemma \ref{lemma:IntAvgUB} and, for more thorough considerations, \cite{criscassdobsonmichela,crisanottobre}). The above estimates, which are related to those   studied in \cite{crisanottobre}, are different from the ones usually appearing in the literature in two respects: first, similarly to  \cite{crisanottobre, criscassdobsonmichela}, these are not smoothing-type  estimates, see Note \ref{note:smoothing}; second, and further to \cite{crisanottobre, criscassdobsonmichela, dragoni2010ergodicity},  the functions $f$ and $p$ for which they are obtained in this paper need not be bounded and indeed they will be taken to grow polynomially. As a consequence of the latter fact, the constants $M(x)$ and $M(x,y)$ grow polynomially as well, and we trace carefully such a growth. 
	
	We specify that, in contexts different from averaging, there are important works in the literature which deal with UiT results, e.g. \cite{flandoli,durmus2020elementary,del2019backward,mischler2013kac,uitEuler, Feng_2021, heydecker2019pathwise}  to mention just a few; we also flag up the related \cite{hairer2010simple}, where derivative estimates are used in a vein similar to the present paper.   Moreover, the first, second and fourth authors of this paper have been pushing a programme to show how SES is key in proving uniform in time results, in a variety of settings, including convergence of particle systems and numerical methods for SDEs, see e.g. \cite{barr2020fast, criscassdobsonmichela, angeli2023uniform}. An intuitive explanation of why SES is the key concept is given in Note \ref{Rem: der est for UiT}.

	\medskip
	
	The paper is structured as follows. In Section  \ref{sec:notation} we introduce the necessary notation and state our assumptions. In Section \ref{sec:mainresults} we state our main results on Poisson equations and then on averaging. We state all our main results in the non fully coupled regime first, i.e. in the case in which $a=a(x), \sigma=\sigma(x)$ (the drift coefficients $b$ and $g$ are still allowed to depend on both $x$ and $y$) and then in the fully coupled regime (i.e. in the case in which all the coefficients of the slow-fast system \eqref{slowfullycoupled}-\eqref{fastfully coupled} depend on both $x$ and $y$). 
	In Section \ref{sec:sec3} we gather results on ergodicity and SES of the semigroup  $P_t^x$. Such results will be needed to study the Poisson problem \eqref{poisprobgeneral}. In particular in Section \ref{sec:sec3} we prove exponentially fast decay to $\mu^x$ of the semigroup $P_t^x$, {\em with   rate of exponential decay which is  independent of $x$} and we produce DE of the form \eqref{derestgeneral} for  $y$ derivatives up to order four of $P_t^x$. The latter fact  will be the key instrument in tackling the issue of the smoothness in $x$ of the solution of  the Poisson equation, when the coefficients of $\cL^x$ grow super-linearly.   Section \ref{sec:psneqn} contains the proof of our main result on Poisson equations,  Theorem \ref{mainthmpois}. For readability, and because some of these results are of independent interest, the proof is split in various statements. In particular, Lemma \ref{lemma:probRep} is a well-posedness result for the Poisson problem,  and Proposition \ref{lemma:IntAvgDerRepresentation} is the key place where we use SES of $P_t^x$  to prove smoothness in the parameter $x$ of the solution of the Poisson equation. Several comments on Proposition \ref{lemma:IntAvgDerRepresentation} are contained in Note \ref{Rem:der est for Pois unbounded} (which also explains why we need four derivatives of the  semigroup $P_t^x$). Section  \ref{subsec:detailedproof} and Section \ref{sec:boundsonft1} are devoted to the proof of our UiT averaging result. More specifically, in Section \ref{subsec:detailedproof} we explain our strategy of proof, providing first some heuristics (Section \ref{subsec:heur}), then a strategy of proof for the non fully coupled case (Section \eqref{subsec:proofofmainthm}) and then explaining how such a strategy can be extended to the fully coupled case (Section \ref{subsec:fullycoupledsketch}).  The strategy explained in Section \ref{subsec:detailedproof} clarifies that the two main ingredients in our approach are SES of the averaged semigroup $\bar{\mathcal P}_t$ and the study of  Poisson equations, see Note \ref{Rem:poisson for averaging}.  Accordingly, in  Section \ref{sec:boundsonft1} we first obtain DEs for $\bar{\mathcal P}_t$ and then apply the results on Poisson equations stated in previous sections to the Poisson problem at hand. It is important to notice that the DEs for $\bar{\mathcal P}_t$ cannot be obtained in the same way as those for $P_t^x$. Indeed, while $\bar{\mathcal P}_t$ and $P_t^x$ are both Markov Semigroups, the coefficients of the SDE associated with $P_t^x$ are known explicitly, those for the SDE \eqref{eqn:averagedfully coupled} associated to $\bar{\mathcal P}_t$ are not, in the sense that  if $\mu^x$ is not known explicitly then the coefficient $\bar{b}$  (similarly for $\hat{\sigma}$) is only defined via \eqref{avgDef}, which is a highly non-linear expression in $x$, see Note \ref{rem:xavgderivs}. 
	Section \ref{sec:numerics} contains some numerical experiments. We complement our results with a number of examples, both throughout and in Section \ref{sec:numerics}. Note that when $b$ is in gradient form, our results can be applied to Langevin-type dynamics, which are ubiquitous in applications, see e.g. \cite{pavliotis2008multiscale} and references therein, and Section \ref{sec:mainresults}.
	
	Finally, with our scheme of proof passing from the non fully coupled to the fully coupled regime is actually simple and the real difficulty is to understand how to deal with the non fully coupled case first.  The only difference is that in the non fully coupled case one needs to control four $y$ derivatives of the `frozen semigroup' $P_t^x$ and two $x$ derivatives of the averaged semigroup $\bar{\mathcal P}_t$, in the fully coupled regime one needs  the same number of  $y$ derivatives of the frozen semigroup $P_t^x$ and but two additional $x$ derivatives of the `averaged semigroup' $\bar{\mathcal P}_t$, i.e. the fully coupled case is just more `computationally intensive', but conceptually it does not require any new ideas. 
	Hence in  Section \ref{sec:mainresults} we first state  all our results in the non fully coupled case and then in the fully coupled one and the whole  paper refers to the non fully coupled case, with the exception of Section \ref{subsec:fullycoupledresults} and Section \ref{subsec:fullycoupledsketch}, where the fully coupled regime is treated. 
	Because in many cases of interest in applications  the diffusion coefficients are constant, see \cite{hairerPav,coti2020homogenization,Carsten}  and references therein (hence falling in the non fully coupled case) in the non fully coupled regime we express all our theorems in terms of directly verifiable conditions  on the coefficients of \eqref{slowfullycoupled}-\eqref{fastfully coupled} (not just in terms of SES requirements). Similarly, the techniques of this paper can be used, without requiring any new ideas, to study the homogenization regime and the case in which $W_t$ and $B_t$ are not independent of each other. We don't do it here to contain the length of the paper. The extension that would require more care is the one to the hypoelliptic case. All such questions will be covered in future work.

	\section{Notation and Assumptions}\label{sec:notation} 
	From now on, unless otherwise stated, we only refer to the non fully coupled regime, i.e. we consider the following slow-fast system 
	\begin{align}
		dX_{t}^{\epsilon,x,y} & =  b(X_{t}^{\epsilon,x,y}, Y_{t}^{\epsilon,x,y}) dt + \sqrt{2}\sigma(X_{t}^{\epsilon,x,y}) \, dW_t \label{slow} \\ 
		dY_{t}^{\epsilon,x,y} & = \frac{1}{\epsilon}g(X_{t}^{\epsilon,x,y},Y_{t}^{\epsilon,x,y} ) dt + \sqrt{\frac{{2}}{{\epsilon}}} a(X_{t}^{\epsilon,x,y})\, dB_t  \,. \label{fast} 
	\end{align}
	The corresponding `frozen process' of interest is then 

		\begin{equation}\label{fastF}
			d\fastFixed_t  = g({x},\fastFixed_t) dt + \sqrt{2} a({x})\, dB_t \,,  \qquad \fastFixed_0 = y\,,
		\end{equation}
	and the limiting, averaged process is 
			\begin{equation}\label{eqn:averaged}
			d\bar{X}_{t}^{x} = \bar b(\bar{X}_{t}^{x}) dt +  \sqrt{2}\sigma(\bar{X}_{t}^{x}) dW_t, \quad \bar X_0 = x,    \end{equation}
		where $\bar{b}$ is defined as in \eqref{avgDef}. The fully coupled regime is treated separately in Section \ref{subsec:fullycoupledresults} and Section \ref{subsec:fullycoupledsketch}.
	
	
		\subsection{Notation and Assumptions}Let $\bar \cP_t$ be the semigroup associated with the process \eqref{eqn:averaged}, i.e.
	
	\begin{equation} \label{eqn:avgedSGdef}
		(\bar \cP_t f)(x):= \E\left[f(\bar{X}_{t}^{x})\right],  \qquad f \in C_b(\R^n)\,, \end{equation} 
	where $\E$ denotes expectation. It is well-known that under Assumptions \ref{ass:polGrowth}-\ref{DriftAssumpS}, which we will state below, such a semigroup is a classical solution to the PDE (see \cite[Theorem 1.6.2]{cerrai2001second})
	\begin{equation}\label{eq:barPDE}
		\begin{cases}\partial_t(\cPbar_t f)(x)=\widebar{\cL}\,  \cPbar_t f(x)\\
			(\cPbar_0 f)(x)=f(x),
		\end{cases}
	\end{equation}
	where $\bar \cL$ is the second order differential operator formally acting on smooth functions as 
	\begin{equation}\label{bar L}
		(\bar\cL f)(x) := (\bar b (x) , \nabla_x f(x)) + \sigma\sigma^T(x) : \text{Hess}_x f(x).
	\end{equation}
	In the above and throughout $M:N\coloneqq \text{Tr}\{M^TN\}=\sum_{i,j}M_{ij}N_{ij}$  denotes the Frobenius inner product between two matrices $M=(M_{ij})$ and $N=(N_{ij})$ and  $\text{Tr}$ denotes the trace. Moreover, coherently with \eqref{avgDef}, for any function $\psi: \mathbb{R}^n \times \mathbb{R}^d \rightarrow \mathbb{R}$ we let \begin{equation}\label{eqn:barNotation}
		\widebar{\psi}(x):= \int_{\R^d} \psi(x,y) d\mu^{x}(y)\,.
	\end{equation}
	We denote by $\cP\eps_t$ the semigroup associated with the slow-fast dynamics \eqref{slow}-\eqref{fast}, acting on functions $\psi \in C_b(\R^{n}\times \R^d)$, i.e.
	\begin{equation}
		(\cP\eps_t\psi)(x,y):= \E \left[\psi(X_{t}^{\epsilon,x,y}, Y_{t}^{\epsilon,x,y})\right] \,. \label{eqn:fullSGdef}
	\end{equation}
	The generator of this semigroup is the second order differential operator formally defined to act on smooth functions as
	\begin{align}\label{decompLepsilon}
		(\mathfrak{L}_\epsilon \psi)(x,y) =  (\cL_S\psi)(x,y)+ \frac{1}{\epsilon} (\cL^x\psi)(x,y)
	\end{align}
	where
	\begin{align}\label{LFast and LSlow}
			(\cL_S\psi)(x,y) &= (b (x,y) , \nabla_x \psi(x,y)) + \sigma(x) \sigma(x)^T : \text{Hess}_x \psi(x,y)\\
			(\cL^x \psi)(x,y) & = (g (x,y) , \nabla_y \psi(x,y)) + a(x)  a(x)^T : \text{Hess}_y \psi(x,y).  \label{eqn:gen}
	\end{align}
	We emphasise  that $\cL_S$ and $\cL^x$ are differential operators  in the $x$ and $y$ variables only,  respectively, and they correspond to the slow and fast part of the dynamics, respectively. We let
	\begin{equation}\label{eqn:bigDiffdef}
		\Sigma (x) \coloneqq \sigma\sigma^T (x), \quad  A(x) \coloneqq aa^T(x).
	\end{equation}
	We want to compare the dynamics $\Xbar_t$ with the dynamics $X_{t}^{\epsilon,x,y}$, but $X_{t}^{\epsilon,x,y}$ alone does not generate a semigroup. Hence we consider the semigroup $\cP_t^\epsilon$, which corresponds to the system $(X_t^\epsilon,Y_t^\epsilon)$, and restrict our attention to the case in which such a semigroup acts on functions $f\colon \R^n \rightarrow \R$ which depend on the variable $x$ only. Note that while $f$ depends only on the variable $x$, the function $(\cP_t\eps f)(x,y)$ depends on both variables. In other words, we restrict to considering initial value problems where the initial profile is a function independent of $y$: 
	\begin{equation}\label{eq:epsPDE}
		\begin{cases}\partial_t(\cP_t\eps f)(x,y)=\mathfrak{L}_\epsilon (\cP_t\eps f)(x,y)\\
			(\cP_0\eps f){(x,y)}=f(x) \,.
		\end{cases}
	\end{equation}
	

	We will use the following notation:
	
	\begin{itemize}
		\setlength\itemsep{0.5em}
		\item 
		We denote by $P^x_t$ the `frozen semigroup', i.e. the semigroup associated with the process $\fastFixed_t$ (defined in \eqref{fastF}):
		\be\label{infty}	
		(\SGfastsxn{t}{x} f)(y):= \E\left[f(\fastFixed_t)\right],  \qquad f \in C_b(\R^d)\,.
		\ee
		\item For every fixed $x \in \R^n$, the invariant measure of $\fastFixed_t$ is denoted $\mu^x$  (by ergodicity of $\fastFixed_t$, see Section \ref{sec:sec3} such a measure does not depend on $y$).  
		As customary, for any function $\psi : \R^d \rightarrow \R$ and any measure $\nu$ on $\R^d$, we write		\begin{equation*}
			\nu(\psi) \coloneqq \int_{\mathbb{R}^d}\psi(\tilde{y}) \nu(d\tilde{y}).
		\end{equation*}
		When we integrate against $\mu^x$, by \eqref{eqn:barNotation} we have  $$\mu^x(\psi)=\bar \psi(x) \,.$$
		We define the space $L^2(\mu^x)$ as the space of functions $h:\R^d\to\R$ such that $$\int_{\R^d}|h(y)|^2 \mu^x(dy) < \infty.$$
		\item The partial derivative with respect to the $i$-th coordinate of $x \in \R^n$ is denoted $\slowDer{i}$ and for higher derivatives we write $\slowDerDer{i}{j}$. We  will use the multi-index notation $\partial_x^{\gamma}= \partial_{x_1}^{\gamma_1}\ldots \partial_{x_n}^{\gamma_n}$ where $\gamma = (\gamma_1, \ldots, \gamma_n)$ is an index of length $n$ and $\gamma_i \in \mathbb{N}\cup \{0\}$. Similarly, we write $\partial_y^{\gamma}$ where  $\gamma = (\gamma_1, \ldots, \gamma_d)$, for partial derivatives in $y$ (if we differentiate with respect to $y$ then it is understood that $\gamma$ is of length $d$). We also let  $|\gamma|_{*} \coloneqq \sum_{i=1}^d \gamma_i$ (whether the sum is up to $n$ or $d$ will be clear from context).
		\item For a multivariable function $\phi : \mathbb{R}^n \times \mathbb{R}^d \rightarrow \mathbb{R}^n$ and $i\in\{1,\ldots,n\}$ we refer to the $i^{th}$ coordinate as $\phi_i(x,y)$, and similarly for $\phi : \mathbb{R}^n \times \mathbb{R}^d \rightarrow \mathbb{R}^n \times \mathbb{R}^d$ we refer to the $(i,j)^{th}$ element as $\phi_{ij}(x,y)$.
		\item $C^k(\R^n)$  ($C_b^k(\R^n)$, respectively),   will denote the set of continuous (continuous and bounded, respectively) real valued functions on $\R^n$ with continuous (continuous and bounded, respectively) derivatives up to and including order $k$. For $f \in C_b^k(\R^n)$ we define the norm $\lVert f \rVert_{C_b^k}$ as 
		\begin{equation*}
			\lVert f\rVert_{C_b^k} = \sum_{0\leq \lvert\gamma\rvert_*\leq k} \lVert \partial_x^\gamma f\rVert_\infty,
		\end{equation*}
		where $\lVert \cdot \rVert_\infty$ denotes the sup-norm. Analogously, the space $C^{k,\ell}(\R^n\times \R^d)$ 
		is the space of  functions $g:\R^n\times\R^d\to\R$ such that the derivatives $\partial_x^\gamma\partial_y^{\tilde{\gamma}}g$ exist and are  continuous for any $\gamma$ such that $\lvert\gamma\rvert_*\leq k$ and $\tilde{\gamma}$ such that  $\lvert\tilde{\gamma}\rvert_*\leq \ell$.
		\item For any $\phi\in C^{0,k}(\R^n\times\R^d)$ we define the following seminorm 
		\begin{equation}\label{eqn:Vseminormdef}
			\Vseminorm{\phi}{k,m,m'} = \sup_{1\leq \lvert\gamma\rvert_*\leq k}\sup_{x\in\R^n,y\in\R^d} \left\lvert \frac{\partial_y^\gamma \phi(x,y)}{1+\mathbb{I}_{m>0}\lvert x\rvert^m+\mathbb{I}_{m'>0}\lvert y\rvert^{m'}}\right\rvert \, ,
		\end{equation}
		where $\gamma$ varies over all multi-indices of length $d$ and $\mathbb{I}$ is the indicator function.\footnote{The reason we introduce the indicator function in the notation is so that, in the case $m=0$, $\Vseminorm{\phi}{k,0,m'}$ is the constant such that
			$\lvert\partial_y^\gamma \phi(x,y)\rvert\leq\Vseminorm{\phi}{k,m,0}(1+|y|^{m'})$, rather than the constant such that
			$\lvert\partial_y^\gamma \phi(x,y)\rvert\leq\Vseminorm{\phi}{k,0,m'}(2+|y|^{m'})$.} Note that $\Vseminorm{\phi}{k,m,m'}<\infty$ if and only if all $y$ derivatives of $\phi$ up to order $k$ are polynomially bounded with exponent at most $m$ in $x$ and $m'$ in $y$, however the function itself need not be bounded. 
		The next norm we define does include derivatives in $x$ and the function itself:  for $k\in \N\cup\{0\}$ even,  $m,m'\in \N\cup\{0\}$ and $\phi\in C^{k/2,k}(\R^n\times\R^d)$ let 
		\begin{equation*}
			\Vnorm{\phi}{k,m,m'} =  \sup_{ 0\leq 2\lvert\gamma \rvert_*+\lvert\tilde{\gamma}\rvert_*\leq k} \sup_{x\in\R^n,y\in\R^d}\left\lvert \frac{\partial_x^{\gamma}\partial_y^{\tilde{\gamma}} \phi(x,y)}{1+\mathbb{I}_{m>0}\lvert x\rvert^m+\mathbb{I}_{m'>0}\lvert y\rvert^{m'}}\right\rvert,
		\end{equation*}
		where $\gamma$ and $\tilde{\gamma}$ are indices of length $n$ and $d$ respectively.
		We say that $\phi \in \funcSpace{m}{m'}$ if $\phi \in \C^{2,4}(\R^n\times\R^d)$ and $\Vnorm{\phi}{4,m,m'}<\infty$. 
		We will often use the sets $\funcSpace{0}{0}$ and $\funcSpace{0}{m}$ so let us spell out what they contain. A function $\phi $ is in  $\funcSpace{0}{0}$ if and only if $\phi\in \C^{2,4}(\R^n\times\R^d)$ and all the derivatives $\pa_x^{\gamma}\pa_y^{\tilde{\gamma}} \phi$ such that $2|\gamma|_*+|\tilde\gamma|_* \leq 4$ are bounded; $\phi\in \funcSpace{0}{m}$ if and only if $\phi\in \C^{2,4}(\R^n\times\R^d)$  and all the derivatives $\pa_x^{\gamma}\pa_y^{\tilde{\gamma}} \phi$ such that $2|\gamma|_*+|\tilde\gamma|_* \leq 4$ are  bounded in $x$ and polynomially bounded in $y$ by $c(1+|y|^m)$. Motivation for the choice of the norm $\lVert\cdot\rVert_{k,m,m'}$ can be found in Note \ref{Rem:der est for Pois unbounded}. For a multivariable function $\phi \colon \R^n \times \R^d \rightarrow \R^{n_0}$,  $n_0 \geq 0$, we define $\Vnorm{\phi}{k,m,m'} := \max_{1\leq i\leq n_0}\Vnorm{\phi_i}{k,m,m'}$,  for every $k,m,m' \geq 0$.
		\item For any $0<\nu<1$,  $C^{4+\nu}(\R^d)$  is the set of four times differentiable functions whose $4^{th}$ order derivatives are locally $\nu$-H\"older continuous.
	\end{itemize}
	
	We now list our main assumptions and then comment on them in Note \ref{rem:commentonassumptions}.
	\begin{assumption}[Growth of coefficients]\label{ass:polGrowth} Recall $\Sigma$ and $A$ defined by \eqref{eqn:bigDiffdef}. 
		\begin{enumerate}[label=\textnormal{[C\arabic*]},ref={[C\arabic*]}]
			\item\label{ass:polGrowthDrift}
			There exist $\pow{b}{x},\pow{b}{y}>0$ such that $b_i \in \funcSpace{\pow{b}{x}}{\pow{b}{y}}$ for all $1 \leq i \leq n$.
			\item\label{ass:polGrowthDiff}
			$\Sigma_{ij} \in \funcSpace{0}{0}$ for all $1 \leq i,j \leq n$.
			\item 	\label{ass:polGrowthDriftFast}
			There exists $0<\nu<1$ such that for each $x \in \R^n$, $g_i(x, \cdot) \in C^{4+\nu}(\R^d)$, and there exists $\pow{g}{y}>0$ such that $g_i \in\funcSpace{0}{\pow{g}{y}}$ for all $1 \leq i \leq d$.
			\item 	\label{ass:polGrowthDiffFast} $A_{ij} \in  \funcSpace{0}{0}$ for all $1 \leq i,j \leq d$.
		\end{enumerate}
	\end{assumption}
	
	
	
	\begin{assumption}[Uniform ellipticity]\label{UniformEllip}
		There exist constants $\lambda_{-}$, and  $\lambda_{+}$ such that 
		\begin{equation}\label{eqn:unifEllip1}
			0 < \lambda_{-} \leq \left(\frac{A(x)\xi}{|\xi|}, \frac{\xi}{|\xi|}\right) \leq\lambda_{+}\, , \quad 
			\mbox{for every } x \in \R^n \mbox{ and }  \xi \in \mathbb{R}^d \setminus \{0\}.
		\end{equation}
	and
		\begin{equation}\label{eqn:unifEllip2}
			0 < \lambda_{-} \leq \left(\frac{\Sigma(x)\xi}{|\xi|}, \frac{\xi}{|\xi|}\right) \leq\lambda_{+}\, , \quad \mbox{for every } x \in \R^n \mbox{ and }  \xi \in \mathbb{R}^n \setminus \{0\}.
		\end{equation}
		
	\end{assumption} 
	\begin{assumption}[Lyapunov condition for frozen process]
		
		
		\label{DriftAssump}
		For every integer $k\geq 0$, there exist constants $r_k, C_k > 0$ (independent of $x \in \mathbb{R}^n$ and $y \in \R^d$) such that
		\begin{equation} \label{eq:driftassump}
			\left(g(x,y), y\right) + (k-1)a(x):a(x) \leq -r_k|y|^2 + C_k
		\end{equation}
		for every $x\in \mathbb{R}^n$ and $y \in \mathbb{R}^d$.
		
	\end{assumption}
	
	
	\begin{assumption}[Lyapunov condition for slow process]\label{DriftAssumpS}
		There exist constants $\tilde{r}, \tilde{C} > 0$ (independent of $x\in \R^n, y \in \mathbb{R}^d$), such that
		\begin{equation*} 
			\left(b(x,y), x\right) + (4\pow{b}{x}-1)\sigma(x)\colon \sigma(x) \leq -\tilde{r}|x|^2 + \tilde{C}
		\end{equation*}
		for every $x\in \mathbb{R}^n$ and $y \in \mathbb{R}^d$.

	\end{assumption}
	
	\begin{assumption}[Drift condition for the frozen process]\label{ass:SGcond}
		%
		There exist $\cb,\zeta_1,\zeta_2,\zeta_3>0$ independent of $x,y,\xi$ such that for any $x\in \R^n, y\in \R^d, \xi\in \R^d$ we have
		\begin{equation}\label{eq:driftcondition}
			\begin{aligned}
				&2 \sum_{i,j=1}^d \partial_{y_i}g_j(x,y) \xi_i\xi_j+\sum_{i,j=1}^d\zeta_1(\partial_{y_i}\partial_{y_j}g(x,y), \xi)^2 +\sum_{i,j,k=1}^d\zeta_2(\partial_{y_i,y_j,y_k}g(x,y), \xi)^2\\
				&+\sum_{i,j,k,\ell=1}^d\zeta_3(\partial_{y_i,y_j,y_k,y_\ell}g(x,y), \xi)^2 \leq -\cb \lvert\xi\rvert^2.
			\end{aligned}
		\end{equation}
	\end{assumption}
	If  Assumption \ref{ass:polGrowth} \ref{ass:polGrowthDrift}, \ref{ass:polGrowthDiff} holds, 
	then	there exist some  constants $\yfourder,\dots, \ytwoderyy \geq 0$  such that the following holds
	\begin{align}
		&|\partial^{\gamma}_yb_i(x,y)| \leq C\left(1+|x|^{\yfourder}+|y|^{\yfourdery}\right)\quad \text{ for all }|\gamma|_*=4, \nonumber \\
		&|\partial^\gamma_y\partial^{\tilde{\gamma}}_xb_i(x,y)| \leq C\left(1+|x|^{\ytwoderyonex}+|y|^{\ytwoderyoney}\right) \quad\text{ for all } |\gamma|_*=2,|\tilde{\gamma}|_*=1, \nonumber \\
		&|\partial^{\tilde{\gamma}}_xb_i(x,y)| \leq C\left(1+|x|^{\ytwoderx}+|y|^{\ytwoderxy}\right)\quad \text{ for all }|\tilde{\gamma}|_*=2,  \label{starstarAss1}\\
		&|\partial^\gamma_yb_i(x,y)| \leq \Vseminorm{b}{2,\ytwodery ,\ytwoderyy}\left(1+|x|^{\ytwodery}+|y|^{\ytwoderyy}\right)\quad \text{ for all }1 \leq |\gamma|_* \leq 2, \nonumber \\
		&|\partial^{\gamma}_x\Sigma_{ij}(x)| \leq K_\Sigma\quad\text{ for all } 1\leq |\tilde{\gamma} |_* \leq2, \nonumber 
	\end{align}
	for some $C,K_\Sigma\geq0$ and for  all $i,j=1,\ldots,d, $ $x\in\R^n,y\in\R^d$, where in the above $\gamma,\tilde{\gamma}$ are indices of length $d$ and $n$ respectively.
	Furthermore, by Assumption \ref{ass:polGrowth} \ref{ass:polGrowthDriftFast},\ref{ass:polGrowthDiffFast} there exist constants $\onederxg, K_A\geq 0$ such that for all $i,j=1\ldots d$,
	\begin{align*}
		&|\partial^{\tilde{\gamma}}_xg_i(x,y)| \leq K_g\left(1+|y|^{\onederxg}\right)\quad \text{ for all }|\tilde{\gamma} |_*= 1, \\
		&|\partial^{\tilde{\gamma}}_xA_{ij}(x)| \leq K_{A}\quad\text{ for all } 1 \leq |\tilde{\gamma} |_* \leq 2,
	\end{align*}
	where we set $K_g := \max_{1\leq |\tilde{\gamma}|\leq 2}\Vnorm{\partial^{\tilde{\gamma}}_{x} g}{0,0 ,\onederxg}.$
	
	\begin{assumption}[Drift condition for slow process]\label{ass:avgSGcond}
		Let Assumptions \ref{ass:polGrowth} to \ref{ass:SGcond} hold. 
		Assume that there exists $\zeta>0$ independent of $x,y,\xi$ such that for any $x, \xi\in \R^n, y \in \R^d$ we have
		\begin{align}\label{eq:avgdriftcondition}
			\begin{split}
				\sum_{i,j=1}^n \partial_{x_i}b_j(x,y) \xi_i\xi_j 
				\leq -  \Bigg(n\frac{2D_0}{\cb} \Vseminorm{b}{2,\ytwodery,\ytwoderyy} D_b(x)
				+\zeta(1+|x|^{\yfourder} +|x|^{  \ytwoderyonex} +|x|^{  \ytwoderx})+\frac{K_\Sigma^2n^3}{4\lambda_-}\Bigg)|\xi|^2
			\end{split}
		\end{align}
		where $\cb$ is given by \eqref{eq:driftcondition},
		$
		D_0 := \sqrt{3}\max\{(\zeta_1\lambda_-)^{-1/4} ,1\}d\left( d+\sqrt{\zeta_1\lambda_-}d^2\right)^{1/2}
		$
		and 
		\begin{align*} 
			D_b(x) &:= 
			\left(\left(K_g+K_{A}\right)\left(1+2\sqrt{\frac{C'_{2\ytwoderyy}}{r'_{2\ytwoderyy}}}\right)+K_g\left(\frac{C'_{\ytwoderyy+\onederxg}}{r'_{\ytwoderyy+\onederxg}}+\frac{C'_{\onederxg}}{r'_{\onederxg}}+ \sqrt{\frac{C'_{2\ytwoderyy}}{r'_{2\ytwoderyy}}}\frac{C'_{\onederxg}}{r'_{\onederxg}}\right)\right) \\
			&+\mathbb{I}_{\ytwodery> 0}\left( K_g+K_{A} + K_g\frac{C'_{\onederxg}}{r'_{\onederxg}} \right)|x|^{\ytwodery}.\end{align*}
		
		In the above, for $m>0$, $C_m'$ and $r_m'$ are defined in Lemma \ref{prop:momentBounds};  $C'_0 =0$ and $r'_0=1$.  
		
		
	\end{assumption}
		
	\begin{Note}\label{rem:commentonassumptions}
		Let us comment on each of the above assumptions in turn.
		\begin{enumerate}
			\item Assumption \ref{ass:polGrowth} requires that all the coefficients are $C^{2,4}$, the drift coefficients $b$ and $g$ have at most polynomial growth and the diffusion coefficients $\Sigma$ and $A$ are bounded. 
			\item Assumption \ref{UniformEllip} is a uniform ellipticity assumption and is used to ensure differentiablility of the semigroups $\cP_t^\epsilon, \SGfastsxn{t}{x}$ and, $\bar{\cP}_t$. More precisely, \eqref{eqn:unifEllip1} gives differentiability in $y$ of $\cP_t^\epsilon f$ and $P^x_t f$, while \eqref{eqn:unifEllip2} gives differentiability in $x$ of $\bar \cP_t f$ and $\cP_t^\epsilon f$.  
			
			\item Assumption \ref{DriftAssump} is relatively standard in the literature, see for example \cite{liu2020averaging}, and (since it is enforced for every $k$) it ensures  that all the moments of \eqref{fastF}  are  uniformly bounded in $t$. In principle we require for Assumption \ref{DriftAssump} to hold for every $k$. However,  since $a$ is bounded (as $A$ is bounded, by Assumption \ref{ass:polGrowth}),  if Assumption \ref{DriftAssump} holds with $k=1$ then it holds for every integer $k\geq 1$ (if it is true for $k=1$ then we can take $r_k = r_1$, $C_k = C_1 + (k-1)\lambda_+$).  
			When we also impose Assumption \ref{DriftAssumpS},  we have that the process \eqref{slow}-\eqref{fast} has sufficient number of moments uniformly bounded in  $t$ (and $\epsilon$), see Lemma \ref{prop:fullMombound}. 
			\item Assumption \ref{DriftAssump} and Assumption \ref{UniformEllip} combined,  give existence and uniqueness of the invariant measure $\mu^x$ of the semigroup $P_t^x$ (for each $x$ fixed) and exponential convergence to $\mu^x$ as well. Such ergodic properties of $P_t^x$ are known, when $x$ is fixed. What we will additionally need is to control the way in which the rate of exponential convergence depends on $x$, see comments after Proposition \ref{lemma:IntAvg} and Section \ref{sec:sec3}.

			\item Assumption \ref{ass:polGrowth}, Assumption \ref{DriftAssump} and Assumption \ref{DriftAssumpS} are sufficient to have pathwise well-posedness of the SDEs \eqref{slow}-\eqref{fast}, \eqref{eqn:averaged} and \eqref{fastF}, by \cite[Theorem 3.1.1]{LiuRockner}.\footnote{\cite[Theorem 3.1.1]{LiuRockner} has two main assumptions, a local weak monotonicity condition,  \cite[Equation (3.3)]{LiuRockner}, and weak coercivity condition,  \cite[Equation (3.4)]{LiuRockner}. The former follows  if all the coefficients are locally Lipschitz which in turn follows if they  are continuously differentiable, which is the case by   Assumption \ref{ass:polGrowth} for  $b,g,\sigma,a$. 
				We show in Proposition \ref{lemma:IntAvgDerRepresentation} that $\bar{b}$ is continuously differentiable and therefore local weak monotonicity holds for \eqref{eqn:averaged} as well. 
				The weak coercivity condition holds for the SDE \eqref{fastF} by Assumption \ref{DriftAssump}, and for the SDE \eqref{slow}-\eqref{fast} by  Assumption \ref{DriftAssumpS}.  Integrating Assumption \ref{DriftAssumpS} with respect to $\mu^x$, such a condition also holds for the SDE \eqref{eqn:averaged}. }
			
			\item Assumption \ref{ass:SGcond} is used to ensure SES (i.e. decay of the derivative) of the `frozen semigroup' $\SGfastsxn{t}{x}$, that is to show that $\eqref{derestgeneral}_2$ (or, more precisely,  \eqref{eqn:SGdecayConc}) holds. Therefore we can replace Assumption \ref{ass:SGcond} by assuming that the frozen semigroup is SES, see Section \ref{sec:mainresults} for more details.  
			Similarly, Assumption \ref{ass:avgSGcond} is used to ensure the decay of the derivative of the averaged semigroup $\bar{\cP}_{t}$ and may be replaced by assuming that $\bar{\cP}_t$ is SES. We emphasize that  Assumption \ref{ass:avgSGcond} is an assumption on the drift $b$ which will be used to derive properties of the averaged coefficient $\bar b$. 
			
			\item If the second, third and fourth order derivatives of $g$ are bounded then Assumption \ref{ass:SGcond} reduces to finding $\overline{\cb}>0$ such that
			\begin{equation}\label{eqn:driftCondSimpler}
				2 (\xi, \nabla_y g(x,y)\xi) \leq -\overline{\cb} \lvert \xi\rvert^2.
			\end{equation}
			Indeed \eqref{eqn:driftCondSimpler} implies that Assumption \ref{ass:SGcond} holds for some $\cb>0$ (by taking $\zeta_1,\zeta_2$ and $\zeta_3$ sufficiently small).

			\item Let us give two examples in which Assumption \ref{ass:avgSGcond} takes a simpler form. Firstly, if $b$ is of the form $b(x,y) = b_1(x) + b_2(y)$, then $\ytwodery$, as defined before Assumption \ref{ass:avgSGcond}, vanishes, implying that $D_b(x)$ is constant. If, further, $b_2$ has bounded (in $y$) first derivatives and hessian, then $D_b$ simplifies to		\begin{equation*}
				D_b(x) = 3\left(K_{g}+K_{A}\right)
			\end{equation*}
		\end{enumerate}
	\end{Note}
	An example satisfying all of the above assumptions is given by the following system:
	\begin{align}\label{eqn:slowEx}
		dX_{t}^{\epsilon,x,y} & =  (-\left(X_t^{\epsilon,x,y}\right)^3 - X_t^{\epsilon,x,y} +b_0( Y_{t}^{\epsilon,x,y}) )dt + \sqrt{2} \, dW_t \\ \label{eqn:fastEx}
		dY_{t}^{\epsilon,x,y} & = \frac{1}{\epsilon}(-\left(Y_t^{\epsilon,x,y}\right)^3 - Y_t^{\epsilon,x,y}+g_0(X_{t}^{\epsilon,x,y})) dt + \frac{1}{\sqrt{\epsilon}} \sqrt{2}\, dB_t
	\end{align}
	where $(X_{t}^{\epsilon,x,y}, Y_t^{\epsilon,x,y}) \in \R \times \R$ and $b_0,g_0\in C_b^\infty(\R)$. Both the slow and fast component are Langevin-type dynamics, we will make more remarks on the above system in Example  \ref{ex:nonLip}. 
	 Further examples are given in the next section and in Section \ref{sec:numerics}.
	\begin{Note}\label{rem:xavgderivs}
		One difficulty in obtaining derivative estimates for the semigroup $\bar{\cP}_tf$ given by \eqref{eqn:avgedSGdef}, is that the coefficients of the SDE \eqref{eqn:averaged} depend on the invariant measure of  equation \eqref{fastF}, which is itself a function of $x$. Indeed, recall that the coefficient $\overline{b}$ is defined by \eqref{avgDef} and depends on $x$ both through the function $b$ itself and through the measure $\mu^x$ so in general has a complicated dependence on $x$. With an explicit expression for $ \mu^x$ (which might or might not be available, depending on the specific application at hand, see e.g. \cite{hairerPav,pavliotis2022derivative}), and hence for $\bar b$, one could verify an assumption of the type \eqref{eq:driftcondition} (or, in the case of non-constant drift, use Theorem \ref{thm:lorenziderest} directly) to obtain the desired DEs. In the absence of such an expression, we use a different approach, see Section \ref{sec:SGder}. 
	\end{Note}
	

	
	\section{Main Results} \label{sec:mainresults}
	In this section we gather our main results. In particular, in Subsection \ref{subsec:nonfullycoupledresults} we gather our results regarding the non fully coupled regime, in Subsection \ref{subsec:fullycoupledresults} the results in the fully coupled case.  
	\subsection{The non fully coupled regime} \label{subsec:nonfullycoupledresults}
	We start by stating results on Poisson equations and then our UiT averaging result for SDEs.

	{\bf Poisson equations.} We consider equations of the form \eqref{poisprobgeneral}
	where $u:\R^n\times \R^d \rightarrow \R$ is the unknown,  while $\phi$ is a given function, assumed to be in the space  $\funcSpace{m}{m'}$ for some $m,m'\geq 0$ and the operator $\mathcal{L}^x$ is a second order differential operator of the form \eqref{eqn:gen} (not of the form \eqref{heart}, which is of interest only when studying the fully coupled regime). Let us clarify that, throughout the paper, with the exception of  Section \ref{subsec:fullycoupledresults} and Section \ref{subsec:fullycoupledsketch} where the fully coupled regime is treated or unless otherwise stated, when we refer to the Poisson equation \eqref{poisprobgeneral}, we mean \eqref{poisprobgeneral} with $\mathcal L^x$ as in \eqref{eqn:gen}. 
	 
  Because $x$ in the above appears only as a parameter,   it is sometimes useful to refer to a function $\phi(x,y)$ as $\phi^x(y)$ and from now on we will use these notations interchangeably;  moreover,  when we need to emphasise the dependence of the solution $u$ on $\phi$, we denote the solution to the Poisson equation \eqref{poisprobgeneral} as $u(x,y)=u^x_\phi(y)$. As we have already said, such a  solution is in general not unique, since constants belong to the kernel of $\cL^x$. However, we restrict to the set of mean-zero functions, i.e. to the set of functions $f$ such that $\mu^x(f)=0$, to ensure the solution,  if it exists, is unique . When we refer to  the solution to \eqref{poisprobgeneral}, we understand this to mean the mean-zero solution. 
	Moreover, in order for the Poisson problem \eqref{poisprobgeneral} to have a solution, it is necessary for  the RHS of \eqref{poisprobgeneral} to be a  mean-zero function,  as can be seen by integrating the LHS of  \eqref{poisprobgeneral} with respect to $\mu^x$;  this is satisfied in \eqref{poisprobgeneral} by the definition of $\bar \phi$.
	
	\begin{theorem}\label{mainthmpois}
		Let $\drift$ and $\diffusion=aa^T$ satisfy Assumption \ref{ass:polGrowth} \ref{ass:polGrowthDriftFast},\ref{ass:polGrowthDiffFast}, Assumption \ref{UniformEllip}, Assumption \ref{DriftAssump} and Assumption \ref{ass:SGcond}.  Let $\phi \in \funcSpace{\pow{}{x}}{\pow{}{y}}$, for some $\pow{}{x},\pow{}{y}\geq 0$. Then the function $\bar{\phi}$ is well defined, the solution $u_\phi$ of the Poisson equation \eqref{poisprobgeneral} exists, it is unique (in the class of mean-zero functions) and it is given by
		\begin{equation} \label{ft1rep}
		u^x_\phi(y) \coloneqq -\intPos \left(\SGfastsx{s}{x}\phi^x(y) - \bar{\phi}(x)\right) ds\,.
	\end{equation}
		 Moreover, 
		 
		$\bullet$ There exists some $C>0$ (which may depend on $\pow{}{x},\pow{}{y}$ but is independent of the choice of $\phi\in \funcSpace{\pow{}{x}}{\pow{}{y}}$) such that the solution $u_\phi^x$ satisfies the following bound
		\begin{equation}\label{eqn:IntAvg2}
			\left| u_\phi^x(y) \right| \leq C\Vnorm{\phi}{0,\pow{}{x},\pow{}{y}}(1+|x|^{\pow{}{x}}+|y|^{\pow{}{y}}) \quad \text{ for all } x\in \R^n, y\in \R^d.
		\end{equation}
		
		$\bullet$  
		 Both the solution $u_\phi$ and the function $\bar{\phi}$ are twice differentiable in $x$ and there exists $C>0$ such that 
		\begin{align}
			&\left| \slowDer{i}\bar{\phi} \right| \leq C\Vnorm{\phi}{2,\pow{}{x},\pow{}{y}}(1+|x|^{2\pow{}{x}}),\label{eq:intAvgDerInv} \\
			&\left| \slowDerDer{i}{j}\bar{\phi} \right| \leq C\Vnorm{\phi}{4,\pow{}{x},\pow{}{y}}(1+|x|^{4\pow{}{x}}),\label{intAvgDerDerInv}\\ 
			&\left| (\slowDer{i} u_{\phi}^x)(y) \right| \leq C\Vnorm{\phi}{2,\pow{}{x},\pow{}{y}}(1+ |y|^{M^{g}_{y}}+|x|^{2\pow{}{x}}),\label{eq:intAvgDer} \\
			&\left| \slowDerDer{i}{j} u_{\phi}^x(y) \right| \leq C\Vnorm{\phi}{4,\pow{}{x},\pow{}{y}}(1+ |y|^{M^{g}_{y} + \pow{\drift}{y}}+|x|^{4\pow{}{x}})\label{eqn:intAvgDerDer},
		\end{align}
		for all $x\in\R^n,y\in \R^d$,  $i,j\in \{1,\ldots,n\}$, where $M^{g}_{y}\coloneqq \max\{2\pow{\drift}{y},\pow{\drift}{y}+\pow{}{y}\}$.
	\end{theorem}
	\begin{proof}  Note \ref{note:avg} explains that $\phi$ is integrable with respect to $\mu^x$ and hence the function $\bar{\phi}$ is well defined.
		The proof of this theorem can be found in Section \ref{sec:psneqn}. In particular, well-posedness of  the Poisson problem \eqref{poisprobgeneral} is given in Lemma \ref{lemma:probRep}. The estimate \eqref{eqn:IntAvg2} follows from \eqref{ft1rep}, Lemma \ref{lemma:probRep} and  Proposition \ref{lemma:IntAvg}. For the proof of \eqref{eq:intAvgDerInv} and \eqref{eq:intAvgDer} see Proposition \ref{lemma:IntAvgDerRepresentation}, and for the proof of \eqref{intAvgDerDerInv} and \eqref{eqn:intAvgDerDer} see Proposition \ref{lemma:IntAvgDerDer}.
	\end{proof}
	
	The above result is of independent interest, but also instrumental to solving the averaging problem, see Note \ref{Rem:poisson for averaging}. Explanations on how SES  of the  semigroup $P_t^x$ helps to tackle the issue posed by the unboundedness of the drift coefficient of the operator $\cL^x$ in the study of the smoothness in $x$ of the solution Poisson problem \eqref{poisprobgeneral} can be found in Note \ref{Rem:der est for Pois unbounded}.

	{\bf Averaging for SDEs. } Our main result on averaging for SDEs in the non fully coupled regime \eqref{slow}-\eqref{fast} is the following. 
	
	\begin{theorem}\label{mainthm} Consider the slow-fast system \eqref{slow}-\eqref{fast} and the semigroups $\{ \cP^\epsilon_t\}_{t\geq0}$ and $\{ \cPbar_t\}_{t\geq0}$ defined in  (\ref{eqn:fullSGdef}) and (\ref{eqn:avgedSGdef}), respectively. Let Assumption \ref{ass:polGrowth} to Assumption \ref{ass:avgSGcond} hold. Then,  for every $f\in C^2_b(\mathbb{R}^n)$, there exists a constant $C>0$, {\em independent of time},  such that
		\begin{equation}\label{eqn:conclnonfullycoupled}
			\left| (\cP_t^\epsilon f)(x,y) - (\cPbar_t f)(x) \right| \leq \epsilon C\|f\|_{C_b^2} (1+ |y|^{M^{g, b}_{y}+\pow{g}{y}}+|x|^{4\pow{b}{x}}), \quad \forall x \in \R^n, \,y\in \R^d, 
		\end{equation}
		where $M^{g, b}_{y}\coloneqq \max\{2\pow{g}{y},\pow{g}{y}+m^{b}_{y}\}$.
	\end{theorem}
	
	We prove Theorem \ref{mainthm} in Section \ref{subsec:detailedproof}. 
	In Example \ref{ex:nonLip} we show that system \eqref{eqn:slowEx}-\eqref{eqn:fastEx} satisfies all the assumptions of Theorem \ref{mainthm}.

	Theorem \ref{mainthm} contains sufficient conditions, phrased in terms of the coefficients of system \eqref{slow}-\eqref{fast}, in order for the UiT averaging result to hold. This makes it ready to use - as one needs only check conditions on the coefficients - but it does not help to highlight the role of DEs and SES. We therefore rephrase it below in terms of SES properties of the semigroups $\cPbar_t$ and $\SGfastsxn{t}{x}$ (we will do a similar thing also for Theorem \ref{mainthmpois} in the fully coupled case, see Theorem \ref{mainthmpoisFullyCoupled}).

	\begin{theorem}\label{thm:nonfullycoupledderestversion}
		Consider the semigroup $\{\cP_t^\epsilon\}$ associated to \eqref{slow}-\eqref{fast} and the semigroup $\{\bar{\cP}_t\}_{t\geq 0}$ associated to \eqref{eqn:fullSGdef}. Let Assumptions \ref{ass:polGrowth} to  \ref{DriftAssumpS} hold. Suppose, furthermore, that  there exists a constant $K>0$ and  $\pow{}{x},\pow{}{y}\geq 0$ such that for any $\psi\in C^{0,k}(\R^n\times\R^d)$ with $\Vseminorm{\psi}{k,\pow{}{x},\pow{}{y}}< \infty$ and for all $x \in \R^n$
		we have
		\begin{equation}\label{eqn:4der}
			\lvert\SGfastsxn{t}{x}\psi^x \rvert_{k,m_x, m_y}\leq K|\psi|_{k,m_{x}^{}, m_{y}^{}}e^{-\cb t}, \quad k\in \{2,4\}\, ,
		\end{equation}
		and that there exist constants $\tilde{K},C>0$ such that for any $\psi\in C_b^2(\R^d)$ we have
		\begin{equation}\label{eqn:2der}
			\sup_{1 \leq \lvert\gamma\rvert_* \leq 2}\lVert \partial^{\gamma}_{x} \bar{\cP}_{t}\psi \rVert_\infty  \leq \tilde{K}e^{-C t}.
		\end{equation}
		Then \eqref{eqn:conclnonfullycoupled} holds. 
	\end{theorem}
	
	Theorem \ref{mainthm} can be seen as a consequence of Theorem \ref{thm:nonfullycoupledderestversion}. Indeed,  the  structure of the proof of  Theorem \ref{mainthm} clearly shows that \eqref{eqn:conclnonfullycoupled} is implied by Assumptions \ref{ass:polGrowth} to Assumption \ref{DriftAssumpS} plus \eqref{eqn:4der} and \eqref{eqn:2der}. The proof of Theorem \ref{mainthm} then goes further and shows that 
	Assumptions \ref{ass:polGrowth} to Assumption \ref{ass:SGcond} imply \eqref{eqn:4der} (see Proposition \ref{prop:derivativeest}) and that Assumptions \ref{ass:polGrowth} to \ref{ass:avgSGcond} imply \eqref{eqn:2der}(see Proposition \ref{prop:avgderivativeest}). So, the line of reasoning in the proof of Theorem \ref{mainthm} implies Theorem \ref{thm:nonfullycoupledderestversion}, and we don't prove the latter separately. The DEs \eqref{eqn:4der} and \eqref{eqn:2der} are central to the proof for the following reason: using relatively standard tricks from semigroup theory, one can express the difference between $\cP_t^\epsilon$  and $\cPbar_t$ in terms of two main objects, namely the second derivative of the semigroup $\cPbar_t$ and the second $x$ derivative of the solution of a Poisson equation of the form \eqref{poisprobgeneral} - see \eqref{rDiff}, Note \ref{Rem:poisson for averaging} and the calculations in Proposition \ref{psnEqDer}. Understanding both these objects requires the analysis of the second $x$ derivatives of  $P_t^x$. Indeed, because the coefficients of the generator of the averaged semigroup $\cPbar_t$ contain the measure $\mu^x$, one needs to study the second derivative of $\mu^x$ with respect to $x$ - see Proposition \ref{lemma:IntAvgDerDer}. In turn,  $\mu^x$ is the limit as $t \rightarrow \infty$ of the semigroup $P_t^x$, hence the appearance of  the second $x$ derivative of  $P_t^x$. Moreover, because the solution of Poisson equations of the form \eqref{poisprobgeneral}  can be expressed through the semigroup $P_t^x$, see \eqref{ft1rep},  it is clear that also the study of the second $x$ derivatives of the  Poisson equation leads one to consider second $x$ derivatives of  $P_t^x$.  As discussed in the introduction, the smoothness in $x$ of the semigroup $P_t^x$ is non-trivial and is gained through some ``transfer formulas" (see for example   \eqref{eqn:repSG} and Note \ref{Rem:der est for Pois unbounded}) which allow one to obtain smoothness in $x$ from the smoothness in $y$ (which is instead straightforward). When using such transfer formulas one $x$ derivative comes at the price of two $y$ derivatives, and this is the reason why we will need four $y$ derivatives of $P_t^x$.  We note that, as opposed to previous literature,  we write these formulas in a way to highlight how they contain appropriate semigroup derivatives and hence the role of such derivatives in our analysis. 
	
	Finally let us give an intuitive explanation on how SES of the semigroup $\cPbar_t$, i.e. \eqref{eqn:2der},  is instrumental to obtain an averaging result which is uniform in time  in Note \ref{Rem: der est for UiT} below.

	\begin{Note}[SES for UiT averaging]\label{Rem: der est for UiT} \textup{ To explain in a simplified setting why SES is key to proving uniform in time convergence, let us consider two Markov semigroups, say $\mathcal T_t$ and $\widebar{\mathcal {T}}_t$. With standard manipulations, the difference between such semigroups can be expressed in terms of the difference between their respective generators, say $\mathcal G$ and 
			$\bar{\mathcal G}$, as follows
			\begin{align*}
				(\widebar{\mathcal {T}}_t \varphi)(z) -  (\mathcal{T}_t \varphi)(z) & =  \int_0^t ds \frac{d}{ds} \mathcal{T}_{t-s}\bar{\mathcal{T}}_s \varphi (z)  =  
				\int_0^t \!\!\!ds \,  \mathcal{T}_{t-s} (\bar{\mathcal G} - \mathcal G) \bar{\mathcal T}_s \varphi (z) \\
				&\leq \int_0^t ds \| \mathcal T_{t-s} (\bar{\mathcal G} - \mathcal G) \bar{\mathcal T}_s \varphi \|_{\infty}
				\leq \int_0^t ds \|  (\bar{\mathcal G} - \mathcal G) \bar{\mathcal T}_s \varphi \|_{\infty}
			\end{align*}
		If $\mathcal G$ and $\bar{\mathcal{G}}$ are differential operators then the latter difference involves derivatives of the semigroup $\bar{\mathcal T}_t$. If such derivatives decay exponentially fast in time, then the difference between such semigroups can be estimated by a constant (independent of time) rather than with exponential growth, which is what would happen by using Gronwall-type arguments. This line of reasoning, when applied to $\bar \cP_t$ and $\cP_t^{\epsilon}$ rather than $\bar{\mathcal T}_t$ and $\mathcal T_t$, inspires our approach - though the precise proof does not exactly follow the above calculation and some further manipulations are required (to obtain the correct power of $\epsilon$ on the RHS). Details of our strategy of proof can be found in Section \ref{subsec:detailedproof}.
		}
	\end{Note}

	\begin{example}\label{ex:nonLip} Let us come back in more detail to the system \eqref{eqn:slowEx}-\eqref{eqn:fastEx}. First of all,  both the slow and the fast components of  \eqref{eqn:slowEx}-\eqref{eqn:fastEx} are  Langevin-type dynamics for the potential $V(x)=x^4/4+x^2/2$, with bounded perturbation $b_0$ (or $g_0$). This setup is rather important in applications, see e.g. \cite{pavliotis2008multiscale}. However we point out that because of our assumptions on the drifts $b$ and $g$ (Assumptions \ref{ass:SGcond} and \ref{ass:avgSGcond}), the case of Langevin dynamics in double well potentials is not covered in this paper. In the numerics section, Section \ref{sec:numerics}, we provide evidence supporting the idea that a UiT result should be  true even in that case, and we reserve this for upcoming work. \\
		 Let us now come to show that Assumptions \ref{ass:polGrowth}-\ref{ass:avgSGcond} are verified  for system \eqref{eqn:slowEx}-\eqref{eqn:fastEx} and hence prove that the UiT result of Theorem \ref{mainthm} holds for such a system. It is immediate to see that Assumption \ref{ass:polGrowth} holds with $m^b_x=3, m^b_y=0, \pow{g}{y}=3$ and that Assumption \ref{UniformEllip} holds with $\lambda_-=\lambda_+=1$. As observed in Note \ref{rem:commentonassumptions}, in order to show that Assumption \ref{DriftAssump} holds it is sufficient to consider $k=1$. For $k=1$, \eqref{eq:driftassump} holds with $C_1=\lVert g_0\rVert_\infty^2/2, r_1=1/2$. Similarly, Assumption \ref{DriftAssumpS} holds with $\tilde{C}_1=\lVert b_0\rVert_\infty^2/2, \tilde{r}_1=1/2$. Assumption \ref{ass:SGcond} holds with $\zeta_1 = \frac{1}{6}$ and any $\kappa < 2$, with a corresponding value $\zeta_2 \leq (2-\kappa)/36$. Since we can take $\kappa$ arbitrarily close to $2$ and the value of $\zeta_2$ does not make a difference to the calculations of the other constants, in the below we take $\kappa = 2$,   therefore obtaining a strict inequality in \eqref{eqn:examplecond}. It remains to verify Assumption \ref{ass:avgSGcond}. Observe that $\yfourder=\yfourdery=
		\ytwoderyonex= \ytwoderyoney= \ytwoderxy= \ytwodery = \ytwoderyy = \onederxg = 0$, and $\ytwoderx=1$. We also have $D_0 = (\sqrt{6}+1)^{1/2}$, $D_b= \lVert \partial_x g_0 \rVert_{\infty}$ and $K_A = K_\Sigma = 0$. Therefore, Assumption \ref{ass:avgSGcond} holds provided \begin{equation}\label{eqn:examplecond}
			\max\{\lVert \partial_y b_0\lVert_{\infty},\lVert \partial_{yy} b_0\lVert_{\infty}\}<\frac{1}{2\sqrt{3}(\sqrt{6}+1)^{1/2}\lVert \partial_x g_0 \rVert_{\infty}}.\end{equation}
		
		In conclusion, if \eqref{eqn:examplecond} holds, then we may apply Theorem \ref{mainthm} to \eqref{eqn:slowEx}-\eqref{eqn:fastEx} so there exists $C>0$, such that
		\begin{equation*}
			\left| \cP_t^\epsilon f(x,y) - \cPbar_t f(x) \right| \leq \epsilon C\|f\|_{C_b^2}(1+ |y|^{9}+|x|^{12}), \text{ for every $f\in C^2_b(\mathbb{R}^n)$.}
		\end{equation*}
		
	\end{example}


	\subsection{The fully coupled regime}\label{subsec:fullycoupledresults}
	
	In this subsection and in Subsection \ref{subsec:fullycoupledsketch} only we consider the fully coupled system \eqref{slowfullycoupled}-\eqref{fastfully coupled} as opposed to \eqref{slow}-\eqref{fast}. The only difference between these two systems is that in the former $a=a(x,y), \sigma=\sigma(x,y)$, i.e. the diffusion coefficients $a, \sigma$ are allowed to depend on both variables, in the latter $a=a(x), \sigma=\sigma(x)$. 
	Therefore, here and in Subsection \ref{subsec:fullycoupledsketch} by $P_t^x, \bar\cP_t$ and $\cP_t^{\epsilon}$ we mean the semigroups associated with the processes \eqref{introfastprocessepsilon=1}, \eqref{eqn:averagedfully coupled} and  \eqref{slowfullycoupled}-\eqref{fastfully coupled}, respectively. In particular,  the generator $\cL^x$ of $\cP_t^x$  is intended to be given by \eqref{heart}, as opposed to \eqref{eqn:gen} and when we say that Assumptions \ref{ass:polGrowth} to Assumption \ref{DriftAssumpS} hold, we mean that they hold for $a=a(x,y), \sigma=\sigma(x,y)$ and Assumption \ref{DriftAssumpS}, \ref{ass:polGrowthDiffFast} should be modified to:
			there exists $0<\nu<1$ such that for each $x\in \R^n$, $A_{ij}(x, \cdot) \in C^{4+\nu}(\R^d)$, and $A_{ij} \in \funcSpace{0}{0}$ for all $1 \leq i,j \leq d$.
	
	
	
	\begin{theorem}\label{mainthmpoisFullyCoupled}
		Let $\drift$ and $\diffusion=aa^T$ satisfy Assumption \ref{ass:polGrowth}, Assumption \ref{UniformEllip}, Assumption \ref{DriftAssump} and suppose \eqref{eqn:4der} holds.  Let $\phi \in \funcSpace{\pow{}{x}}{\pow{}{y}}$ for some $\pow{}{x},\pow{}{y}\geq 0$. Then the solution $u_\phi$ of the Poisson equation \eqref{poisprobgeneral} (with $\mathcal L^x$ given by \eqref{heart}) exists and is unique (in the class of mean-zero functions) and the representation formula \eqref{ft1rep} holds as well. Moreover, 
		\begin{itemize}
			\item There exists some $C>0$ (which may depend on $\pow{}{x},\pow{}{y}$ but is independent of the choice of $\phi$) such that the solution $u_\phi^x$ satisfies the following bound
			\begin{equation}\label{eqn:IntAvg2fullycoupled}
				\left| u_\phi^x(y) \right| \leq C\Vnorm{\phi}{0,\pow{}{x},\pow{}{y}}(1+|x|^{\pow{}{x}}+|y|^{\pow{}{y}}) \quad \text{ for all } x\in \R^n, y\in \R^d.
			\end{equation}
			\item The function $\bar{\phi}$ is well defined, both the solution $u_\phi$ and the function $\bar{\phi}$ are twice differentiable in $x$ and there exist $C,m_x',m_y'>0$ such that 
			\begin{align}\label{eq:intAvgDerInvcoupled}
				&\sup_{1\leq \lvert\gamma\rvert_*\leq 2}\left| \partial_x^\gamma\bar{\phi} \right| \leq C\Vnorm{\phi}{4,\pow{}{x},\pow{}{y}}(1+|x|^{m_x'}), \\\label{eq:intAvgDercoupled}
				&\sup_{1\leq \lvert\gamma\rvert_*\leq 2}\left| \partial_x^\gamma u_{\phi} \right| \leq C\Vnorm{\phi}{4,\pow{}{x},\pow{}{y}}(1+|x|^{m_x'}+\lvert y\rvert^{m_y'}),
			\end{align}
			for all $x\in\R^n,y\in \R^d$.
		\end{itemize}
	\end{theorem}
  \begin{proof}
     The proof of Theorem \ref{mainthmpoisFullyCoupled} is very similar to that of Theorem \ref{mainthmpois}, so we omit the details here. Well-posedness of the Poisson problem is, again, given by Lemma \ref{lemma:probRep}, the proof of which does not rely on the diffusion coefficient $a$ being independent of $y$. The proof of \eqref{eqn:IntAvg2fullycoupled} and \eqref{eq:intAvgDerInvcoupled} are Proposition \ref{lemma:IntAvgDerRepresentation} in the case of the first derivative and Proposition \ref{lemma:IntAvgDerDer} in the case of the second. Note that, while Proposition \ref{lemma:IntAvgDerRepresentation}
and Proposition \ref{lemma:IntAvgDerDer} require $a(x,y)$ to be independent of $y$, this is only used for the calculation of the constants in the estimates 
\eqref{eq:intAvgDerInv}-\eqref{eqn:intAvgDerDer}. More specifically, the LHS of \eqref{eqn:elemBound3} and \eqref{eqn:elemBound4} in the proof of Proposition \ref{lemma:IntAvgDerRepresentation} (similarly the LHS of \eqref{eqn:term2} and \eqref{eqn:term3} in the proof of Proposition \ref{lemma:IntAvgDerDer}) contain extra terms, corresponding to the $y$ derivative of the diffusion coefficient $A$. Since $A \in \funcSpace{0}{0}$, these derivatives exist and are bounded. In addition, \eqref{eqn:4der} means that the semigroup derivative estimates can be used in the same way as in the proofs Proposition \ref{lemma:IntAvgDerRepresentation} and Proposition \ref{lemma:IntAvgDerDer}, which instead require Assumption \ref{ass:SGcond}. Hence, the RHS of \eqref{eqn:elemBound3} and \eqref{eqn:elemBound4} are unchanged (they just hold with $C$ being a different constant). This gives the final result. 
\end{proof}
	Before stating the averaging result we require SES of the averaged SDE \eqref{eqn:averagedfully coupled}.

	\begin{assumption}\label{ass:averagederivativeestimate4}
		Let $\{\bar{\cP}_{t} \}_{t\geq 0}$ denote the semigroup associated to \eqref{eqn:averagedfully coupled}, and assume that Assumptions \ref{ass:polGrowth}-\ref{DriftAssumpS} hold.
		Assume there exist constants $\tilde{K},C>0$ such that for any $\psi\in C_b^4(\R^n)$ we have
		\begin{equation*}
			\sup_{1 \leq \lvert\gamma\rvert_* \leq 4}\lVert \partial^{\gamma}_{x} \bar{\cP}_{t}\psi \rVert_\infty  \leq \tilde{K}e^{-C t}\lVert \psi\rVert_{C_b^4(\R^n)}.
		\end{equation*}
	\end{assumption}
	
	\begin{theorem}\label{fullycoupledaveragingtheorem}
		Consider the semigroup $\{\cP_t^\epsilon\}$ associated to \eqref{slowfullycoupled}-\eqref{fastfully coupled} and the semigroup $\{\bar{\cP}_t\}_{t\geq 0}$ associated to \eqref{eqn:averagedfully coupled}. Let Assumption \ref{ass:polGrowth} to Assumption \ref{DriftAssumpS} hold, together with  \eqref{eqn:4der} and Assumption \ref{ass:averagederivativeestimate4}. Then for every $f\in C_b^4(\R^n)$ there exist constants $C,M_x,M_y>0$ independent of $f,t,x,y$ such that
		\begin{equation*}
			\left| (\cP_t^\epsilon f)(x,y) - (\cPbar_t f)(x) \right| \leq \epsilon C\|f\|_{C_b^4} (1+ |y|^{M_{y}}+|x|^{M_x}), \quad \forall x \in \R^n, \,y\in \R^d. 
		\end{equation*}
	\end{theorem}
A sketch of the  proof of Theorem \ref{fullycoupledaveragingtheorem} is contained in Subsection \ref{subsec:fullycoupledsketch}. 
	The main difference between the assumptions of the above theorem and Theorem \ref{thm:nonfullycoupledderestversion} is that in the latter we require $x$ derivative estimates of $\bar{\cP}_t$ up to order $2$ whereas for Theorem \ref{fullycoupledaveragingtheorem} we require $x$ derivatives of $\bar{\cP}_t$ up to order $4$. As explained after Theorem \ref{thm:nonfullycoupledderestversion} the two main objects in the proof of Theorem \ref{thm:nonfullycoupledderestversion} are the second $x$ derivative of the semigroup $\bar{\cP}_t$ and the second $x$ derivative of the solution of a Poisson equation. In contrast, in the setting of Theorem \ref{fullycoupledaveragingtheorem} we need fourth order $x$ derivatives of the semigroup $\bar{\cP}_t$ but still only two $x$ derivatives of the solution of a Poisson equation. In general to obtain derivative estimates of $\bar{\cP}_t$ up to order $4$ of the form in Assumption \ref{ass:averagederivativeestimate4} requires the coefficients of the generator $\bar{\cL}$ to be $x$-differentiable up to order $4$. Since these coefficients involve the measure $\mu^x$, one needs $\mu^x$ to be $4$ times $x$-differentiable. As $\mu^x$ is the limit of $P_t^x$, one can show the $x$-regularity of $\mu^x$ by first showing the $x$-regularity of $P_t^x$. The $x$ derivatives of $P_t^x$ can be bounded using the $y$ derivatives of $P_t^x$ by using the ``transfer formulae'' (see for example \eqref{eqn:repSG}) however this would then require \eqref{eqn:SGdecayConc} to hold instead for derivatives up to order $8$. We avoid needing $y$ derivatives up to order $8$ by assuming that Assumption \ref{ass:averagederivativeestimate4} hold. (However in order to obtain bounds on the second $x$ derivative of the solution to a Poisson equation we still require bounds on $y$ derivatives of $P_t^x$ up to order 4.) 
	With specific examples at hand, checking that Assumption  \ref{ass:averagederivativeestimate4} holds is not difficult, but it can be rather lengthy. An example of a system that satisfies all the assumptions of Theorem \ref{fullycoupledaveragingtheorem} is the following: 
	
\begin{example}
    Let $d=n=1$ and consider the following system:
\begin{align*}
        dX_t^{\epsilon,x,y} &= (-X_t^{\epsilon,x,y}-\frac{1}{2}e^{-(Y_t^{\epsilon,x,y})^2})dt +\sqrt{2}\sqrt{1+\frac{1}{4}\sin(Y_t^{\epsilon,x,y})-\frac{X_t^{\epsilon,x,y}}{4\sqrt{e}\sqrt{1+(X_t^{\epsilon,x,y})^2}}}dW_t\\
        dY_t^{\epsilon,x,y} &= \frac{1}{\epsilon}\left[\arctan(X_t^{\epsilon,x,y})-Y_t^{\epsilon,x,y}\right]dt+\sqrt{\frac{2}{\epsilon}}\left(1+\frac{\sin(b_0(X_t^{\epsilon,x,y}) +Y_t^{\epsilon,x,y})}{12}\right) dB_t\\
        &+\frac{1}{\epsilon}\left[\frac{\cos(b_0(X_t^{\epsilon,x,y})+Y_t^{\epsilon,x,y})}{6} \left(1+\frac{\sin(b_0(X_t^{\epsilon,x,y})+Y_t^{\epsilon,x,y})}{12} \right)\right]dt .
    \end{align*}
    Here $b_0:\R\to\R$ are smooth bounded functions with bounded derivatives of all orders. {Observe that square root term in $X_t^{\epsilon,x,y}$ is always real valued since the expression in the square root is bounded below by $1-\frac{1}{4}-\frac{1}{4\sqrt{e}}$ which is positive, this follows since the functions $\sin$ and $x\mapsto \frac{x}{\sqrt{1+x^2}}$ are bounded in absolute value by $1$.}
    With calculations analogous to those that we will do to obtain the derivative estimates of the next section (which we don't repeat here for this specific example)  one can show that the corresponding frozen semigroup is SES. Moreover we have that $\mu^x$ is normally distributed with mean $\arctan(x)$ and variance $1$ and hence the averaged equation is given by
   \begin{equation*}
        d\bar{X}_t^{x} = (-\bar{X}_t^{x}-\frac{1}{2}\exp(-\frac{1}{3}\arctan(\bar{X}_t^{x})^2))dt +\sqrt{2}dW_t.
    \end{equation*}
    Since the drift is strictly monotone it is immediate to check that it satisfies conditions of the form \eqref{eqn:driftCondSimpler} and hence the averaged semigroup is SES.
\end{example}



	
	\section{Preliminary Results on the frozen semigroup $P_t^x$} \label{sec:sec3}
	Looking at the representation formula \eqref{ft1rep}, it is clear that the ergodic properties of the semigroup $P_t^x$ are central to the study of the Poisson problem \eqref{poisprobgeneral} -- trivially, if $P_t^x$ does not decay fast enough to $\bar{\phi}^x$ then \eqref{ft1rep} makes no sense.   With this motivation,  Lemma \ref{prop:momentBounds} and Proposition \ref{lemma:IntAvg} address the ergodic properties of the semigroup ${P}^x_t$. In Proposition \ref{prop:derivativeest} we show SES for the same semigroup. 
	Throughout the section we assume  that the operator $\Lfastx{x}$ is uniformly elliptic, i.e $ \diffusion$ satisfies Assumption \ref{UniformEllip} and we don't repeat this fact in every statement.
	

	\begin{lemma}\label{prop:momentBounds}
		
		
		Let $\drift$ and $\diffusion$ satisfy Assumption \ref{ass:polGrowth} (\ref{ass:polGrowthDriftFast},\ref{ass:polGrowthDiffFast}), 
		and Assumption \ref{DriftAssump}. Let $\fastFixedI_{t}$ be the solution of the SDE (\ref{fastF}) (which does exist under our assumptions, see Note \ref{rem:commentonassumptions}). Then for any integer $k \geq 1$, 
		$x\in \mathbb{R}^n, y\in \mathbb{R}^d$ we have
		\begin{align}
			\label{eqn:momBoundsY1}
			&\fastExp{y} \left[|\fastFixedI_t|^k\right] \leq e^{-r_k' t}|y|^k + \frac{ C_k'}{r_k'} \\\label{eqn:momBoundsYinv}
			&\int_{\mathbb{R}^d}|{y}|^k \mu^x(d{y}) \leq \frac{C_k'}{r_k'},
		\end{align}
		where  {$C_1'=r_1'\sqrt{\frac{C_2'}{r_2'}}$, $r_1'=r_2'/2, C_2'=C_2/r_2, r_2'=2r_2$,  and for $k> 2$}
		\begin{equation*}
	C_k'=2C_k\left(\frac{2C_k(k-2)}{kr_k}\right)^{\frac{k-2}{2}}, \qquad 
			r_k'=\frac{kr_k}{2} \,.
		\end{equation*}
		
	\end{lemma}
	We give a sketch of the proof of Lemma \ref{prop:momentBounds} in Appendix \ref{app:proofofprelims} and there expand on how Assumption \ref{DriftAssump} implies moment bounds by showing that Assumption \ref{DriftAssump} implies that the function $V(y) = |y|^k$ is a Lyapunov function for the SDE.

	\begin{Note}\label{note:avg}
		Recall from Note \ref{rem:commentonassumptions} that $\mu^x$ exists and is unique. Therefore, if $\mu^x$ has polynomial moments of all orders, then $\overline{\psi}$ (introduced in \eqref{eqn:barNotation}) is well defined for any arbitrary $\psi\in \funcSpace{m}{m'}$ and $m,m'>0$. The measure $\mu^x$ has polynomial moments of all orders by Lemma \ref{prop:momentBounds}, and hence $\overline{\psi}$ is well defined. Moreover we have that $\widebar{\psi} \in \funcSpace{m}{0}$.
	\end{Note}

	\begin{proposition}\label{lemma:IntAvg}
		Let $\SGfastsx{t}{x}$ be given by \eqref{infty}.  Let $\drift$ and $\diffusion$ satisfy Assumption \ref{ass:polGrowth} (\ref{ass:polGrowthDriftFast},\ref{ass:polGrowthDiffFast}), \ref{UniformEllip} and \ref{DriftAssump}. Let $\phi \in C(\R^n \times \R^d)$ with $\Vnorm{\phi}{0,\pow{}{x},\pow{}{y}}<\infty$ for some  $\pow{}{x},\pow{}{y}>0$; then there exist constants $c,C>0$ (independent of $x\in\R^n,y\in\R^d$) such that the following holds.
		\begin{equation}\label{eqn:SGconv}
			\left| \left(\SGfastsx{s}{x} \phi^x\right) (y) - \mu^x(\phi^x)\right| \leq C \Vnorm{\phi}{0,\pow{}{x},\pow{}{y}} e^{-cs}(1+|x|^{\pow{}{x}}+|y|^{\pow{}{y}}) \quad \text{ for all } x \in \R^n, y\in \R^d, s>0.
		\end{equation}
		As a consequence, there exists some $C'>0$ such that
		\begin{equation*}
			\left| \int_0^\infty  \left[ \left(\SGfastsx{s}{x} \phi^x\right) (y) - \mu^x(\phi^x)  \right]ds \right| \leq C'\Vnorm{\phi}{0,\pow{}{x},\pow{}{y}} (1+|x|^{\pow{}{x}}+|y|^{\pow{}{y}}).
		\end{equation*}
	\end{proposition}
	\begin{proof}
		The proof can be found in Appendix \ref{app:proofofprelims}.
	\end{proof}
	\begin{Note}\label{note:coupledlemmas}
		 Lemma \ref{prop:momentBounds} and Proposition \ref{lemma:IntAvg} still hold if the diffusion coefficient $a$ is also a function of $y$ and indeed we do the proofs in this case. This is relevant in the proofs for the fully coupled case.
	\end{Note}
	The above proposition would be standard if we were not tracing the dependence on the parameter $x$ in the main result. That is, using techniques from \cite{meynTweedie} (see, e.g \cite[Theorem 2.5 ]{MATTINGLYSTUARTergodicity} or, alternatively, using the results of \cite{xie2020ergodicity}) one could obtain
	$$
	\left| \left(\SGfastsx{s}{x} \phi^x\right) (y) - \mu^x(\phi^x)\right| \leq C(x)e^{-c(x)s}.
	$$
	This is not sufficient for our purposes,  since we need a control on the constants $C,c$ which is uniform in the parameter  $x$. Hence the need to adapt the strategy and track the dependence on $x$ in  the proofs of \cite[Theorem 2.5]{MATTINGLYSTUARTergodicity}.  The bulk of the work lies in showing that a minorisation condition holds uniformly in $x$, which substantially follows from Assumption \ref{UniformEllip} (but is still not immediate, as  Assumption \ref{UniformEllip} gives us strict positivity of the transition probabilities for each $x\in \R^n$, but not directly that this strict positivity is uniform in $x$, see Appendix \ref{app:proofofprelims}).\\
	We also note that the convergence  result of the above proposition can also be obtained using semigroup DE such as those in the conclusion of Proposition \ref{prop:derivativeest}, see Note \ref{Rem:der est for Pois unbounded} and  Lemma \ref{lemma:IntAvgUB} on this point.


	\begin{proposition}\label{prop:derivativeest}
		Let $\{\SGfastsx{t}{x} \}_{t\geq 0}$ denote the semigroup defined in \eqref{infty}. Suppose that Assumptions \ref{ass:polGrowth}, \ref{UniformEllip}, \ref{DriftAssump} and \ref{ass:SGcond} hold.
		Let $\psi\in C^{0,2}(\R^n\times\R^d)$ such that $\Vseminorm{\psi}{2,\pow{}{x},\pow{}{y}}< \infty$ for some $\pow{}{x},\pow{}{y}\geq 0$. Then there exists a constant $K>0$ independent of $x \in \R^n$ such that 
		
		\begin{equation}\label{eqn:SGdecayConc2}
			\lvert\SGfastsx{t}{x}\psi^x \rvert_{2,m_x, m_y}\leq K|\psi|_{2,m_{x}^{}, m_{y}^{}}e^{-\cb t}\, , \quad \mbox{for every $x \in \R^n, t\geq 0$}\,.
		\end{equation}
		Moreover, for any $\psi\in C^{0,4}(\R^n\times\R^d)$ such that $\Vseminorm{\psi}{4,\pow{}{x},\pow{}{y}}< \infty$ for some $\pow{}{x},\pow{}{y}\geq 0$,  we have
		\begin{equation}\label{eqn:SGdecayConc}
			\lvert\SGfastsx{t}{x}\psi^x \rvert_{4,\pow{}{x}, \pow{}{y}}\leq K|\psi|_{4,\pow{}{x}, \pow{}{y}}e^{-\cb t}\, , \quad \mbox{for every $x \in \R^n, t>0$}\,.
		\end{equation}
		We recall that $\kappa$ is the constant appearing in Assumption \ref{ass:SGcond} and that the seminorm $\Vseminorm{\cdot}{4,\pow{}{x},\pow{}{y}}$ (defined in \eqref{eqn:Vseminormdef})  contains all the $y$ derivatives up to order $4$ but no $x$ derivatives and does not contain the function itself.
	\end{proposition}
	\begin{proof}
		The proof of this Proposition can be found in Appendix \ref{app:proofofprelims}.
	\end{proof}
		In Section \ref{sec:psneqn} we will require control of both the second and fourth order derivatives of the semigroup $\SGfastsx{t}{x}$. The bound \eqref{eqn:SGdecayConc} implies a bound on the second order derivatives of the semigroup; that is, it implies
		\begin{equation*}
			\lvert\SGfastsx{t}{x}\psi^x \rvert_{2,\pow{}{x}, \pow{}{y}} \leq K|\psi|_{4,\pow{}{x}, \pow{}{y}}e^{-\cb t} \,.
		\end{equation*}
		However the above cannot be applied to  functions $\psi$ that are only twice differentiable; this fact  motivates obtaining \eqref{eqn:SGdecayConc2} as well as \eqref{eqn:SGdecayConc}. \\
		Finally, note that if Assumption \ref{ass:SGcond} holds  with $\zeta_2,\zeta_3 >0$ then it holds with $\zeta_2=\zeta_3 =0$ as well and moreover note that   the derivative estimate \eqref{eqn:SGdecayConc2} can be proven under Assumption \ref{ass:SGcond} with $\zeta_2=\zeta_3 =0$. That is, if there exist $\cb,\zeta_1>0$ independent of $x,y,\xi$ such that for any $x\in \R^n, y, \xi\in \R^d,$ we have
		\begin{equation}\label{eq:driftcondition2}
			\begin{aligned}
				2  \sum_{i,j=1}^n \partial_{y_i}g_j(x,y) \xi_i\xi_j+\sum_{i,j=1}^d\zeta_1(\partial_{y_i}\partial_{y_j}g(x,y), \xi)^2 \leq -\cb \lvert\xi\rvert^2, 
			\end{aligned}
		\end{equation} 
		then \eqref{eqn:SGdecayConc2} holds. This is shown in the proof of Proposition \ref{prop:derivativeest}. 

	\begin{Note}\label{note:smoothing} The DE  of Proposition \ref{prop:derivativeest}, are not smoothing-type estimates; indeed,  smoothing estimates for a given Markov Semigroup  $\cP_t$ are, generally, of the form
	$$
	\lv \mathfrak{D}\cP_t f(x) \rv \leq \frac{1}{t^{\gamma}} \|f\|_{\infty}, \quad f \in C_b, t\in (0,1), x\in \R^n,
	$$
	where $\mathfrak{D}$ is some appropriate differential operator and $\gamma>0$ depends on $\mathfrak{D}$ (for example, in \eqref{derestgeneral} $\mathfrak{D}$ is the usual gradient; if the process is hypoelliptic $\gamma$ will depend on the number of commutators needed to obtain the direction $\mathfrak{D}$), see \cite{bakry2014analysis, crisanottobre} and references therein for a comprehensive review. Note that in the above $f \in C_b$ while in this paper the function $f$ will not be bounded but will be $C^2$ at least; smoothing estimates can be seen as quantifying the ``explosion" of heat-type semigroups as $t\rightarrow 0$. Here we want a specific (i.e. exponential) quantitative estimate for $t$ large. Note that by the semigroup property the short and long-time estimates could be ``glued" together, and this is routinely done in the literature; we do not  do it here as the smoothing effect is not the main concern of the paper, what we are interested in is the long-time regime.
	\end{Note}
	
	
	\section{Poisson equation with parameter} \label{sec:psneqn}
	In this section we study the Poisson problem \eqref{poisprobgeneral}. In particular, in Lemma \ref{lemma:probRep} we show that 
	the representation formula \eqref{ft1rep}  provides a classical solution to  \eqref{poisprobgeneral} (in the sense that the function in \eqref{ft1rep} is smooth in $y$)  in $L^2(\mu^x)$. \footnote{This functional choice is quite natural as the semigroup $\SGfastsx{t}{x}$  can be extended to a $C_0$-semigroup in $L^2(\mu^x)$ (for each $x$, because $\mu^x$ is an invariant measure), and it had already been made in \cite{cattiaux2011central}.}  In Proposition \ref{lemma:IntAvgDerRepresentation} (and in Proposition \ref{lemma:IntAvgDerDer}) we prove regularity properties in the parameter $x$ of the solution and of the measure $\mu^x$.
	
	Formula \eqref{ft1rep} is by now standard in the case in which the operator $\cL^x$ has bounded coefficients. To the best of our knowledge, when the coefficients of $\cL^x$ are unbounded,  such a representation formula  has only been studied in \cite{cattiaux2011central}. The validity of formula \eqref{ft1rep} below can heuristically be seen as follows: for $\phi^x \in L^2(\mu^x)$ and  for $T>0$ fixed, let
	\begin{equation}\label{eqn:truncPsn}
		u^{x,T}(y) = -\int_0^T \left((\SGfastsx{s}{x} \phi^x)(y)-\bar{\phi}(x)\right) ds \,.
	\end{equation}
	By the fundamental theorem of calculus, and using the fact that $\cL^x$ is the generator of $\SGfastsx{s}{x}$  we have 
	\begin{equation}
		\cL^x u^{x,T} = \lim_{h \rightarrow 0}\frac{\SGfastsx{h}{x}u^{x,T} - u^{x,T}}{h} = -\left[\partial_h \int^{T+h}_h \left( \SGfastsx{s}{x} \phi^x-\bar{\phi}(x)\right) ds \right]_{h=0} = \phi^x- \SGfastsx{T}{x}\phi^x.
	\end{equation}
	Letting $T \rightarrow \infty$ and using the convergence  $\SGfastsx{T}{x}\phi^x \rightarrow \bar{\phi}(x)$ as $T \rightarrow \infty$,  yields $\cL^x u^x_\phi = \phi^x-\bar{\phi}(x)$, provided $u^{x,T} \rightarrow u^x_\phi$ in $L^2(\mu^x)$ for some function $u^x_\phi  \in L^2(\mu^x)$, since $\cL^x$ is closed. Lemma \ref{lemma:probRep} is largely devoted to proving the latter convergence.

	Finally, before stating Lemma \ref{lemma:probRep} below, we clarify that the lemma and its proof are a simple adaptation of \cite[Corollary 3.10]{cattiaux2011central},  the main difference being that in Lemma \ref{lemma:probRep}, we explicitly impose conditions on the coefficients of the operator $\cL^x$ which give the desired result while \cite[Corollary 3.10]{cattiaux2011central} is phrased in terms  of requirements on the time-behaviour of the semigroup $\SGfastsx{s}{x}$. 
	
	\begin{lemma}\label{lemma:probRep}
		Consider the Poisson equation \eqref{poisprobgeneral}, and let $\SGfastsx{t}{x}$ be given by \eqref{infty}. Assume that $\drift$ and $a$ satisfy Assumption \ref{ass:polGrowth} (\ref{ass:polGrowthDriftFast},\ref{ass:polGrowthDiffFast}), \ref{UniformEllip} and \ref{DriftAssump}.
		Then, for any $\pow{}{x},\pow{}{y} \geq 0$ and  $\Vnorm{\phi^x}{0, m_x,m_y} < \infty$, the function $u^x_\phi$ defined by \eqref{ft1rep} 
		is well defined, (i.e for each $x \in \R^n$ the function $u_\phi^x$ exists and $u_\phi^x= u_\phi^x(y)<\infty$ for almost every $y\in \R^d$), it is a classical solution to \eqref{poisprobgeneral}, and $u^x_\phi \in L^2(\mu^x)$ for each $x$. Furthermore this solution is infinitely differentiable in the $y$ variable.
	\end{lemma}
	\begin{proof}
		The proof of Lemma \ref{lemma:probRep} can be found in  Appendix \ref{app:proofspsneqn} . \footnote{
				The proof of the representation formula \eqref{ft1rep} that we do in appendix follows the approach of \cite{cattiaux2011central}.  An alternative approach to obtain \eqref{ft1rep} is to follow steps (a)-(d) in the proof of \cite[Theorem 1]{pardoux2001} and try and relax the assumption on the boundedness of the coefficients there. The step where removing the boundedness of the coefficients causes the most difficulty is step (c) which concerns proving continuity of the solution. This can be done in our setting,  since we assume greater regularity on $\phi^x$ and therefore can construct a sequence $u^{x,T}$, as defined in \eqref{eqn:truncPsn}, and one can show that this sequence converges locally uniformly as $T\to \infty$ to $u_\phi^x$ and is continuous. We don't take this approach here as we will need anyway our approach of using DE to remove boundedness of the coefficients when proving the smoothness in $x$ of higher derivatives of the solution (i.e. not just the representation formula). 
		}
	\end{proof}
We now  study the properties of the solution  $u_\phi^x$. 
	We  use the expression $\LfastxDeri$ to denote the second order differential operator obtained by taking the partial derivative of the coefficients of $\Lfastx{x}$ by $x_i$, i.e
	\begin{equation}\label{threestar}
		\left(\LfastxDeri u^{x}\right)(y) := \slowDer{i}g (x,y) \cdot \nabla_y u^x(y) + \slowDer{i}\left(a(x)  a(x)^T\right) : \text{Hess}_y u^x(y),
	\end{equation}
	where $\slowDer{i}\left(a(x)  a(x)^T\right)$ is meant element-wise. 
	
	
	\begin{proposition}
		\label{lemma:IntAvgDerRepresentation}
		Let $\SGfastsx{t}{x}$ be given by \eqref{infty}.  Let $\drift$ and $\diffusion$ satisfy Assumption \ref{ass:polGrowth} (\ref{ass:polGrowthDriftFast},\ref{ass:polGrowthDiffFast}), \ref{UniformEllip}, \ref{DriftAssump} and \ref{ass:SGcond}. Furthermore let $\phi \in \funcSpace{\pow{}{x}}{\pow{}{y}}$ for some $\pow{}{x},\pow{}{y}>0$. Then:
		
		i) the functions $x \mapsto \left(\SGfastsx{t}{x} \phi^{x}\right)(y)$, $x \mapsto \pa_{y_i}\left(\SGfastsx{t}{x} \phi^{x}\right)(y)$, 
		$x \mapsto \pa^2_{y_iy_j}\left(\SGfastsx{t}{x} \phi^{x}\right)(y)$ are continuous locally uniformly with respect to $x$ and $y$. Indeed, there exist $c,C>0$ such that for every $x \in \R^n$, $y\in \R^d$ and $0<h< 1$, the following bounds hold for all $t \geq 0$ :
		
		\begin{equation}
			\label{unifConv1}
			\lvert\left(\SGfastsx{t}{x+h} \phi^{x+h}\right)(y)-\left(\SGfastsx{t}{x} \phi^{x}\right)(y)\rvert \leq 	hC \Vnorm{\phi}{2,\pow{}{x},\pow{}{y}} (1+ |y|^{M^{g}_{y}}+|x|^{2\pow{}{x}}),
		\end{equation}
		and
		\begin{equation}\label{eqn:unifConv}
			\begin{split}
				\lvert\left(\fastDer{i}\SGfastsx{t}{x+h}\phi^{x+h}\right)(y) &-  \left(\fastDer{i}\SGfastsx{t}{x}\phi^x\right)(y)\rvert + \lvert  \left(\fastDerDer{i}{j}\SGfastsx{t}{x+h}\phi^{x+h}\right)(y) -  \left(\fastDerDer{i}{j}\SGfastsx{t}{x}\phi^x\right)(y) \rvert
				\\&
				\leq 
				hC \Vnorm{\phi}{4,\pow{}{x},\pow{}{y}} (1+ |y|^{M^{g}_{y}}+|x|^{2\pow{}{x}}),
			\end{split}
		\end{equation}
		where $M^{g}_{y} \coloneqq \max\{2\pow{g}{y},\pow{g}{y}+m^{}_{y}\}$.
		
		ii) the functions $x \mapsto \left(\SGfastsx{t}{x} \phi^{x}\right)(y) $ and $x \mapsto \mu^x( \phi^{x})$ are continuously differentiable in $x$ and we have the following representation formulas for their derivatives:
		\begin{equation}\label{eqn:repSG}
			\partial_{x_i} \left(\SGfastsx{t}{x} \phi^{x}\right)(y) =  \left(\SGfastsx{t}{x}\partial_{x_i}  \phi^{x}\right)(y) + \int_0^{t}\left[ \SGfastsx{t-s}{x} \LfastxDeri \SGfastsx{s}{x}\phi^{x}\right](y) ds
		\end{equation}
		and
		\begin{equation}\label{eqn:repInvMeas}
			\partial_{x_i}\mu^x( \phi^{x})= \mu^x(\partial_{x_i} \phi^{x})  +  \int_0^{\infty} \mu^x\left( \LfastxDeri \left(\SGfastsx{s}{x}\phi^{x}\right)\right) ds\,,
		\end{equation}
		for all $i \in \{1, \cdots, n \}$. Furthermore, for all $i \in \{1, \cdots, n \}$, there exist $C, C'>0$ (which are independent of $x$ and $y$) such that \eqref{eq:intAvgDerInv} and \eqref{eq:intAvgDer} hold.

	\end{proposition}
	\begin{Note}\label{Rem:der est for Pois unbounded}
		We make a number of comments on Proposition \ref{lemma:IntAvgDerRepresentation}.
		\begin{enumerate}
			\item Equation \eqref{eqn:repSG} (and similarly \eqref{eqn:repSGderder} below) is called a  ``transfer formula'',  as it allows to write $x$ derivatives of the semigroup $\SGfastsx{t}{x}\phi^x$ in terms of $y$ derivatives of the same semigroup (with the advantage that the latter derivatives exist by ellipticity in $y$ of the generator $\cL^x$). Indeed, the left hand side of \eqref{eqn:repSG} contains an $x$ derivative of $\SGfastsx{t}{x}\phi^x$, while  the right hand side contains (besides $x$ derivatives of $\phi^x$ and of the coefficients $g, a$)  two $y$ derivatives of the semigroup $\SGfastsx{t}{x}\phi^x$ (and no $x$ derivatives of  $\SGfastsx{t}{x}\phi^x$), as the operator $\frac{\partial \cL^x}{\partial x_i}$ is a second order differential operator in the variable $y$, see \eqref{threestar}. In other words, ``one $x$ derivative comes at the cost of two $y$ derivatives''. In Section \ref{subsec:detailedproof}  we will need bounds on the second $x$ derivative of $\SGfastsx{t}{x}\phi^x$, hence the need to control $y$ derivatives of the semigroup up to order four (see also \eqref{eqn:repSGderder}). This motivates the definition of the seminorm $\Vnorm{\cdot}{4,\pow{}{x},\pow{}{y}}$, which involves $y$ derivatives up to order four,  but only $x$ derivatives up to order $2$.
			\item 
			The $y$ derivatives of $\SGfastsx{t}{x}$ not only exist,   but they can also be estimated, and indeed in Proposition \ref{prop:derivativeest} we obtained exponentially fast decay for such derivatives.  Since the estimates of Proposition \ref{prop:derivativeest} can be integrated  in time over $(0, \infty)$, combined with the transfer formula, they can be employed to obtain formula \eqref{eqn:repInvMeas} for $\mu^x$.
			\item The right hand side of  \eqref{eqn:repSG} can be written in terms of the density of the transition probabilities of the process \eqref{fastF} rather than in terms of the semigroup $\SGfastsx{t}{x}$. This is what is done in \cite{pardoux2003}.  If our drift $g$ was bounded then we could use estimates on the derivatives of the transition density, as is the approach of \cite{pardoux2003}. In the case of locally Lipschitz coefficients, these estimates are hard to obtain. Hence, we use Proposition \ref{prop:derivativeest} in order to bound the derivatives of the semigroup directly.
			\item An important ingredient in the proof of Proposition \ref{lemma:IntAvgDerRepresentation} (and Proposition \ref{lemma:IntAvgDerDer} below) is the convergence of the semigroup $\SGfastsx{t}{x}\phi^x$ as $t\to \infty$. This was obtained in Proposition \ref{lemma:IntAvg} by using Lyapunov techniques, however an alternative strategy of proof is to use the derivative estimates of Proposition \ref{prop:derivativeest}. The latter strategy is used in Lemma \ref{lemma:IntAvgUB}. The main difference between Proposition \ref{lemma:IntAvg}  and Lemma \ref{lemma:IntAvgUB} is the class of functions for which they hold, as Proposition \ref{lemma:IntAvg} holds for $\phi$ with $\Vnorm{\phi}{0,\pow{}{x},\pow{}{y}}<\infty$ whereas Lemma \ref{lemma:IntAvgUB} requires $\Vseminorm{\phi}{2,\pow{}{x},\pow{}{y}}<\infty$; that is, the latter requires control on higher derivatives of the function to which it is applied (though not of the function itself). This would cause complications if we were to use Lemma \ref{lemma:IntAvgUB} within the proof of Proposition \ref{lemma:IntAvgDerRepresentation} (and Proposition \ref{lemma:IntAvgDerDer}), as the functions to which the result needs to be applied in that context are quite involved;  this is the main reason why Proposition \ref{lemma:IntAvg} is needed  in this paper. 
		\end{enumerate}
	\end{Note}
	\begin{proof}[Proof of Proposition \ref{lemma:IntAvgDerRepresentation}]
		Because we need to prove continuity and differentiability, throughout the proof we will consider increments $he_i$ where $e_i$ is the $i^{th}$ element of the standard basis of $\R^n$ and $h\in \R$. To simplify notation we write the proof as if in one dimension. The higher dimensional proof follows by replacing $h$ with $h e_i$. Throughout the proof we take $|h| < 1$.
		
		
		Before outlining the steps of this proof, we first show some preliminary bounds which will be used repeatedly. First,  notice the following:
		\begin{equation}\label{eqn:elemBound1}
			\begin{split}
				\left| \frac{\Lfastx{x+h}-\Lfastx{x}}{h}\left( \SGfastsx{s}{x+h}\phi^{x+h}\right)(y)\right| &\leq 
				\sum_{j=1}^d \left|\frac{\drift_j(x+h,y)-\drift_j(x,y)}{h}\fastDer{j}\SGfastsx{s}{x+h}\phi^{x+h}(y)\right| \\
				&+ \sum_{j,k = 1}^d \left|  \frac{\diffusion_{jk}(x+h)-\diffusion_{jk}(x) }{h}\fastDerDer{j}{k}\SGfastsx{s}{x+h}\phi^{x+h}(y)\right| 
				\\&\leq  C\Vseminorm{\phi}{2,\pow{}{x},\pow{}{y}}e^{-cs}(1+ |y|^{M^{g}_y}+|x|^{2\pow{}{x}}),
			\end{split}
		\end{equation}
		by \eqref{eqn:SGdecayConc2} and Assumption \ref{ass:polGrowth} (\ref{ass:polGrowthDriftFast},\ref{ass:polGrowthDiffFast}), where we use the fact that that the coefficients $g$ and $A$ are differentiable with a bound on the derivative. {Here $C$ and $c$ are constants which may change line by line.}
		Applying the semigroup to both sides of \eqref{eqn:elemBound1}, we get \begin{equation}\label{eqn:elemBound2}
			\begin{split}
				\left|\left(\SGfastsx{t-s}{x}\frac{\Lfastx{x+h}-\Lfastx{x}}{h} \SGfastsx{s}{x+h}\phi^{x+h}\right)(y)\right| &\leq C\Vseminorm{\phi}{2,\pow{}{x},\pow{}{y}}e^{-cs}(1 + \mathbb{E}\left[|\fastFixedI_{t-s}|^{M^{g}_y}\right]+|x|^{2\pow{}{x}}) \\
				&\stackrel{\eqref{eqn:momBoundsY1}}{\leq} C\Vseminorm{\phi}{2,\pow{}{x},\pow{}{y}}e^{-cs}(1+ |y|^{M^{g}_y}+|x|^{2\pow{}{x}}).
			\end{split}
		\end{equation}
		The LHS expression of \eqref{eqn:elemBound2} will appear in the representation formula for the difference quotient that we consider (see Step 1 below). 
		Similarly to what we have done to obtain \eqref{eqn:elemBound1},  we get the bounds
		\begin{equation}\label{eqn:elemBound3}
			\left|\fastDer{i}\left(\frac{\Lfastx{x+h}-\Lfastx{x}}{h}\left( \SGfastsx{s}{x+h}\phi^{x+h}\right)\right)(y)\right|
			\leq  C\Vseminorm{\phi}{4,\pow{}{x},\pow{}{y}}e^{-cs}(1+ |y|^{M^{g}_y}+|x|^{2\pow{}{x}})
		\end{equation}
		and 
		\begin{equation}\label{eqn:elemBound4}
			\left|\fastDerDer{i}{l}\left(\frac{\Lfastx{x+h}-\Lfastx{x}}{h}\left( \SGfastsx{s}{x+h}\phi^{x+h}\right)\right)(y)\right|
			\leq  C\Vseminorm{\phi}{4,\pow{}{x},\pow{}{y}}e^{-cs}(1+ |y|^{M^{g}_y}+|x|^{2\pow{}{x}})
		\end{equation}
		by \eqref{eqn:SGdecayConc} and Assumption \ref{ass:polGrowth} (\ref{ass:polGrowthDriftFast},\ref{ass:polGrowthDiffFast}). The final bound we need before outlining the structure of this proof is
		\begin{equation}
			\begin{split}\label{eqn:SGdecay}
				\left|\LfastxDeri \SGfastsx{s}{x}\phi^{x}(y)\right| &= \sum_{j = 1}^{d}  \left|\slowDer{i}\drift_j(x,y)\fastDer{j}\SGfastsx{s}{x}\phi^{x}(y)\right| +  \sum_{j,k = 1}^{d}\left| \slowDer{i} \diffusion(x)_{jk}\fastDerDer{j}{k}\SGfastsx{s}{x}\phi^{x}(y)\right| \\ 
				&\leq C\Vseminorm{\phi}{2,\pow{}{x},\pow{}{y}}e^{-cs} (1+ |y|^{M^{g}_y}+|x|^{2\pow{}{x}}),
			\end{split}
		\end{equation}
		which follows from \eqref{eqn:SGdecayConc2} and Assumption \ref{ass:polGrowth} (\ref{ass:polGrowthDriftFast},\ref{ass:polGrowthDiffFast}).
		
		Now we outline the strategy of proof, before addressing each step consecutively.
		
		\begin{itemize}
			\item Step 1: Define the difference quotient \begin{equation}\label{eqn:quotient} q_{t}^{h,\phi^{x}}( y) \coloneqq \frac{1}{h}\left[ \left(\SGfastsx{t}{x+h}\phi^{x+h}\right)(y) -  \left(\SGfastsx{t}{x}\phi^x\right)(y) \right]\end{equation} (again, $q_{t}^{h,\phi^{x}}$ should be $q_{t}^{he_i,\phi^{x}}$) and obtain the representation formula			\begin{equation}\label{eqn:quotientRep}
				q_{t}^{h,\phi^{x}}( y) = \left(\SGfastsx{t}{x} q_0^{h,\phi^x}\right)(y) +\int_0^{t} \left[ \SGfastsx{t-s}{x}
				\left[\left(\frac{\Lfastx{x+h}- \Lfastx{x}}{h}\right)\left(\SGfastsx{s}{x+h}\phi^{x+h}\right)\right]\right](y)ds.
			\end{equation}
			The above expression will be instrumental for proving $(i)$ as well as both \eqref{eqn:repSG} and \eqref{eqn:repInvMeas}, which are obtained in Step 2 and Step 3 respectively.
			\item Step 2: Let $h \rightarrow 0$ in the above expression to obtain (\ref{eqn:repSG}). At least formally, it is easy to see that letting $h \rightarrow 0$ one obtains the desired result, so this step consists of justifying using the Dominated Convergence Theorem (DCT) to take this limit.
			\item Step 3: Let $t \rightarrow \infty$ in (\ref{eqn:repSG}) to obtain \eqref{eqn:repInvMeas}. Again, formally it is easy to see that letting $t \rightarrow \infty$ one obtains \eqref{eqn:repInvMeas}, so this step primarily consists of justifying the limit. The bound \eqref{eq:intAvgDerInv} is then a consequence of \eqref{eqn:repInvMeas}, as we will show.
			\item Step 4: Integrate $	\partial_{x_i} \left(\SGfastsx{t}{x}\phi^x\right)(y) -	\partial_{x_i}\mu^x\left(\phi^x\right)$ with respect to $t$ using the representation formulas \eqref{eqn:repSG}-\eqref{eqn:repInvMeas} to obtain \eqref{eq:intAvgDer}.
		\end{itemize}
		
		Step 1:
		We differentiate (\ref{eqn:quotient}) in time to get
		\begin{align*}
			\partial_{t} q_{t}^{h,\phi^{x}}( y) &= \frac{1}{h}\left[\left(\Lfastx{x+h} \SGfastsx{t}{x+h}\phi^{x+h}\right)(y) - \left(\Lfastx{x} \SGfastsx{t}{x}\phi^{x}\right)(y) \right] \\
			&= \left[\left(\frac{\Lfastx{x+h}- \Lfastx{x}}{h}\right)\left(\SGfastsx{t}{x+h}\phi^{x+h}\right)\right](y)+ \left[ \Lfastx{x}\left( \frac{\SGfastsx{t}{x+h}\phi^{x+h} - \SGfastsx{t}{x}\phi^{x}}{h} \right)\right](y) \\
			&=
			\left[\left(\frac{\Lfastx{x+h}- \Lfastx{x}}{h}\right)\left(\SGfastsx{t}{x+h}\phi^{x+h}\right)\right](y)+ \left(\Lfastx{x}  q_{t}^{h,\phi^{x}}\right)(y).
		\end{align*}
		We can now use the variation of constants formula to obtain (\ref{eqn:quotientRep}). To prove \emph{(i)}, observe that by \eqref{eqn:quotientRep} we have
		\begin{equation}\label{eqn:diffquotient}
			\begin{split}
				&\left(\SGfastsx{t}{x+h}\phi^{x+h}\right)(y) -  \left(\SGfastsx{t}{x}\phi^x\right)(y) \\
				&= h\left( \left(\SGfastsx{t}{x} q_0^{h,\phi^x}\right)(y) +\int_0^{t} \left[ \SGfastsx{t-s}{x}
				\left[\left(\frac{\Lfastx{x+h}- \Lfastx{x}}{h}\right)\left(\SGfastsx{t}{x+h}\phi^{x+h}\right)\right]\right](y)ds\right).
			\end{split}
		\end{equation}
		Then, by differentiating in $y$ (which is allowed because of Assumption  \ref{UniformEllip}), one gets
		\begin{equation}\label{eqn:1diffquotient}
			\begin{split}
				&\left(\fastDer{i}\SGfastsx{t}{x+h}\phi^{x+h}\right)(y) -  \left(\fastDer{i}\SGfastsx{t}{x}\phi^x\right)(y) \\&= h\Bigg( \left(\fastDer{i}\SGfastsx{t}{x} q_0^{h,\phi^x}\right)(y)
				+\int_0^{t} \left[ \fastDer{i}\SGfastsx{t-s}{x}
				\left[\left(\frac{\Lfastx{x+h}- \Lfastx{x}}{h}\right)\left(\SGfastsx{t}{x+h}\phi^{x+h}\right)\right]\right](y)ds\Bigg),
			\end{split}
		\end{equation}
		where we can take the derivative under the integral sign because of \eqref{eqn:elemBound3}.
		If we differentiate again we obtain
		\begin{equation}\label{eqn:2diffquotient}
			\begin{split}
				&\left(\fastDerDer{i}{j}\SGfastsx{t}{x+h}\phi^{x+h}\right)(y) -  \left(\fastDerDer{i}{j}\SGfastsx{t}{x}\phi^x\right)(y) \\
				&= h\Bigg( \left(\fastDerDer{i}{j}\SGfastsx{t}{x} q_0^{h,\phi^x}\right)(y) +\int_0^{t} \left[ \fastDerDer{i}{j}\SGfastsx{t-s}{x}
				\left[\left(\frac{\Lfastx{x+h}- \Lfastx{x}}{h}\right)\left(\SGfastsx{t}{x+h}\phi^{x+h}\right)\right]\right](y)ds\Bigg),
			\end{split}
		\end{equation}
		
		where again the derivative and integral commute because of \eqref{eqn:elemBound4}. To work on the first addends of \eqref{eqn:diffquotient}-\eqref{eqn:2diffquotient}, we observe  that, since $\phi \in \funcSpace{\pow{}{x}}{\pow{}{y}}$ (so that in particular that $\phi \in C^{2,4}(\R^n \times \R^d)$). we have 
		\begin{align*}
			q_0^{h,\phi^x}(y)  &= \frac{\phi^{x+h}(y) - \phi^x(y)}{h}= \partial_{x_i}\phi^{\xi}(y),
		\end{align*}
		for some $\xi \in (x, x+h)$.
		Hence we obtain the following bounds, for $|h|<1$:  
		\begin{equation}\label{eqn:boundfirsttermnodiff}
			\left|q_0^{h,\phi^x}(y)\right| \leq C\Vnorm{\phi}{2,\pow{}{x},\pow{}{y}}(1+ |y|^{\pow{}{y}}+|x|^{\pow{}{x}})\, ,
		\end{equation}
		and
		\begin{equation}\label{eqn:boundfirstterm}
			\left|\fastDer{j}q_0^{h,\phi^x}(y)\right|, \left|\fastDerDer{j}{k}q_0^{h,\phi^x}(y) \right| \leq C\Vnorm{\phi}{4,\pow{}{x},\pow{}{y}}(1+ |y|^{\pow{}{y}}+|x|^{\pow{}{x}}).
		\end{equation}
		From \eqref{eqn:elemBound2}, \eqref{eqn:diffquotient} and \eqref{eqn:boundfirsttermnodiff}, we obtain \eqref{unifConv1}.
		Now, using \eqref{eqn:SGdecayConc2} (which we can use because of \eqref{eqn:elemBound3}, \eqref{eqn:elemBound4} and \eqref{eqn:boundfirstterm}), \eqref{eqn:1diffquotient} and \eqref{eqn:2diffquotient} we have, for $|h|<1$,
		\begin{align*}
			\lvert\left(\fastDer{i}\SGfastsx{t}{x+h}\phi^{x+h}\right)(y) &-  \left(\fastDer{i}\SGfastsx{t}{x}\phi^x\right)(y)\rvert +\lvert  \left(\fastDerDer{i}{j}\SGfastsx{t}{x+h}\phi^{x+h}\right)(y) -  \left(\fastDerDer{i}{j}\SGfastsx{t}{x}\phi^x\right)(y) \rvert
			\\&\leq 
			hC \Vnorm{\phi}{4,\pow{}{x},\pow{}{y}}
			 (1+ |y|^{M^{g}_y}+|x|^{2\pow{}{x}}),
		\end{align*}
		which implies \eqref{eqn:unifConv}.

		Step 2: We now justify letting $h \rightarrow 0$ in \eqref{eqn:quotientRep}. We start from the first  addend in \eqref{eqn:quotientRep}; by definition
		$$
		\left(\SGfastsx{t}{x} q_0^{h,\phi^x}\right)(y) = \mathbb E \left[ \frac{\phi^{x+h}(\fastFixedI_{t}) - \phi^x(\fastFixedI_{t})}{h}\right]\,.
		$$
		Since $\phi$ is differentiable, the integrand on the RHS of the above can be trivially bounded by
		$$
		\left\vert \frac{\phi^{x+h}(\fastFixedI_{t}) - \phi^x(\fastFixedI_{t})}{h} \right \vert
		\leq \left \vert\pa_{x_i} \phi^{\xi}(\fastFixedI_{t}) \right \vert,
		$$
		where $\xi \in (x, x+h)$ so that now using polynomial growth of $\slowDer{i}\phi$, the bound \eqref{eqn:momBoundsY1} allows to apply the DCT to the first  addend in \eqref{eqn:quotientRep}, and conclude that such a term
		tends to $\SGfastsx{t}{x}\partial_x \phi^{x}(y)$ as $h\rightarrow 0$. 
		
		We now move on to the second addend in \eqref{eqn:quotientRep}. For each $x \in \R^n,y\in \R^d$ and $s>0$ we have
		\begin{equation*}
			\lim_{h\rightarrow 0} \left[ \frac{\Lfastx{x+h}-\Lfastx{x}}{h}\left( \SGfastsx{s}{x+h}\phi^{x+h}\right)(y) \right]= \LfastxDeri \SGfastsx{s}{x}\phi^{x}(y),
		\end{equation*}
		by \emph{(i)}, the smoothness of $\phi$ and the coefficients of $\mathcal{L}^x$. Hence, if the DCT holds for the integral(s) in \eqref{eqn:quotientRep}, then we are done. Notice, indeed, that since $\SGfastsx{t-s}{x}$ is an integral itself, we need to justify using the DCT to pass the limit $h\rightarrow 0$ both under the time integral and under the integral implied by the definition of $\SGfastsx{t-s}{x}$. The bound \eqref{eqn:elemBound2} can be used for such tasks. The RHS of \eqref{eqn:elemBound2} is indeed integrable both in time and in space, the latter because of Lemma \ref{prop:momentBounds}.
		
		Step 3: Now we let $t \rightarrow \infty$. Letting $t \rightarrow \infty$ in the first term on the RHS of \eqref{eqn:repSG} is straightforward from \eqref{eqn:SGconv} of Proposition \ref{lemma:IntAvg}. As for the second addend, we write the integral as
		\begin{equation*}
			\int_0^{t}\left[ \SGfastsx{t-s}{x} \LfastxDeri \SGfastsx{s}{x}\phi^{x}\right](y) ds =
			\int_0^{\infty} \mathbf{1}_{s < t}\left[\SGfastsx{t-s}{x} \LfastxDeri \SGfastsx{s}{x}\phi^{x}\right](y) ds.
		\end{equation*}
		Then,  by (\ref{eqn:SGconv}) of Proposition \ref{lemma:IntAvg}, which we are able to use because of (\ref{eqn:SGdecay}), we have
		\begin{equation*} \lim_{t \rightarrow \infty}\mathbf{1}_{s < t}\left[\SGfastsx{t-s}{x} \LfastxDeri \SGfastsx{s}{x}\phi^{x}\right](y)  = \int_{\mathbb{R}^d} \LfastxDeri \SGfastsx{s}{x}\phi^{x}(y) \mu^x(dy), 
		\end{equation*}
		for every $s$ fixed; hence, assuming we can use the DCT, we have
		\begin{equation}\label{eqn:lim}
			\lim_{t \rightarrow \infty} \int_0^{\infty} \mathbf{1}_{s < t}\left[\SGfastsx{t-s}{x} \LfastxDeri \SGfastsx{s}{x}\phi^{x} \right](y)ds =\int_0^{\infty} \int_{\mathbb{R}^d} \LfastxDeri \SGfastsx{s}{x}\phi^{x}(y) \mu^x(dy) ds.
		\end{equation} To justify the use of the DCT in the above, we apply the semigroup to \eqref{eqn:SGdecay} and use (\ref{eqn:momBoundsY1}) to get
		\begin{equation*}
			\left|\mathbf{1}_{s < t}\left[\SGfastsx{t-s}{x} \LfastxDeri \SGfastsx{s}{x}\phi^{x}\right](y) \right| \leq C\Vseminorm{\phi}{2,\pow{}{x},\pow{}{y}}e^{-cs}(1+ |y|^{M^{g}_y}+|x|^{2\pow{}{x}}),
		\end{equation*}
		where the expression on the RHS is integrable in $s$. This justifies the limit \eqref{eqn:lim}.
		So far we have proved that the RHS of \eqref{eqn:repSG} tends to the RHS of \eqref{eqn:repInvMeas} as $t \rightarrow \infty$. We now want to do the same for the LHS and to this end we must show that we are allowed to exchange the limit $t \rightarrow \infty$ with the derivative $\slowDer{i}$, which we will do by proving that the convergence shown thus far is locally uniform. If we prove this, then \eqref{eqn:repInvMeas} follows again from Proposition \ref{lemma:IntAvg}.
		From \eqref{eqn:repSG} and \eqref{eqn:lim} we have 
		\begin{align*}
			\partial_{x_i} \left(\SGfastsx{t}{x}\phi^x\right)(y) &- \lim_{t\rightarrow\infty} 	\partial_{x_i} \left(\SGfastsx{t}{x}\phi^x\right)(y)	 \\
			&= \SGfastsx{t}{x}\partial_{x_i} \phi^{x}(y) - \int_{\mathbb{R}^d} \partial_{x_i} \phi^{x}(y) \mu^x(dy) \\
			&+\int_0^{t} \SGfastsx{t-s}{x}\left[ \LfastxDeri \SGfastsx{s}{x}\phi^{x}\right](y)ds- \intPos \int_{\mathbb{R}^d} \LfastxDeri \SGfastsx{s}{x}\phi^{x}(y) \mu^x(dy) ds    \\
			&= \left(\SGfastsx{t}{x} \slowDer{i}\phi^x\right) (y) - \mu^x(\slowDer{i}\phi^x)\\
			&+\underbrace{\int_0^{t}\left(\SGfastsx{t-s}{x}-\mu^x\right) \left(  \LfastxDeri \SGfastsx{s}{x}\phi^{x} \right)ds}_{\eqqcolon\RNum{1}}\\ &- \underbrace{\int_{t}^\infty\left( \int_{\mathbb{R}^d} \LfastxDeri \SGfastsx{s}{x}\phi^{x}(\tilde{y}) \mu^x(d\tilde{y})\right) ds}_{\eqqcolon\RNum{2}}.
		\end{align*}
		We now want to show that each of the addends in the above converge to zero locally uniformly in $x$ and $y$.
		The claim is trivial for the first term, since by Proposition \ref{lemma:IntAvg}, we have 
		\begin{equation}\label{eqn:unifConv1st}
			|\left(\SGfastsx{t}{x} \partial_{x_i} \phi^x\right) (y) - \mu^x(\partial_{x_i} \phi^x)| \leq C\Vnorm{\phi}{2,\pow{}{x},\pow{}{y}}e^{-ct}(1+ |y|^{\pow{}{y}}+|x|^{\pow{}{x}}) \,. 
		\end{equation}
		If we show the following two bounds, the proof is concluded:
		\begin{align}\label{eqn:int1}
			|\RNum{1}| &\leq C\Vseminorm{\phi}{2,\pow{}{x},\pow{}{y}}e^{-ct}(1+ |y|^{M^{g}_y}+|x|^{2\pow{}{x}}) \\
			\label{eqn:int2}
			|\RNum{2}| &\leq C\Vseminorm{\phi}{2,\pow{}{x},\pow{}{y}}e^{-ct}(1+|x|^{2\pow{}{x}}).
		\end{align}	
		The bound \eqref{eqn:int1} follows from \eqref{eqn:SGdecay} and Proposition \ref{lemma:IntAvg}. We note to the reader that the constants in the exponential seen in \eqref{eqn:SGdecay} and Proposition \ref{lemma:IntAvg} are not necessarily equal. As for \eqref{eqn:int2}, using again \eqref{eqn:SGdecay}, we have
		
		\begin{align*}
			|\RNum{2}| &\leq \int_{t}^\infty\left( \int_{\mathbb{R}^d} \left| \LfastxDeri \SGfastsx{s}{x}\phi^{x}(\tilde{y})\right| \mu^x(d\tilde{y})\right)  ds \\
			&\leq C\Vseminorm{\phi}{2,\pow{}{x},\pow{}{y}} \int_{t}^\infty e^{-cs}\left( \int_{\mathbb{R}^d} (1+ |\tilde{y}|^{M^{g}_y}+|x|^{2\pow{}{x}}) \mu^x(d\tilde{y})\right)  ds,
		\end{align*}
		so that \eqref{eqn:int2} follows from \eqref{eqn:momBoundsYinv} in Lemma \ref{prop:momentBounds}.
		From \eqref{eqn:unifConv1st}, \eqref{eqn:int1} and \eqref{eqn:int2}, we have
		\begin{equation} \label{justification}
			\begin{split}
				|\partial_{x_i} \left(\SGfastsx{t}{x}\phi^x\right)(y) - \lim_{t\rightarrow\infty}\partial_{x_i}\left(\SGfastsx{t}{x}\phi^x\right)(y)|
				&\leq C\Vnorm{\phi}{2,\pow{}{x},\pow{}{y}}e^{-ct}(1+ |y|^{M^{g}_y}+|x|^{2\pow{}{x}})\end{split}
		\end{equation}
		meaning that the convergence of the derivative is locally uniform, and so we can exchange the limit and derivative. This concludes the proof of (\ref{eqn:repInvMeas}).
		
		Now we show (\ref{eq:intAvgDerInv}). Using (\ref{eqn:momBoundsYinv}), (\ref{eqn:repInvMeas}) and (\ref{eqn:SGdecay}) we have
		\begin{equation}\label{eqn:invMeasfirst}
			\begin{split}
				\left|\partial_{x_i}\int_{\mathbb{R}^d} \phi^{x}(y) \mu^x(dy)\right|&\leq   \left|\int_{\mathbb{R}^d} \partial_{x_i} \phi^{x}(y) \mu^x(dy)\right|+  \left|\int_0^{\infty} \int_{\mathbb{R}^d} \LfastxDeri \SGfastsx{s}{x}\phi^{x}(y) \mu^x(dy) ds \right| \\
				&\leq C\Vnorm{\phi}{2,\pow{}{x},\pow{}{y}}(1+|x|^{\pow{}{x}})\\ 
				&+  \int_0^{\infty} \int_{\mathbb{R}^d} C\Vseminorm{\phi}{2,\pow{}{x},\pow{}{y}}e^{-cs} (1+ |y|^{M^{g}_y}+|x|^{2\pow{}{x}}) \mu^x(dy) ds.
			\end{split}
		\end{equation}
		We use \eqref{eqn:momBoundsYinv} again to obtain 
		\begin{equation}\begin{split}\label{eqn:invMeassecond}
				\int_0^{\infty} \int_{\mathbb{R}^d} e^{-cs} (1+ |y|^{M^{g}_y}+|x|^{2\pow{}{x}}) \mu^x(dy) ds
				\leq C\int_0^{\infty} e^{-cs} (1+|x|^{2\pow{}{x}})ds 
				\leq \frac{C}{c}(1+|x|^{2\pow{}{x}}).
			\end{split}
		\end{equation}
		Using \eqref{eqn:invMeasfirst} and \eqref{eqn:invMeassecond} together concludes the proof of \eqref{eq:intAvgDerInv}.

		Step 4: Lastly, we show \eqref{eq:intAvgDer}.
		First we write
		\begin{equation*}
			\left| \partial_{x_i}\int_0^\infty   \SGfastsx{t}{x}\left(\phi^x -\mu^x(\phi^x)\right)(y)dt \right| = \left|\int_0^\infty  \partial_{x_i} \left(\SGfastsx{t}{x}\phi^x\right)(y) -\partial_{x_i}(\mu^x(\phi^x))  dt\right|.
		\end{equation*}
		
		This means that we have, using (\ref{justification})
		\begin{align*}
			\left| \partial_{x_i}\int_0^\infty   \SGfastsx{t}{x}\left(\phi^x -\mu^x(\phi^x)\right)(y)dt \right| &\leq\int_0^\infty |\partial_{x_i} \left(\SGfastsx{t}{x}\phi^x\right)(y) -\partial_{x_i}(\mu^x(\phi^x)) | dt \\
			&\leq  C\Vnorm{\phi}{2,\pow{}{x},\pow{}{y}}(1+ |y|^{M^{g}_y}+|x|^{2\pow{}{x}}),
		\end{align*}
		and so the proof is done.
	\end{proof}
	
	\begin{lemma}\label{lemma:IntAvgUB}
		Let the assumptions of Proposition \ref{lemma:IntAvgDerRepresentation} hold. Then for every $\phi\in C^{0,2}(\R^n\times\R^d)$ such that  $\Vseminorm{\phi}{2,\pow{}{x},\pow{}{y}}<\infty$ for some $\pow{}{x},\pow{}{y}>0$ we have
		\begin{equation}\label{eqn:SGconvUB}
			\left| \left(\SGfastsx{s}{x} \phi^x\right) (y) - \mu^x(\phi^x)\right| \leq C\Vseminorm{\phi}{2,\pow{}{x},\pow{}{y}}e^{-cs}(1+ |y|^{M^{g,}_y}+|x|^{2\pow{}{x}}),
		\end{equation}
		where $M^{g}_{y}\coloneqq \max\{2\pow{\drift}{y},\pow{\drift}{y}+m^{}_{y}\}$.
		Combining \eqref{ft1rep} and \eqref{eqn:SGconvUB} gives
		\begin{equation}\label{eqn:IntAvg1}
			\lvert u_\phi^x(y)\rvert\leq \left| \int_0^\infty  \left[ \left(\SGfastsx{s'}{x} \phi^x\right) (y) - \mu^x(\phi^x)  \right]ds' \right| \leq C\Vseminorm{\phi}{2,\pow{}{x},\pow{}{y}}(1+ |y|^{M^{g}_y}+|x|^{2\pow{}{x}}).
		\end{equation}
	\end{lemma}
	\begin{proof}
		The proof can be found in Appendix \ref{app:proofspsneqn}.
	\end{proof}
	
	We now state the analogous result of Proposition \ref{lemma:IntAvgDerRepresentation}, for higher derivatives. The proof is analogous to the proof of Proposition \ref{lemma:IntAvgDerRepresentation}, hence it is deferred to Appendix \ref{app:proofspsneqn}.

	\begin{proposition}
		\label{lemma:IntAvgDerDer}
		Let the assumptions of Proposition \ref{lemma:IntAvgDerRepresentation} hold and furthermore let $\phi \in \funcSpace{\pow{}{x}}{\pow{}{y}}$ for some $\pow{}{x},\pow{}{y}>0$. Then:

		i) for all $i,j,k \in \{1, \ldots, n \}$
		\begin{equation}\label{eqn:intAvgDerFast}
			\left| \fastDer{j}\slowDer{i}\left(\SGfastsx{t}{x} \phi^{x}\right) (y) \right| + \left| \fastDerDer{j}{k}\slowDer{i}\left(\SGfastsx{t}{x} \phi^{x}\right) (y) \right| \leq  C \Vnorm{\phi}{4,\pow{}{x},\pow{}{y}}e^{-ct} (1+ |y|^{M^{g}_{y}}+|x|^{2\pow{}{x}}),
		\end{equation}
		where $M^{g}_{y}$ is as defined in Proposition \ref{lemma:IntAvgDerRepresentation}. The functions $x \mapsto \fastDer{j}\slowDer{i}\left(\SGfastsx{t}{x} \phi^{x}\right)(y) $ and $x \mapsto \fastDerDer{j}{k}\slowDer{i}\left(\SGfastsx{t}{x} \phi^{x}\right)(y) $ are continuous in $x$.
		
		ii) the functions $x \mapsto \left(\SGfastsx{t}{x} \phi^{x}\right)(y) $ and $x \mapsto \mu^x( \phi^{x})$ are twice continuously differentiable and we have the following representation formulas for their second derivatives:
		\begin{equation}\label{eqn:repSGderder}
			\begin{split}
				&\slowDerDer{i}{j} \left(\SGfastsx{t}{x}\phi^x\right)(y) =  (\SGfastsx{t}{x}\slowDerDer{i}{j}\phi^x)(y) \\
				&+ \int_{0}^{t} \SGfastsx{t-s}{x}\Big[\LfastxDeriDerj \SGfastsx{s}{x}\phi^{x}(y) + \LfastDerj{x}\slowDer{i} \left(\SGfastsx{s}{x}\phi^{x}\right)(y) +  \LfastDeri{x}\slowDer{j} \left(\SGfastsx{s}{x}\phi^{x}\right)(y) \Big]ds.
			\end{split}
		\end{equation}
		and
		\begin{equation}\label{eqn:invMeasDerDer}
			\begin{split}
				\slowDerDer{i}{j}\mu^x(\phi^x) &=  \int_{\mathbb{R}^d}\slowDerDer{i}{j}\phi^x(y)\mu^x(dy)\\
				&+ \int_{0}^{\infty} \int_{\mathbb{R}^d}\left[\LfastxDeriDerj \SGfastsx{s}{x}\phi^{x}(y) +\LfastDerj{x} \partial_{x_j} \SGfastsx{s}{x} \phi^{x}(y)  + \LfastDeri{x} \partial_{x_i} \SGfastsx{s}{x} \phi^{x}(y) \right]\mu^x(dy) ds.
			\end{split}
		\end{equation}
		Furthermore, for all $i \in \{1, \cdots, n \}$, there exists $C>0$ (which is independent of $x$ and $y$) such that the estimates \eqref{intAvgDerDerInv} and \eqref{eqn:intAvgDerDer} hold. 
	\end{proposition}	

		This reasoning and line of proof can be extended to any number of derivatives by an induction step. For an illustration of this in the bounded coefficients context see \cite[Section $4.5$]{pardoux2003}.

	\section{Averaging: Proof of Theorem \ref{mainthm} and of Theorem \ref{fullycoupledaveragingtheorem}}\label{subsec:detailedproof}
	
	This section contains the Proof of Theorem \ref{mainthm}. In particular, in Section \ref{subsec:heur} we give a heuristic argument which explains the approach we take; Section \ref{subsec:proofofmainthm} contains the proof itself.

	
	\subsection{Heuristics}\label{subsec:heur} 
	The structure of the proof of Theorem \ref{mainthm} is analogous to the one introduced in \cite[Section 4]{barr2020fast}. 
	Let us start by recalling the heuristic argument which  motivates the approach taken in the proof.  In particular, the heuristic argument below (which relies on the linearity of all the involved semigropus and associated generators and we don't repeat this every time) shows how Poisson equations and their derivatives  enter the picture. 
	
	To find an expression for the difference $\cP_t^{\epsilon} f - \bar \cP_t f$ in which we are interested, 
	we start by formally expanding the semigroup $\cP_t^{\epsilon}$ in powers of $\epsilon$:
	\begin{equation} \label{Meqn:asymExp}
		\cP_t^{\epsilon} f = f^0_t + \epsilon f^1_t+...
	\end{equation}
	From \eqref{eq:epsPDE}, we have 
	$$
	\partial_t\cP_t^\epsilon f - \mathfrak{L}_{\epsilon} (\cP_t^\epsilon f) = 0.
	$$
	Plugging the asymptotic expansion (\ref{Meqn:asymExp}) into the above, using the decomposition of the generator \eqref{decompLepsilon}   and  collecting terms with the same power of $\epsilon$ gives us the following set of equations
	\begin{align}
		\mathcal{O} \left( \epsilon^{-1} \right)&: \cL^x f_t^0 = 0 \label{Meqn:ft0},\\
		\mathcal{O} (\epsilon^0)&: \partial_t f_t^0 - \cL_S f_t^0 = \cL^x f_t^1 \label{Meqn:ft1}\,.
	\end{align}
	From the ergodicity of the process $\fastFixed_t$, defined by \eqref{fastF}, (and associated to the generator $\cL^x$ which, we recall, is a differential operator in the $y$ variable only),  equation  (\ref{Meqn:ft0}) implies that  $f_t^0(x,y)$ is constant in $y$,  \footnote{For every fixed $x$, the semigroup $\SGfastsxn{t}{x}$ has a unique invariant measure; hence, by \cite[Proposition 12.27 \& Theorem 12.31]{daintrostochasticanalysis}, the set of stationary points of $\SGfastsxn{t}{x}$ (i.e. the set of functions $f$ such that $\SGfastsxn{t}{x}f(y)=f(y)$ for almost every $y$) comprises only of constant functions. However if $\cL^xf=0$ then $\partial_t\SGfastsxn{t}{x}f(y) =\SGfastsxn{t}{x}\cL^x f(y) =0$, that is, $f$ is a stationary point of $\SGfastsxn{t}{x}$ and hence it must be constant as a function of $y$.} i.e.  $f_t^0(x,y) = f_t^0 (x)$. Now we fix $x$ and integrate (\ref{Meqn:ft1}) with respect to the  invariant measure $\mu^x(dy)$ to get
	\begin{equation*}
		\int \partial_t f_t^0(x) \mu^x(dy) - \int (\cL_S f_t^0)(x,y) \mu^x(dy) = \int (\cL^x f_t^1)(x,y) \mu^x(dy).
	\end{equation*}
	The right hand side of the above vanishes because $\mu^x$ is an invariant measure of the process $\fastFixed_t$,  and hence $\mu^x(\cL^x f)=0$ for all functions $f$ in the domain of the generator $\cL^x$. 
	Using the expressions for $\cL_S$ and $\bar\cL$ (namely, equations \eqref{avgDef}, \eqref{bar L} and \eqref{LFast and LSlow}-\eqref{eqn:gen}), from the above we obtain
	\begin{equation*}
		\partial_t f_t^0(x) - (\bar \cL f_t^0)(x) = 0.
	\end{equation*}
	The above equation has a unique solution (see e.g. \cite[Proposition 4.1.1]{lorenzi2006analytical}) which, by \eqref{eq:barPDE}, needs to coincide with the semigroup $\bar \cP_t$, i.e. $f_t^0 = \bar \cP_t f$.  Substituting this expression into the expansion  (\ref{Meqn:asymExp}) and into \eqref{Meqn:ft1} gives 
	\begin{align}
		&\cP_t^{\epsilon} f - \bar \cP_t f = \epsilon f^1_t +  \dots \label{ded1}\\
		&\cL^x f_t^1  = \left[\left( \widebar{\mathcal{L}}-\mathcal{L_S} \right)\widebar{\mathcal{P}}_t f\right](x) \,.\label{ded2}
	\end{align}
	The difference on the LHS of \eqref{ded1} is precisely the quantity which we want to study and  \eqref{ded1} indicates that in order to understand the behaviour of such a quantity we need to study  $f_t^1$. In turn, \eqref{ded2} can be seen as a Poisson equation in the unknown $f_t^1$ of the type studied in Section \ref{sec:psneqn}.  Indeed, let 
	\begin{align}
		v_t(x,y) &:= 	\left( \widebar{\mathcal{L}}-\mathcal{L_S} \right)\widebar{\mathcal{P}}_t f(x) \nonumber\\
		&=  \sum^n_{i=1}\left(\widebar{b}_i(x) - b_i(x,y)\right)\slowDer{i} \widebar{\mathcal{P}}_t f(x)\, ,\label{eqn:vt}
	\end{align}
	where in the above the terms involving $\Sigma$ vanish when taking the difference $\widebar{\mathcal{L}}-\mathcal{L_S}$, since the operators $\widebar{\mathcal{L}}$ and $\cL_S$ have the same diffusion coefficient (because of \eqref{eqn:bigDiffdef}). 
	Then  \eqref{ded2} becomes
	\begin{equation} \label{eqn:psnEqn}
		\mathcal{L}^x f_t^1(x,y) = v_t(x,y), \quad x \in \mathbb{R}^n, \, y \in \mathbb{R}^d \,.
	\end{equation}
	Recall that $\cL^x$ is a differential operator in the $y$ variable, with coefficients depending on the parameter $x$. Since $\bar{\cP}_tf(x)$ does not depend on $y$ and $\mu^x(b)=\bar{b}$ by definition,  the right hand side of \eqref{eqn:psnEqn} is of the same form as the right hand side of  \eqref{poisprobgeneral} and, by Assumption \ref{ass:polGrowth} (\ref{ass:polGrowthDrift}) and Proposition \ref{prop:avgderivativeest},  it is of the  growth required to apply Lemma \ref{lemma:probRep} (indeed for some $m_x,m_y >0$, $\Vnorm{b_i\slowDer{i} \widebar{\mathcal{P}}_t f(x)}{0,m_x,m_y} < \infty$ for every $i$) .
	Hence, from  Lemma \ref{lemma:probRep}, there exists a solution to \eqref{eqn:psnEqn}.

	Finally, we emphasize  that $f_t^1$ is not the unique solution to the Poisson equation \eqref{eqn:psnEqn} and  the proof of Theorem \ref{mainthm} does not rely on having a unique solution to such an equation. However,  due to such non-uniqueness,  $f_t^1$ may not represent the entire order $\epsilon$ term in the expansion \eqref{Meqn:asymExp}.

	\subsection{Proof of Theorem \ref{mainthm}}\label{subsec:proofofmainthm}
	
	\begin{proof}[Proof of Theorem \ref{mainthm}] Let $f_t^1$ be the function defined in \eqref{eqn:psnEqn}.
		From Lemma \ref{lemma:probRep}, this function is well defined and is a classical solution of the Poisson equation \eqref{eqn:psnEqn}.  Motivated by the heuristics presented in Section \ref{subsec:heur}, we define
		\begin{equation} \label{Meqn:r}
			r_t^\epsilon(x,y) \coloneqq (\cP_t^\epsilon f)(x,y) - (\bar \cP_t f)(x) -\epsilon f_t^1(x,y).
		\end{equation}
		Here we recall $f$ is any function in $C_b^2(\R^n)$.
		With this definition, we have
		\begin{equation} \label{eqn:boundterm}
			|\left(\cP_t^\epsilon f\right) (x,y) - \bar \cP_t f(x)| = |\epsilon f_t^1(x,y) + r_t^\epsilon (x,y)|,
		\end{equation}
		so that we need to study the terms $f_t^1$ and $r_t^\epsilon$. We first turn our attention to the latter. By differentiating \eqref{Meqn:r} with respect to time we have the following:
		\begin{align*}
			\partial_t r_t^\epsilon(x,y) &= \partial_t(\cP_t^\epsilon f)(x,y) - \partial_t (\bar \cP_t f)(x) -\epsilon \partial_t f_t^1(x,y) \\
			& \stackrel{\eqref{eq:epsPDE}}{=} \mathfrak{L}_{\epsilon} \cP_t^\epsilon f(x,y) -   \partial_t (\bar \cP_t f)(x) -\epsilon \partial_t f_t^1(x,y) \,.
		\end{align*}
		Rearranging (\ref{Meqn:r}), we have   $(\cP_t^\epsilon f)(x,y) = r_t^\epsilon(x,y) + (\bar \cP_t f)(x) +\epsilon f_t^1(x,y) $;
		using this fact,  the decomposition $\mathfrak{L}_{\epsilon} = \frac{1}{\epsilon}\cL^x + \cL_S$, and remembering that $(\bar \cP_t f)(x)$ does not depend on the variable $y$ (so that $\mathfrak{L}_{\epsilon} \bar \cP_t = \cL_S \bar \cP_t $), we obtain
		\begin{align*}
			\partial_t r_t^\epsilon(x,y) &= \mathfrak{L}_{\epsilon} r_t^\epsilon(x,y) + \mathfrak{L}_{\epsilon} \bar \cP_t f(x) + \epsilon\mathfrak{L}_{\epsilon} f_t^1 (x,y) - \partial_t \bar \cP_t f(x) - \epsilon \partial_t f_t^1(x,y) \\
			&= \mathfrak{L}_{\epsilon} r_t^\epsilon(x,y) + \cL_S \bar \cP_t f(x) +  \epsilon\mathfrak{L}_{\epsilon} f_t^1 (x,y) -\bar \cL \bar \cP_t f(x) - \epsilon \partial_t f_t^1(x,y),
		\end{align*}
		where for the penultimate addend we have used \eqref{eq:barPDE}.
		We can now use the fact that $f_t^1$ satisfies the Poisson equation \eqref{eqn:psnEqn}, that is,  $(\cL_S-\bar \cL)\bar \cP_t f = -\cL^x f_t^1$, and again \eqref{decompLepsilon} to conclude
		\begin{align*}
			\partial_t r_t^\epsilon(x,y) &= \mathfrak{L}_{\epsilon} r_t^\epsilon(x,y) + \epsilon\left( \cL_S f_t^1 (x,y) - \partial_t f_t^1 (x,y) \right).
		\end{align*}
		The variation of constants formula then  gives 
		\begin{equation} \label{rDiff}
			r_t^\epsilon (x,y) = \cP_t^\epsilon r_0^\epsilon (x,y) + \epsilon \int_0^t \cP_{t-s}^\epsilon (\cL_S f_s^1 - \partial_s f_s^1) (x,y) ds.
		\end{equation}
		Substituting the above expression for the remainder $r_t^\epsilon (x,y)$ into \eqref{eqn:boundterm}, we obtain
		\begin{equation} \label{eqn:error}
			|\cP_t^\epsilon (x,y) - \bar \cP_t f(x)| = |\epsilon f_t^1(x,y) + \cP_t^\epsilon r_0^\epsilon (x,y) + \epsilon \int_0^t \cP_{t-s}^\epsilon (\cL_S f_s^1 - \partial_s f_s^1) (x,y) ds|.
		\end{equation}
		Hence, to prove the statement it suffices to prove the following three estimates: 
		\begin{align}\label{step1bound}
			&|\cP_{t-s}^\epsilon (\cL_S - \partial_s)f_s^1| \leq C\lVert f\rVert_{C_b^2} e^{-\omega s} (1+ |y|^{M^{g, b}_{y}+\pow{g}{y}}+|x|^{4\pow{b}{x}}),\\
			\label{step2bound}
			&|f_t^1(x,y)| \leq C\lVert f\rVert_{C_b^2}e^{-\omega t}(1+ |y|^{m^{b}_{y}}+|x|^{\pow{b}{x}})
			,\\
			\label{step3bound}
			&|\cP_t^\epsilon r_0^\epsilon (x,y)| \leq \epsilon C\lVert f\rVert_{C_b^2}(1+ |y|^{m^{b}_{y}}+|x|^{\pow{b}{x}}) \, ,
		\end{align}
		where in the above and throughout we use $C$ to denote a positive constant that is dependent on $n$ and $f$, but is independent of $x$ and $y$, and may change line by line, and $\omega>0$ is a positive constant. We prove the above three bounds in turn.  
		To prove \eqref{step1bound} it suffices to show the following:
		\begin{equation}\label{blabla}
			\left| (\mathcal{L}_S - \partial_s)f_s^1(x,y) \right| \leq C\lVert f\rVert_{C_b^2} e^{-\omega s} (1+ |y|^{M^{g, b}_{y}+\pow{g}{y}}+|x|^{4\pow{b}{x}}).
		\end{equation}
		Indeed, once the above has been shown, we can apply the 
		semigroup $\{\cP_t^\epsilon \}_{t\geq 0}$ to both sides of the above and bound the right hand side using  Lemma \ref{prop:fullMombound}, to obtain
		\begin{equation*}
			\cP_{t-s}^\epsilon\left( C e^{-\omega s}\lVert f\rVert_{C_b^2}(1+ |y|^{M^{g, b}_{y}+\pow{g}{y}}+|x|^{4\pow{b}{x}})\right) \leq C e^{-\omega s}\lVert f\rVert_{C_b^2}(1+ |y|^{M^{g, b}_{y}+\pow{g}{y}}+|x|^{4\pow{b}{x}}) \, ,
		\end{equation*} which implies the desired result. The proof of \eqref{blabla} heavily relies on our results on Poisson equations, i.e. on the results of Section \ref{sec:psneqn}, and it is deferred to Proposition \ref{psnEqDer}, see  (\ref{eqn:Ls-ds}).
		Similarly, 	the proof of \eqref{step2bound} is deferred to Proposition \ref{psnEqDer}, see \eqref{eqn:ft1Bound} below. To obtain \eqref{step3bound} note that from (\ref{Meqn:r}) we have
		\begin{align*}
			r_0^\epsilon (x,y) = f(x) - f(x) - \epsilon f_0^1 = - \epsilon f_0^1.
		\end{align*}
		
		Now we use \eqref{step2bound} at $t=0$ to conclude \eqref{step3bound}. The proof is hence finished.
	\end{proof}
	
	\begin{Note}\label{Rem:poisson for averaging} {  Using linearity of the Poisson equation we can write
			\begin{equation}\label{eqn:f_t^1}
				f_t^1(x,y)  = -\sum^n_{i=1} u_{ b_i}(x,y)\slowDer{i} \widebar{\mathcal{P}}_t f(x),
			\end{equation}
			where  $u_{ b_i}(x,y)$ is the solution to \eqref{eqn:psnEqn} with the right hand side equal to $b_i$ (compatibly with the notation set at the beginning of Section \ref{sec:mainresults}).}
		With the above formula in mind, we can explain how the results of Section \ref{sec:psneqn} on Poisson equations will be used in the averaging proof. The expression for  $f_t^1$ in the above involves the product of solutions to the Poisson Equation \eqref{eqn:psnEqn} and derivatives of the semigroup $\bar{\cP}_t$ with respect to  $x$. Moreover, in order to obtain \eqref{step1bound} we will need to apply $\cL_S$ (defined in \eqref{LFast and LSlow}) to $f_t^1$;  this in turn requires bounds on the $x$ derivative of  $u_{ b_i}(x,y)$  and $x$ derivatives of the averaged semigroup which decay fast enough in time (fast enough to be integrable over $(0, \infty)$). The derivative estimates for the averaged semigroup are obtained in Section \ref{sec:SGder}. From what we have said so far, it might seem that, when applying $\cL_S$ (which is a second order differential operator) to $f_t^1$ one needs to deal with third order derivatives of $\bar{\cP}_t$. However this is not the case as in \eqref{step1bound} one needs to consider the difference $\cL_S-\pa_t$ applied to $f_t^1$. When taking this difference the third order derivatives cancel, see proof of Proposition \ref{psnEqDer}. 
		
		With regards to the $x$ derivatives of the Poisson equation \eqref{eqn:psnEqn}, note that $x$ appears only as a parameter. Derivatives with respect to a parameter of a Poisson equation were studied in Section \ref{sec:psneqn}. In Section \ref{sec:appPsnEqn}, we use Proposition \ref{lemma:IntAvgDerRepresentation} and Proposition \ref{lemma:IntAvgDerDer} to get the required bounds. 
		

		
	\end{Note}
	
	\begin{lemma}\label{prop:fullMombound}
		Let Assumption \ref{ass:polGrowth}-\ref{DriftAssumpS} hold and
		let $V(x,y)=|x|^{4\pow{b}{x}} + |y|^{k}$;  then, for all $k \geq 0$ and $\epsilon \leq 1$, we have 
		\begin{equation}\label{momBoundFull}
			(\cP_t^\epsilon V)(x,y) \leq e^{-\tilde{r}' t}|x|^{4\pow{b}{x}} +e^{-r'_k t}|y|^{k} +\frac{C'_k}{r'_k} +\frac{\tilde{C}'}{\tilde{r}'},
		\end{equation}
		where the constants $C'_k,r'_k$ are defined in Lemma \ref{prop:momentBounds}, and 
		\begin{align*}
			\tilde{C}'&=2\tilde{C}\left(\frac{\tilde{C}(2\pow{b}{x}-1)}{\pow{b}{x}\tilde{r}}\right)^{2\pow{b}{x}-1},\\
			\tilde{r}'&=2\pow{b}{x}\tilde{r} \,,
		\end{align*}
		with $\tilde{C}, \tilde{r}$ being from Assumption \ref{DriftAssumpS}.
	\end{lemma}
	\begin{proof}
		The proof can be found in Appendix \ref{app:proofsaveraging}.
	\end{proof}

	\subsection{Proof of Theorem \ref{fullycoupledaveragingtheorem}}\label{subsec:fullycoupledsketch}
 We recall that this section is devoted to the study of the fully coupled regime. 
	\begin{proof}[Proof of Theorem \ref{fullycoupledaveragingtheorem}]
		This proof follows that of Theorem \ref{mainthm}, so we point out the places in which it differs, and sketch the rest. The Poisson equation that $f_t^1$ solves in this case is slightly different. In particular, the right hand side of \eqref{eqn:psnEqn} will contain two derivatives of the averaged semigroup, instead of just one. This explains why Assumption \ref{ass:averagederivativeestimate4} contains four derivatives, as opposed to the two required in the proof of Theorem \ref{mainthm}. Indeed, instead of \eqref{eqn:vt}-\eqref{eqn:psnEqn}, $f_t^1$ is the solution of $\cL^x f_t^1=v_t$ with $\mu^x(f_t^1)=0$ where
		\begin{align*}
			v_t(x,y) &= 	\left( \widebar{\mathcal{L}}-\mathcal{L_S} \right)\widebar{\mathcal{P}}_t f(x) \nonumber\\
			&=  \sum^n_{i=1}\left(\widebar{b}_i(x) - b_i(x,y)\right)\slowDer{i} \widebar{\mathcal{P}}_t f(x) + \sum^n_{i,j=1}\left(\widebar{\Sigma}_{ij}(x) - \Sigma_{ij}(x,y)\right)\slowDerDer{i}{j} \widebar{\mathcal{P}}_t f(x) \,.
		\end{align*}
		Again, from  Lemma \ref{lemma:probRep} there exists a solution to \eqref{eqn:psnEqn} (we observe we can apply Lemma \ref{lemma:probRep}  to the fully coupled case as such a lemma  only relies on Proposition \ref{lemma:IntAvg},  which holds also in this setting as explained in Note \ref{note:coupledlemmas}). Using the linearity of the Poisson equation we can write
		\begin{equation}\label{eqn:f_t^1fullycoupled}
			f_t^1(x,y) = -\sum^n_{i=1} u_{ b_i}(x,y)\slowDer{i} \widebar{\mathcal{P}}_t f(x) -\sum^n_{i,j=1} u_{ \Sigma_{ij}}(x,y)\slowDerDer{i}{j} \widebar{\mathcal{P}}_t f(x),
		\end{equation}
		where we recall the notation $u_{ b_i}(x,y)$ as the solution to \eqref{eqn:psnEqn} with the right hand side equal to $b_i$, and analogously for $u_{ \Sigma_{ij}}(x,y)$. Following the proof of Theorem \ref{mainthm} it is sufficient to prove that $f_t^1$ satisfies the estimates \eqref{step1bound}-\eqref{step3bound}. Since \eqref{step1bound} is the most challenging to prove we will give more details for this estimate. Observe that \eqref{step2bound} follows directly from \eqref{eq:intAvgDercoupled} with $\phi=b_i$ or $\phi=\Sigma_{ij}$ and Assumption \ref{ass:averagederivativeestimate4}. The bound \eqref{step3bound} follows from \eqref{step2bound} setting $t=0$. Let us show that \eqref{step1bound} holds, which follows once we have that \eqref{blabla} holds.
		Differentiating \eqref{eqn:f_t^1fullycoupled} with respect to $t$ we get
		\begin{equation*}
		\begin{split}
				\partial_t  f_t^1(x,y) =  -\sum^n_{i=1} u_{b_i}(x,y) \partial_{x_i}\widebar{\cL} \widebar{\mathcal{P}}_t f(x) -\sum^n_{i,j=1} u_{ \Sigma_{ij}}(x,y)\slowDerDer{i}{j}\widebar{\cL} \widebar{\mathcal{P}}_t f(x).
			\end{split}
		\end{equation*}
		Further, using \eqref{bar L} and \eqref{LFast and LSlow} we can write
		\begin{align*}
			&(\mathcal{L}_S - \partial_s)f_s^1(x,y) = \\
			&-\sum^n_{i,j=1}b_i(x,y)\partial_{x_i}u_{ b_j}(x,y)\slowDer{j} \widebar{\mathcal{P}}_s f(x)  - \sum^n_{i,j=1}\left(b_i(x,y)-\bar{b}_i(x)\right)u_{ b_j}(x,y)\slowDerDer{i}{j} \widebar{\mathcal{P}}_s f(x)\\
			&-\sum^{n}_{i,j,k = 1}\Sigma_{ij}(x)\slowDerDer{i}{j}u_{b_k}(x,y) \slowDer{k} \widebar{\mathcal{P}}_s f(x) -2\sum^{n}_{i,j,k = 1}\Sigma_{ij}(x)\slowDer{i}u_{ b_k}(x,y) \slowDerDer{j}{k} \widebar{\mathcal{P}}_s f(x) \\
			&+ \sum^n_{i,j=1}\slowDer{i}\bar{b}_j(x)u_{b_i}(x,y)\slowDer{j} \widebar{\mathcal{P}}_s f(x) +\sum^n_{i,j,k=1}\slowDer{i}\Sigma_{jk}(x)u_{ b_i}(x,y)\slowDerDer{j}{k} \widebar{\mathcal{P}}_s f(x) \\
			&- \sum^n_{i,j,k=1}b_i(x,y)\partial_{x_i}u_{ \Sigma_{jk}}(x,y)\slowDerDer{j}{k} \widebar{\mathcal{P}}_s f(x)  - \sum^n_{i,j,k=1}b_i(x,y)u_{ \Sigma_{jk}}(x,y)\slowDerDerDer{i}{j}{k} \widebar{\mathcal{P}}_s f(x)\\
			&-\sum^{n}_{i,j,k,l = 1}\Sigma_{ij}(x)\slowDerDer{i}{j}u_{\Sigma_{kl}}(x,y) \slowDerDer{k}{l} \widebar{\mathcal{P}}_s f(x) -2\sum^{n}_{i,j,k,l = 1}\Sigma_{ij}(x)\slowDer{i}u_{ \Sigma_{kl}}(x,y) \slowDerDerDer{j}{k}{l} \widebar{\mathcal{P}}_s f(x) \\
			& - \sum^n_{i,j,k,l=1}\left(\Sigma_{ij}(x,y)-\bar{\Sigma}_{ij}(x)\right)u_{ \Sigma_{kl}}(x,y)\slowDerDerDerDer{i}{j}{k}{l} \widebar{\mathcal{P}}_s f(x)  \\
			&+ \sum^n_{i,j,k=1}\slowDerDer{i}{j}\bar{b}_{k}(x)u_{\Sigma_{ij}}(x,y)\slowDer{k} \widebar{\mathcal{P}}_s f(x) +2\sum^n_{i,j,k=1}\slowDer{i}\bar{b}_{k}(x)u_{\Sigma_{ij}}(x,y)\slowDerDer{j}{k} \widebar{\mathcal{P}}_s f(x) \\
			&+ \sum^n_{i,j,k,l=1}\slowDerDer{i}{j}\bar{\Sigma}_{kl}(x)u_{\Sigma_{ij}}(x,y)\slowDerDer{k}{l} \widebar{\mathcal{P}}_s f(x) +2\sum^n_{i,j,k,l=1}\slowDer{i}\bar{\Sigma}_{kl}(x)u_{\Sigma_{ij}}(x,y)\slowDerDerDer{j}{k}{l} \widebar{\mathcal{P}}_s f(x) \\
			&+ \sum^n_{i,j,k=1}\bar{b}_{k}(x)u_{\Sigma_{ij}}(x,y)\slowDerDerDer{i}{j}{k} \widebar{\mathcal{P}}_s f(x)-\sum^n_{i,j,k=1}\left(\Sigma_{ij}(x,y)-\bar{\Sigma}_{ij}(x)\right)u_{ b_{k}}(x,y)\slowDerDerDer{i}{j}{k} \widebar{\mathcal{P}}_s f(x).
		\end{align*}
		We use Assumption \ref{ass:averagederivativeestimate4}, along with \eqref{eqn:IntAvg2fullycoupled}-\eqref{eq:intAvgDercoupled} (with $\phi = b_i$ and $\phi = \Sigma_{ij}$) and Assumption \ref{ass:polGrowth} to get
		\begin{equation*}
			\left| (\mathcal{L}_S - \partial_s)f_s^1(x,y) \right| \leq Ce^{-\omega s}\lVert f\rVert_{C_b^4} (1+ |y|^{M_{y}}+|x|^{M_x}).
		\end{equation*}
		Hence, as in the proof of Theorem \ref{mainthm}, we obtain \eqref{step1bound}-\eqref{step3bound}, and conclude the proof.
	\end{proof}

	\section{Strong ergodicity of the averaged semigroup and  application of Section \ref{sec:psneqn} to Averaging}\label{sec:boundsonft1}
	In this section we prove the remaining results needed to prove Theorem \ref{mainthm}. In Section \ref{sec:SGder} we address the decay in time of the space derivative of the averaged semigroup $\bar{\cP}_t$. In Section \ref{sec:appPsnEqn}, we will use the results of Section \ref{sec:psneqn} and Section \ref{sec:SGder} to prove bounds on $f_t^1$ (solution to the Poisson problem \eqref{eqn:psnEqn}) and its $x$ derivatives. This is the section where we most heavily make use of our results on Poisson equations within the proof of the averaging result, Theorem \ref{mainthm}. The role of these results within the proof of Theorem \ref{mainthm} has been explained in Note  \ref{Rem:poisson for averaging}.
	
	
	\subsection{Derivative estimates for the Averaged Semigroup and Examples}\label{sec:SGder}
	
	We now prove that the Assumption \ref{ass:avgSGcond} made on the coefficients implies the required derivative estimates. In order to prove the derivative estimates we make use of \cite[Theorem 7.1.5]{lorenzi2006analytical}, with an added mollification argument. Indeed, \cite[Theorem 7.1.5]{lorenzi2006analytical} requires in particular that the averaged drift $\bar b$ (see \eqref{avgDef}) is in $C^{2+\nu}(\R^n)$, meaning that the second derivative is H\"older continuous. From our results, see Proposition \ref{lemma:IntAvgDerDer}, we only have that $\bar b$ is twice continuously differentiable. Hence, to prove Theorem \ref{thm:lorenziderest} we follow a mollification argument (similar to that in \cite[Theorem 6.1.9]{lorenzi2006analytical}) which relaxes \cite[Theorem 7.1.5]{lorenzi2006analytical} (in our setting) to twice continuously differentiable coefficients.
	
	\begin{theorem}\label{thm:lorenziderest}
		Let Assumptions \ref{ass:polGrowth} and \ref{UniformEllip} hold. Assume that there exists a polynomial $R\colon \R^n \rightarrow \R$ and a constant $K_{\Sigma}>0$ such that for all $x \in \R^n$ \begin{equation}\label{eqn:lorenziderestAsumpHess}
			\lvert\partial_x^\gamma \overline{b}(x)\rvert\leq R(x) \text{ for all multi-indices $\gamma$ of length $n$ such that  $\lvert \gamma\rvert_*=2$,}
		\end{equation} and $\lvert\partial_x^\gamma \Sigma(x) \rvert\leq K_\Sigma$ for any $\gamma$ such that $\lvert \gamma\rvert_*=1,2.$ Moreover there exists a polynomial $\tilde{R}\colon \R^n \rightarrow \R$ and a constants $L>0$ such that
		\begin{align}\label{eqn:lorenziderestAsump1}
			&\sum_{i,j=1}^n \partial_{x_i}\overline{b}_j(x) \xi_i\xi_j\leq -\tilde{R}(x)\lvert \xi\rvert^2,\\\label{eqn:lorenziderestAsump2}
			&-C_1 \coloneqq \sup_{x\in \R^n} \left(-\tilde{R}(x)+LR(x)+\frac{K_{\Sigma}^2n^3}{4\lambda_{-}}\right) <0.
		\end{align}
		Then, there exist positive constants $M,\omega>0$ such that for all $f\in C_b^2(\R^n)$
		\begin{equation}\label{eqn:lorenziBertoldi1}
			\sum_{1\leq \lvert\gamma\rvert_* \leq 2}\lVert \partial_x^\gamma \bar{\cP}_tf(x)\rVert_{\infty}^2 \leq Me^{-\omega t} \sum_{1\leq \lvert\gamma\rvert_* \leq 2}\lVert \partial_x^\gamma f(x) \rVert_{\infty}^2.
		\end{equation}
	\end{theorem}
	
	\begin{proof}[Proof of Theorem \ref{thm:lorenziderest}]
		In order to use \cite[Theorem 7.1.5]{lorenzi2006analytical}, we need  $\bar b \in C^{2+\nu}$ for some $\nu >0$. Since this is not necessarily true under our assumptions, we apply a classical procedure and first  smooth the coefficient $\bar b$, apply \cite[Theorem 7.1.5]{lorenzi2006analytical} to such regularised coefficients and then obtain estimates independent of the regularization parameter. We follow the smoothing procedure outlined in \cite[Theorem 6.1.9]{lorenzi2006analytical}. To this end,  for any $\delta >0$ let $\varphi^\delta(x) = \delta^{-n}\varphi(x/\delta)$, where $\varphi \in C^{\infty}(\R^n)$ is any non-negative even function compactly supported in the ball centred at $0$ with radius $1$ and  such that $\int_{\R^n}\varphi(x) dx =1$. Denote by $\bar b^\delta$ the convolution between $\bar b$ and $\varphi^\delta$, 
		i.e.
		$$\bar b^{\delta}:= (\bar b * \varphi^\delta)(x) \coloneqq \int_{\R^n}\bar b(\bar x)\varphi^\delta(x-\bar x) d\bar x.$$ We let $\bar \cL^\delta$ be defined as $\bar \cL$ (see \eqref{bar L}) with $\bar b$ replaced by $\bar b^\delta$ for all $i =1, \ldots, n$. We will first prove the estimate \eqref{eqn:lorenziBertoldi1} for the smoothed semigroup $\bar \cP^{\delta}_t$, corresponding to the operator $\bar \cL^{\delta}$, and then argue that we can take $\delta \rightarrow 0$.		
		To prove the required derivative estimate, we can use \cite[Theorem 7.1.5]{lorenzi2006analytical} (where $p=2,k=2$ in their notation). It is important in order to be able to take $\delta \rightarrow 0$ that the constants in the hypotheses of \cite[Theorem 7.1.5]{lorenzi2006analytical} should be uniform in $\delta$. Let us first observe that the hypotheses of \cite[Theorem 7.1.5]{lorenzi2006analytical} are given by \cite[Hypothesis 6.1.1 (i)-(iii), Hypothesis 7.1.3 (ii-2)]{lorenzi2006analytical}. For the readers convenience, we gather the assumptions that involve $\bar b^\delta$ explicitly in Assumption \ref{ass:lorenziBertoldi}, modified to match our notation. Then we will explain how the remaining assumptions are satisfied (the remaining ones are straightforward), before verifying Assumption \ref{ass:lorenziBertoldi} in detail.
		
		\begin{assumption}\label{ass:lorenziBertoldi} Here we collect \cite[Hypothesis 6.1.1 (iii), Equation (6.1.2)]{lorenzi2006analytical} and \cite[Hypothesis 7.1.3 (ii-2)]{lorenzi2006analytical}, noting that we can take the function $\kappa=\kappa(x)$ in \cite{lorenzi2006analytical} to be $\kappa(x) = \lambda_{-}$. 
			\begin{enumerate}
				\item \label{ass:lorenziBertoldi1} There exists a constant $C>0$ such that \begin{equation}\label{eqn:driftCondLorenzi}
					\sum_{i=1}^n \bar b^\delta_i(x) x_i \leq C(1+|x|^2), \quad \mbox{for any $x \in \R^n$}\,.
				\end{equation}
				\item \label{ass:lorenziBertoldi2} $\Sigma_{ij}, \bar{b}^\delta_{j} \in C^{2+\nu}$ for some $\nu \in (0,1)$ and any $i,j = 1, \ldots, n$ and there exist $\pow{\Sigma}{}$, a constant $K_{\Sigma}>0$ and a positive function $r \colon \R^n \rightarrow \R$ such that 
				\begin{equation}\label{eqn:lorenziderestHess1}
					\sup_{|\gamma|_*=2}\lvert \partial^\gamma_x \bar{b}^\delta_i(x)\rvert \leq r(x),
				\end{equation} and $\sup_{1\leq|\gamma|\leq 2}\lvert \partial^\gamma_x \Sigma_{ij}(x) \rvert \leq K_\Sigma $ for any $x \in \R^n$ and $i,j =1,\cdots,n$. Moreover, there exist a function $\ell:\R^n \rightarrow \R$ such that 
				\begin{align}\label{eqn:lorenziderestAsumpdelta}
					\sum_{i,j=1}^n \partial_{x_i}\bar b^\delta_j(x) \xi_i\xi_j&\leq \ell(x)\lvert \xi\rvert^2 \quad \mbox{for all $\xi \in \R^n$}\,, 
				\end{align}
				and  a constant $\tilde{L}>0$ such that
				\begin{equation}\label{eqn:LBdriftCondAss}
					0 > -C_2 \coloneqq \sup_{x \in \R^n} \left( \ell(x) + \tilde{L}r(x) + \frac{K_\Sigma^2 n^3}{4\lambda_{-}} \right).
				\end{equation}
				Finally, there exists a constant $K_1\in \R$ such that
				\begin{equation}\label{eqn:lorenziBertoldiDiff}
					\sum_{i,j,h,k}\partial_{x_h,x_k}\Sigma_{ij}(x)M_{ij}M_{hk} \leq K_1\sum_{h,k =1}^n M_{hk}^2
				\end{equation}
				for any symmetric matrix $M = (M_{hk})_{h,k = 1}^n$ and any $x \in \R^n$.
			\end{enumerate}
		\end{assumption}
		Let us now explain why Assumption \ref{ass:polGrowth}, \ref{UniformEllip} and \ref{ass:lorenziBertoldi} implies that we can apply \cite[Theorem 7.1.5]{lorenzi2006analytical}. 
		Indeed, \cite[Hypothesis 6.1.1 (i)]{lorenzi2006analytical} follows from Assumption \ref{UniformEllip}. \cite[Hypothesis 6.1.1 (ii)]{lorenzi2006analytical} follows from \eqref{eqn:driftCondLorenzi}, with $\varphi = |x|^2$. \cite[Hypothesis 6.1.1 (iii)]{lorenzi2006analytical} follows since $\Sigma$ is bounded (Assumption \ref{ass:polGrowth}) and \eqref{eqn:driftCondLorenzi}. It is now left to show that Assumption \ref{ass:lorenziBertoldi} holds under the assumptions of Theorem \ref{thm:lorenziderest}.

		We start with Assumption \ref{ass:lorenziBertoldi} (\ref{ass:lorenziBertoldi2}). Notice that for any polynomial $R(x) \coloneqq c(1+|x|^m)$ with constants $c,m>0$, we have that $(R*\varphi^{\delta})(x) \leq R(x) + c\delta^m$. Indeed, using the fact  that $\varphi^{\delta}$ is compactly supported on the ball centred at $0$ with radius $\delta$ and that $\varphi^{\delta}$ integrates to $1$, we can write
		\begin{align}\label{eqn:convPoly}\begin{split}
				(R*\varphi^{\delta})(x) &= \int_{\R^n}c(1+|x-\bar x|^m) \varphi^\delta(\bar x) d\bar x \\
				&\leq c(1+|x|^m+\delta^m)\int_{\R^n}  \varphi^\delta(\bar x) d\bar x \\
				&\leq c(1+|x|^m)+c\delta^m.\end{split}
		\end{align}
		By linearity of the convolution we can generalise this to  polynomials with non-negative coefficients to find $(R*\varphi^{\delta})(x)\leq R(x)+ R(\delta\cdot e_1)$, where $e_1$ is the first element of the standard basis of $\R^n$.
		Convolving both sides of \eqref{eqn:lorenziderestAsumpHess} and bounding the convolution of the polynomial on the right hand side using the above, we obtain \eqref{eqn:lorenziderestHess1} with $r(x) = c_1(1+|x|^{m_1})+c_1\delta_0^{m_1}$; similarly \eqref{eqn:lorenziderestAsump1} implies that \eqref{eqn:lorenziderestAsumpdelta} holds with $\ell(x) =-c_2(1+|x|^{m_2})+c_2\delta_0^{m_2}$, for all $\delta \leq \delta_0$, with $\delta_0$ to be chosen later.
		
		
		Also,
		\begin{equation*} \sup_{1\leq|\gamma|_*\leq 2}\lvert \partial^\gamma_x \Sigma_{ij}(x) \rvert \leq K_\Sigma,\end{equation*}
		immediately by Assumption \ref{ass:polGrowth}.
		Using \eqref{eqn:lorenziderestAsump2} and \eqref{eqn:lorenziderestAsumpdelta} (as well as our definitions of $r(x),\ell(x)$ above), we have that for all $\bar L < L$,
		\begin{equation*}
			-C_2 \leq -C_1+  \bar L c_1\delta_0^{m_1} + c_2 \delta_0^{m_2}.
		\end{equation*}
		We pick $\delta_0 = \min\{ 1, \frac{C_1}{2c_2}\}$ and $\bar L = \min\{L, \frac{c_2}{2c_1}\}$ to obtain
		\begin{equation*}
			-C_2 \leq -C_1 + \frac{C_1}{4}+ \frac{C_1}{2} = -\frac{C_1}{4} <0,
		\end{equation*}
		so that \eqref{eqn:lorenziderestHess1}-\eqref{eqn:LBdriftCondAss} hold, independent of $\delta$, for $\delta < \delta_0$.
		By Assumption \ref{ass:polGrowth}, \eqref{eqn:lorenziBertoldiDiff} holds uniformly in $\delta$ because we are not smoothing the diffusion coefficient $\Sigma$, and hence Assumption \ref{ass:lorenziBertoldi} (\ref{ass:lorenziBertoldi2}) holds with constants that are independent of $\delta$. Now we show that in our setting, Assumption \ref{ass:lorenziBertoldi} (\ref{ass:lorenziBertoldi2}) implies Assumption \ref{ass:lorenziBertoldi} (\ref{ass:lorenziBertoldi1}). First, \eqref{eqn:LBdriftCondAss} implies that $\ell(x) < -C_2<0$; using this and \eqref{eqn:lorenziderestAsumpdelta} with $(x, \xi) = (zx,x)$ for $z \in [0,1]$, we have
		\begin{equation}\label{eqn:tobeintegrated}
			\begin{split}
				\frac{d}{dz}\sum_{j=1}^n \bar b^{\delta}_j(z x) x_j&\leq -C_2\lvert x\rvert^2.
			\end{split}
		\end{equation}
		Integrating \eqref{eqn:tobeintegrated} over $z \in [0,1]$, we have
		\begin{equation*}
			\begin{split}
				\sum_{j=1}^n \bar b^{\delta}_j(x) x_j &\leq \sum_{j=1}^n \bar b^{\delta}_j(0) x_j -C_2\lvert x\rvert^2 \\
				&\leq \lvert \bar b^{\delta}_j(0)\rvert\lvert x\rvert -C_2\lvert x\rvert^2 \\
				&\leq C|x| - C_2|x|^2 ,
			\end{split}
		\end{equation*}
		for some $C>0$ and all $0<\delta <1$,
		where we use Assumption \ref{ass:polGrowth} with the bound on a convolved polynomial, as in \eqref{eqn:convPoly}.
		This gives Assumption \ref{ass:lorenziBertoldi} (\ref{ass:lorenziBertoldi1}).
		
		Now, since the conditions do hold for each $\delta$, we obtain \eqref{eqn:lorenziBertoldi1} for the semigroup corresponding to the mollified coefficients. That is, there exist positive constants  $M,\omega, \delta_0>0$ (independent of $\delta$) such that for all $f\in C_b^2(\R^n)$ and $\delta<\delta_0$
		\begin{equation}\label{eqn:lorenziBertoldi1epsilon}
			\sum_{1\leq \lvert\gamma\rvert \leq 2}\lVert \partial_x^\gamma \bar{\cP}^{\delta}_tf(x)\rVert_{\infty}^2 \leq Me^{-\omega t} \sum_{1\leq \lvert\gamma\rvert \leq 2}\lVert \partial_x^\gamma f(x) \rVert_{\infty}.
		\end{equation}
		We observe that by Arzela-Ascoli, we can find a sequence $\delta_k \rightarrow 0$ such that $(\bar{\cP}^{\delta_k}_tf)(x)$ and its first two derivatives converge locally uniformly (in $t$ and $x$) to some $Q_t f(x)$ (and its first two derivatives, respectively) satisfying
		\begin{equation*}
			\sum_{1\leq \lvert\gamma\rvert \leq 2}\lVert \partial_x^\gamma Q_tf(x)\rVert_{\infty}^2 \leq Me^{-\omega t} \sum_{1\leq \lvert\gamma\rvert \leq 2}\lVert \partial_x^\gamma f(x) \rVert_{\infty},
		\end{equation*} with the same positive constants $M, \omega>0$ as \eqref{eqn:lorenziBertoldi1epsilon}. Indeed, we can apply Arzela-Ascoli to the set $\{\bar{\cP}^{\delta}_tf(x), \slowDer{i} \bar{\cP}^{\delta}_tf(x), \slowDerDer{i}{j}\bar{\cP}^{\delta}_tf(x)\}_{\delta}$ for all $i,j=1,\ldots n$, noting that this set is uniformly bounded (by \eqref{eqn:lorenziBertoldi1epsilon}) and equicontinuous, (by Schauder estimates; see \cite[Theorem C.1.4]{lorenzi2006analytical}). It remains to show that $Q_t f = \bar \cP_t f$. We leave the details to \cite[Theorem 6.1.9]{lorenzi2006analytical}, which uses that $\bar b^\delta$ converges locally uniformly to $\bar b$ to show $Q_t f$ solves the PDE \eqref{eq:barPDE}, and hence $Q_t f = \bar \cP_t f$.	
	\end{proof}
	Under Assumptions \ref{ass:polGrowth}-\ref{ass:avgSGcond} the conditions of Theorem \ref{thm:lorenziderest} are satisfied. Hence we get the following result.

	\begin{proposition}\label{prop:avgderivativeest}
		Let $\{\bar{\cP}_{t} \}_{t\geq 0}$ denote the semigroup described by \eqref{eqn:avgedSGdef}, and assume that Assumptions \ref{ass:polGrowth}-\ref{ass:avgSGcond} hold.
		Then there exists constants $\tilde{K},C>0$ such that for any $\psi\in C_b^2(\R^d)$ we have
		\begin{equation*}
			\sup_{1 \leq \lvert\gamma\rvert_* \leq 2}\lVert \partial^{\gamma}_{x} \bar{\cP}_{t}\psi(x) \rVert_\infty^2  \leq \tilde{K}e^{-2C t}\sum_{1 \leq \lvert\gamma\rvert_* \leq 2}\lVert \partial^{\gamma}_{x} \psi(x) \rVert_\infty^2.
		\end{equation*}
	\end{proposition}
	\begin{proof}
		The proof of this result can be found in Appendix \ref{app:proofsstrongergodicity}.
	\end{proof}
	Below, we gather another example for which our main result, Theorem \ref{mainthm}, holds, and in particular for which the semigroup derivative estimates concluded in Proposition \ref{prop:avgderivativeest} hold. We will use this example to illustrate the discussion in Note \ref{rem:xavgderivs}.
	
	\begin{example}\label{ex:Lip}
		Let us consider following the system
		\begin{align*}
			dX_{t}^{\epsilon,x,y} & =  (-X_t^\epsilon +b_0( Y_{t}^{\epsilon,x,y}) )dt + \sqrt{2} \, dW_t \\ 
			dY_{t}^{\epsilon,x,y} & = \frac{1}{\epsilon}(-Y_t^\epsilon+g_0(X_{t}^{\epsilon,x,y})) dt + \frac{1}{\sqrt{\epsilon}} \sqrt{2}\, dB_t
		\end{align*}
		for some $b_0,g_0\in C_b^\infty(\R)$. Though we wish to emphasise the verification of Assumption \ref{ass:avgSGcond} in this example, we first show that Assumptions \ref{ass:polGrowth}-\ref{ass:SGcond} hold. It is immediate to see that Assumption \ref{ass:polGrowth} holds with $m^b_x=1, m^b_y=0, \pow{g}{y}=1$ and that Assumption \ref{UniformEllip} holds with $\lambda_-=\lambda_+=1$. As commented in Note \ref{rem:commentonassumptions} in order to show that Assumption \ref{DriftAssump} holds it is sufficient to consider $k=1$. For $k=1$, \eqref{eq:driftassump} holds with $C_1=\lVert g_0\rVert_\infty^2/2, r_1=1/2$. Similarly, Assumption \ref{DriftAssumpS} holds with $\tilde{C}_1=\lVert b_0\rVert_\infty^2/2, \tilde{r}_1=1/2$. Assumption \ref{ass:SGcond} holds with $\zeta_1=\delta$ and $\kappa=2$ for any $\delta>0$. It remains to verify Assumption \ref{ass:avgSGcond}. Observe that $\yfourder=\yfourdery=
		\ytwoderyonex= \ytwoderyoney=\ytwoderx= \ytwoderxy= \ytwodery = \ytwoderyy = \onederxg = 0$, $D_0 = \sqrt{6}$ (taking $\delta = \lambda_-$), and $D_b= \lVert \partial_x g_0 \rVert_{\infty}$.
		Therefore, Assumption \ref{ass:avgSGcond} holds provided $$\max\{\lVert \partial_y b_0\lVert_{\infty},\lVert \partial_{yy} b_0\lVert_{\infty}\}<\frac{1}{\sqrt{6}\lVert \partial_x g_0 \rVert_{\infty}}.$$
		
		It may not be immediately obvious that there should be such a relation between the $y$ derivatives of the drift of the slow equation  and the $x$ derivative of the drift of the fast equation, so let us motivate it with a specific choice of $b_0$ and $g_0$. As noted in Note \ref{rem:xavgderivs}, it is in some cases possible to obtain an explicit expression for $\bar b$, allowing one to verify a condition of the type Assumption \ref{ass:SGcond}. If we set $g_0(x)=\sin(x)$ and $b_0(y)=-r \cos(y)$, for some $r \in \R$ then the averaged equation for this setting becomes
		\begin{equation*}
			d\bar{X}_{t}^{x}  =  -\bar{X}_{t}^{x} -r\sqrt{\frac{1}{e}} \cos(\sin(\bar{X}_{t}^{x})) dt + \sqrt{2} \, dW_t.
		\end{equation*}
		From the above, this system would satisfy Assumption \ref{ass:avgSGcond} when $|r| < 1/\sqrt{6}$.
		For $|r| <3.5$ the averaged coefficient $\bar{b}(x)$ is monotonic. In particular it would satisfy a drift condition of the form \ref{ass:SGcond} with $g$ replaced by $\bar{b}$ and taking $x$ derivatives instead of $y$ derivatives. Therefore we do have derivative estimates of the form given by the conclusion of Proposition \ref{prop:avgderivativeest}. Hence we see, in both cases, that $r$ is limited to a bounded set of values, but that our Assumption \ref{ass:avgSGcond} is more restrictive. For $1/\sqrt{6}<|r|<3.5$, our Assumption $\ref{ass:avgSGcond}$ does not hold, but we can still obtain derivative estimates of the form given by the conclusion of Proposition \ref{prop:avgderivativeest}, since we have a usable expression for $\bar b$. This partly motivates the discussion in Note \ref{rem:xavgderivs}.
	\end{example}


	\subsection{Application of Section \ref{sec:psneqn} to Averaging}\label{sec:appPsnEqn}
	
	Here we use the results of Section \ref{sec:psneqn} to obtain bounds on $f_t^1$, defined by \eqref{eqn:f_t^1}, which are required for the proof of Theorem \ref{mainthm}. 
	We recall the notation of Section \ref{sec:mainresults} and Section \ref{sec:psneqn} where we refer to the solution of \eqref{poisprobgeneral} as $u_\phi$ for $\phi\in \funcSpace{\pow{}{x}}{\pow{}{y}}$. 
	
	
	\begin{proposition}\label{psnEqDer}
		Let Assumptions \ref{ass:polGrowth}-\ref{ass:avgSGcond} hold. Then there exists $C, \omega > 0$ such that, for all $t>0$, $x\in \R^n$, $y\in \R^d$, we have
		\begin{align}\label{eqn:ft1Bound}
			|f_t^1(x,y)| &\leq C\lVert f\rVert_{C_b^2}(\Vnorm{b}{4,\pow{b}{x},\pow{b}{y}})e^{-\omega t}(1+ |y|^{m^{b}_{y}}+|x|^{4\pow{b}{x}})\\
			\label{eqn:1cor}
			\left|\sum^n_{i = 1} \slowDer{i} f_t^1(x,y)\right| &\leq C\lVert f\rVert_{C_b^2}(\Vnorm{b}{4,\pow{b}{x},\pow{b}{y}})e^{-\omega t}(1+ |y|^{M^{g, b}_{y}}+|x|^{4\pow{b}{x}}) \\
			\label{eqn:2cor}
			\left|\sum^n_{i,j = 1} \slowDerDer{i}{j} f_t^1(x,y)\right| &\leq C\lVert f\rVert_{C_b^2}(\Vnorm{b}{4,\pow{b}{x},\pow{b}{y}})e^{-\omega t}(1+ |y|^{M^{g, b}_{y}+\pow{g}{y}}+|x|^{4\pow{b}{x}}),
		\end{align}
		where $M^{g, b}_{y}$ is defined in Theorem \ref{mainthm}.
		Furthermore,
		\begin{equation}\label{eqn:Ls-ds}
			\left| (\mathcal{L}_S - \partial_t)f_t^1(x,y) \right| \leq Ce^{-\omega t} \lVert f\rVert_{C_b^2}(\Vnorm{b}{4,\pow{b}{x},\pow{b}{y}}+\Vnorm{\Sigma}{4,0,0})(1+ |y|^{M^{g, b}_{y}+\pow{g}{y}}+|x|^{4\pow{b}{x}}) \,.
		\end{equation}
		
	\end{proposition}

	\begin{proof}[Proof of Proposition \ref{psnEqDer}]
		By \eqref{eqn:IntAvg2}, \eqref{eq:intAvgDer} and \eqref{eqn:intAvgDerDer} with $\phi=b_i$ we have for each $i\in \{1,\ldots,n\}$
		\begin{align}
			\left| u_{b_i} \right| &\leq C\Vnorm{b_i}{0,\pow{b}{x},\pow{b}{y}}(1+ |y|^{\pow{b}{y}}+|x|^{\pow{b}{x}})\label{eqn:invL}\\
			\left|\slowDer{i} u_{b_i} \right| &\leq C\Vnorm{b_i}{2,\pow{b}{x},\pow{b}{y}}(1+ |y|^{M^{g, \phi}_{y}}+|x|^{2\pow{b}{x}})\label{eqn:invLder}\\
			\left|\slowDerDer{i}{j} u_{b_i} \right| &\leq C\Vnorm{b_i}{4,\pow{b}{x},\pow{b}{y}}(1+ |y|^{M^{g, \phi}_{y}+\pow{g}{y}}+|x|^{4\pow{b}{x}})\label{eqn:invLderDer}.
		\end{align}
		We use Proposition \ref{prop:avgderivativeest} along with \eqref{eqn:invL}-\eqref{eqn:invLderDer} to conclude \eqref{eqn:ft1Bound}-\eqref{eqn:2cor}.
		Let us now consider \eqref{eqn:Ls-ds}.	We differentiate (\ref{eqn:f_t^1}) with respect to time to get
		\begin{align*}
			\partial_t f_t^1(x,y) = &-\sum^n_{i=1} u_{ b_i}(x,y) \partial_t\partial_{x_i} \widebar{\mathcal{P}}_t f(x).
		\end{align*}
		Hence, by \eqref{eq:barPDE}
		\begin{equation}\label{eqn:ft1partialt}\begin{split}
				\partial_t  f_t^1(x,y) =  -\sum^n_{i=1} u_{b_i}(x,y) \partial_{x_i}\widebar{\cL} \widebar{\mathcal{P}}_t f(x).
			\end{split}
		\end{equation}
		The term $u_{b_i}$ is bounded by \eqref{eqn:IntAvg2}. So we turn our attention to $ \partial_{x_i}\widebar{\cL} \widebar{\mathcal{P}}_t f(x)$:
		\begin{align}\label{eqn:DerAvg}
			\begin{split}
				\slowDer{k} \left(\widebar{\cL}\widebar{\mathcal{P}}_t f(x)\right) &\stackrel{\eqref{bar L}}{=} \sum_{i=1}^n \left(\slowDer{k}\widebar{b}_i(x)\right) \slowDer{i} \widebar{\mathcal{P}}_t f(x) + \sum_{i,j = 1}^n \left(\slowDer{k}\widebar{\Sigma}_{ij}(x)\right) \slowDerDer{i}{j}\widebar{\mathcal{P}}_t f(x) + \bar{\cL} \slowDer{k}\widebar{\mathcal{P}}_t f(x).
			\end{split}
		\end{align}
		
		Further, using \eqref{bar L}, \eqref{LFast and LSlow},  and \eqref{eqn:ft1partialt}-\eqref{eqn:DerAvg} we can write
		\begin{align*}
			&(\mathcal{L}_S - \partial_s)f_s^1(x,y) = \\
			&-\sum^n_{i,j=1}b_i(x,y)\partial_{x_i}u_{ b_j}(x,y)\slowDer{j} \widebar{\mathcal{P}}_s f(x)  - \sum^n_{i,j=1}\left(b_i(x,y)-\bar{b}_i(x)\right)u_{ b_j}(x,y)\slowDerDer{i}{j} \widebar{\mathcal{P}}_s f(x)\\
			&-\sum^{n}_{i,j,k = 1}\Sigma_{ij}(x)\slowDerDer{i}{j}u_{b_k}(x,y) \slowDer{k} \widebar{\mathcal{P}}_s f(x) -2\sum^{n}_{i,j,k = 1}\Sigma_{ij}(x)\slowDer{i}u_{ b_k}(x,y) \slowDerDer{j}{k} \widebar{\mathcal{P}}_s f(x) \\
			&+ \sum^n_{i,j=1}\slowDer{i}\bar{b}_j(x)u_{b_i}(x,y)\slowDer{j} \widebar{\mathcal{P}}_s f(x) +\sum^n_{i,j,k=1}\slowDer{i}\Sigma_{jk}(x)u_{ b_i}(x,y)\slowDerDer{j}{k} \widebar{\mathcal{P}}_s f(x) ,
		\end{align*}
		and using Proposition \ref{prop:avgderivativeest}, \eqref{eq:intAvgDerInv}, \eqref{eqn:invL}, \eqref{eqn:invLder}, \eqref{eqn:invLderDer} and Assumption \ref{ass:polGrowth} (\ref{ass:polGrowthDrift},\ref{ass:polGrowthDiff})
		\begin{equation*}
			\left| (\mathcal{L}_S - \partial_s)f_s^1(x,y) \right| \leq Ce^{-\omega s}\lVert f\rVert_{C_b^2}\Vnorm{b}{4,\pow{b}{x},\pow{b}{y}}(1+\Vnorm{\Sigma}{4,0,0}) (1+ |y|^{M^{g, b}_{y}+\pow{g}{y}}+|x|^{4\pow{b}{x}}) \,.
		\end{equation*}
	\end{proof}

	\section{Numerics and Examples}\label{sec:numerics}
	In this section we provide numerical evidence for the validity of Theorem \ref{mainthm} and illustrate the applicability of the result. The three systems that we numerically solve in this section have different drift terms for the fast equation. The first two differ in how  they are coupled, but both have monotonic drifts (corresponding to a convex potential), and the third does not. The aim of this is to produce a set of numerical results that fit the conditions of Theorem $\ref{mainthm}$, as well as a set of numerical results that can motivate or support future work, being outside the conditions of Theorem $\ref{mainthm}$. The second example will illustrate that we believe the semigroup derivative estimates to be an important part in obtaining a UiT bound. We begin with an example where $g$ in (\ref{fast}) is monotonic.
		\begin{figure}[ht]
		\centering
		\includegraphics[width = \textwidth, keepaspectratio]{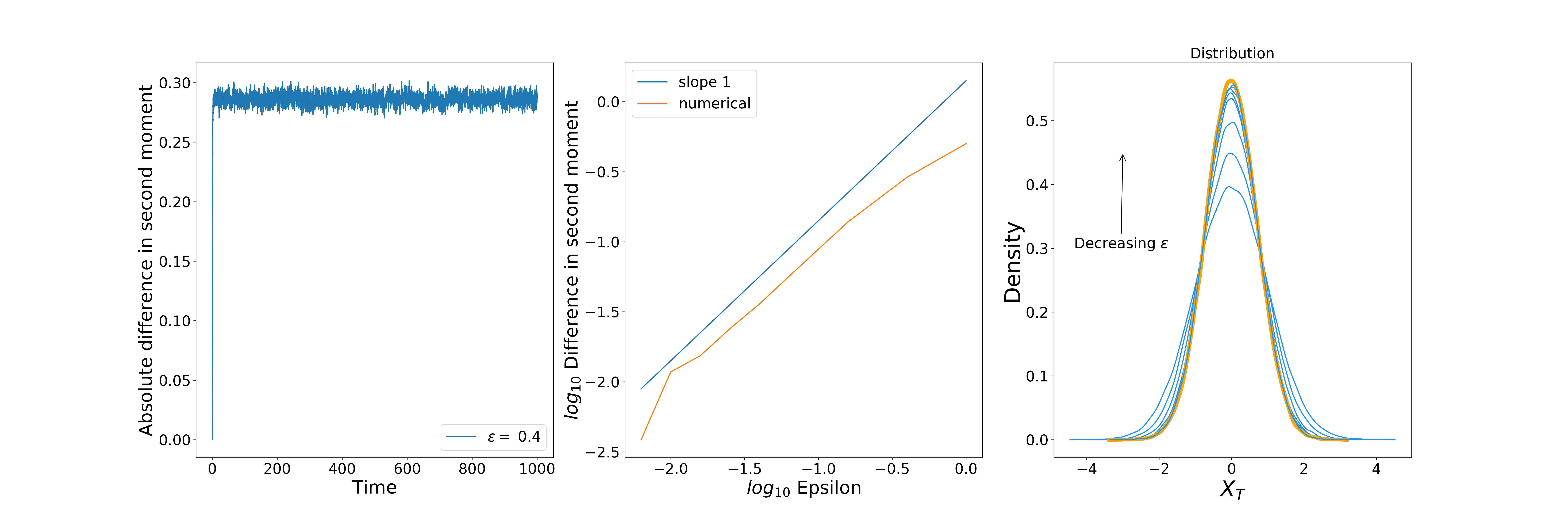}
		\caption{\emph{Left.} Plot of $\left|\mathbb{E}\left| X_{t}^{\epsilon,x,y} \right|^2-\mathbb{E}\left| \bar X_{t}^{x} \right|^2\right|$, plotted up until time $t=10^3$, with $\epsilon = 0.4$. \emph{Middle.} The maximum difference (over time) in second moment, plotted against $\epsilon$. Plotted alongside a line of slope $1$ for comparison, which is what we expect from Theorem \ref{mainthm}. \emph{Right.} The distributions of $X_{T}^{\epsilon,x,y}$ for different $\epsilon$ values. Plotted alongside the distribution of $\bar{X}_{T}^{x}$ (in yellow), which is independent of $\epsilon$. These plots were produced using the behaviour of $10^5$ particles all initialised at $0$, with $\epsilon = 10^{-m}  $ and $\Delta t = 10^{-m-3}$ for  values $m \in \{0,0.4,0.8,1.2,1.4,1.6,1.8,2,2.2 \}$. All realisations were run for $N=10^6$ steps (hence run until $T=10^{-m+3}$) and were initialised at $X_{0}^{\epsilon,x,y}=\bar X_{0}^{x}=Y_{0}^{\epsilon,x,y}=0$.}
		\label{fig:Conv}
	\end{figure}
	\begin{example}[Convex] \label{ex:conv}
		Consider the following coupled slow-fast system
		\begin{align}
			dX_{t}^{\epsilon,x,y} & =  -\left[ Y_{t}^{\epsilon,x,y} + X_{t}^{\epsilon,x,y}\right] dt +  \, dW_t, \label{slowConv} \\ 
			dY_{t}^{\epsilon,x,y} & = -\frac{1}{\epsilon}\left[ Y_{t}^{\epsilon,x,y} \right] dt + \frac{\sqrt{2}}{\sqrt{\epsilon}} \, dB_t,  \label{fastConv} 
		\end{align}
		and the corresponding averaged equation
		\begin{align}
			\label{avgConv}
			d \bar{X}_{t}^{x} = -\bar{X}_{t}^{x} dt + d\tilde{B}_t,
		\end{align}
		where $W_t,B_t,\tilde{B}_t$ are all Brownian motions independent to each other. Note that the above system does not satisfy Assumption \ref{DriftAssumpS}, since the constant $\tilde{C}$ cannot be independent of $y$. However, what we do throughout the following is replace the drift of $X_{t}^{\epsilon,x,y}$ with $b(x,y) = -\left[ y\mathbb{I}_{\left|y\right|<M}+\varphi(y)\mathbb{I}_{\left|y\right|>M} + x\right]$, where $\varphi(y)$ is a smooth function with $\varphi(M)=M, \varphi(-M)=-M, \varphi'( M)=1$ and where $\varphi(y)=\varphi(-y)=0$ for all $y >M+1$. In practice, this makes no difference to the numerics because, for large $M$, say $M=100$, the dynamics (as it is Gaussian) reaches this point with such low probability that it is never seen in the discretisation.
		
		We numerically solve both the coupled system (\ref{slowConv}-\ref{fastConv}) and (\ref{avgConv}) using the Euler-Maruyama scheme (see, e.g, \cite{numericsForSDEs}), with a time step $\Delta t \ll \epsilon$, to ensure negligible discretisation error in the time interval (the exact $\Delta t$ will be given later). Halving the time step resulted in no observable change of results. There are two approximations being made here-- one being the method of averaging, and one being discretising the equations to produce numerical results. Since we are testing a result that is UiT, it is important that the discretisation of the system that we use to solve it is itself UiT.  For this, we refer the reader to \cite{uitEuler}. Using the results of \cite{uitEuler}, one can see that both the coupled system and the averaged equation are approximated uniformly in time by the Euler-Maruyama scheme. In Figure \ref{fig:Conv}, we show that the method of averaging produces a UiT approximation to the system (\ref{slowConv})-(\ref{fastConv}). The left hand plot indicates that the difference between the second moment of $\bar{X}_{t}^{x}$ and the second moment of $X_{t}^{\epsilon,x,y}$ is bounded uniformly in time. \footnote{We choose the second moment as the comparison statistic because of the Gaussian nature of the considered processes, and because the first moment is trivial.} The left hand plot is for one value of $\epsilon$, and is indeed not surprising on its own for a system in which the second moment is bounded uniformly in time itself. The middle plot indicates the preservation of the convergence in $\epsilon$. Here, for each value of $\epsilon$, maximum of the differences in second moment over the time. We have, then, that one statistic of the approximation induced by the method of averaging stays uniformly close in time to the coupled system seen in Example \ref{ex:conv}, and that this closeness increases as $\epsilon$ decreases. In theory, the $f$ from Theorem \ref{mainthm} needed to produce the second moment (i.e the quadratic $f(x) = x^2$) is unbounded, meaning that the results in this paper, strictly speaking, do not apply. In practice, however, this does not particularly pose a problem, since we discretise. Indeed, because we discretise, the domain sampled is finite meaning the function is bounded. In the right hand side plot one can see that this does generalise to the distribution, giving indication that for any bounded statistic $f$, the bound in Theorem \ref{mainthm} would hold. We conclude that these results are in agreement with the theory developed in this paper.
	\end{example}

		Now we introduce an example to illustrate the fact that not all systems for which convergence of the method of averaging holds exhibit uniform in time convergence.
  
	\begin{example}\label{ex:decDerCE}
		\begin{align}
			dX_{t}^{\epsilon,x,y} & =  -\left[ Y_{t}^{\epsilon,x,y} + X_{t}^{\epsilon,x,y} \right] dt +  \, dW_t \label{slowDecDer} \\ 
			dY_{t}^{\epsilon,x,y} & = -\frac{1}{\epsilon}\left[ Y_{t}^{\epsilon,x,y} + X_{t}^{\epsilon,x,y} \right] dt + \frac{\sqrt{1} }{\sqrt{\epsilon}}\, dB_t  \label{fastDecDer} 
		\end{align}
		which yields
		\begin{align}
			\label{avgDecDer}
			d \bar{X}_{t}^{x} = d\tilde{B}_t,
		\end{align}
		where $W_t,B_t,\tilde{B}_t$ are all Brownian motions independent of each other.
		
		 The system \eqref{slowDecDer}-\eqref{fastDecDer} is fairly simple. Indeed, one can solve the coupled system since it is a multivariate Ornstein-Uhlenbeck process (see for e.g. \cite[Section 4.4.6]{gardiner1985handbook}): the law of $X_{t}^{\epsilon,0,0}$ where the system is initialised at $\left( X_{0}^{\epsilon,x,y}, Y_{0}^{\epsilon,x,y}\right) = (0,0)$ is that of a Gaussian random variable with the following mean and variance:
   \begin{equation}\label{eqn:exactdecder}
       \mathbb{E} \left[ X_{t}^{\epsilon,0,0} \right] = 0, \quad \mathbb{E} \left[ \left(X_{t}^{\epsilon,0,0}\right)^2 \right] = \frac{\epsilon^2 \left(1-e^{-\frac{2(1+\epsilon)t}{\epsilon}} \right)+2(1+\epsilon)t}{2(1+\epsilon)^2}.
   \end{equation}
   Of course, we can also exactly solve the averaged equation \eqref{avgDecDer}, which, when initialised at $\bar{X}_{0}^{x} =0$, has mean $0$ and variance $t$. We can see that, for every $\epsilon >0$, the variance of \eqref{avgDecDer} diverges from the variance of \eqref{slowDecDer}. We plot these functions for a couple of different values of $\epsilon$ in Figure \ref{fig:decderplotvars}. Hence, one cannot expect uniform in time convergence to hold.
   \begin{figure}[ht]
		\centering
		\includegraphics[width = 0.6\textwidth, keepaspectratio]{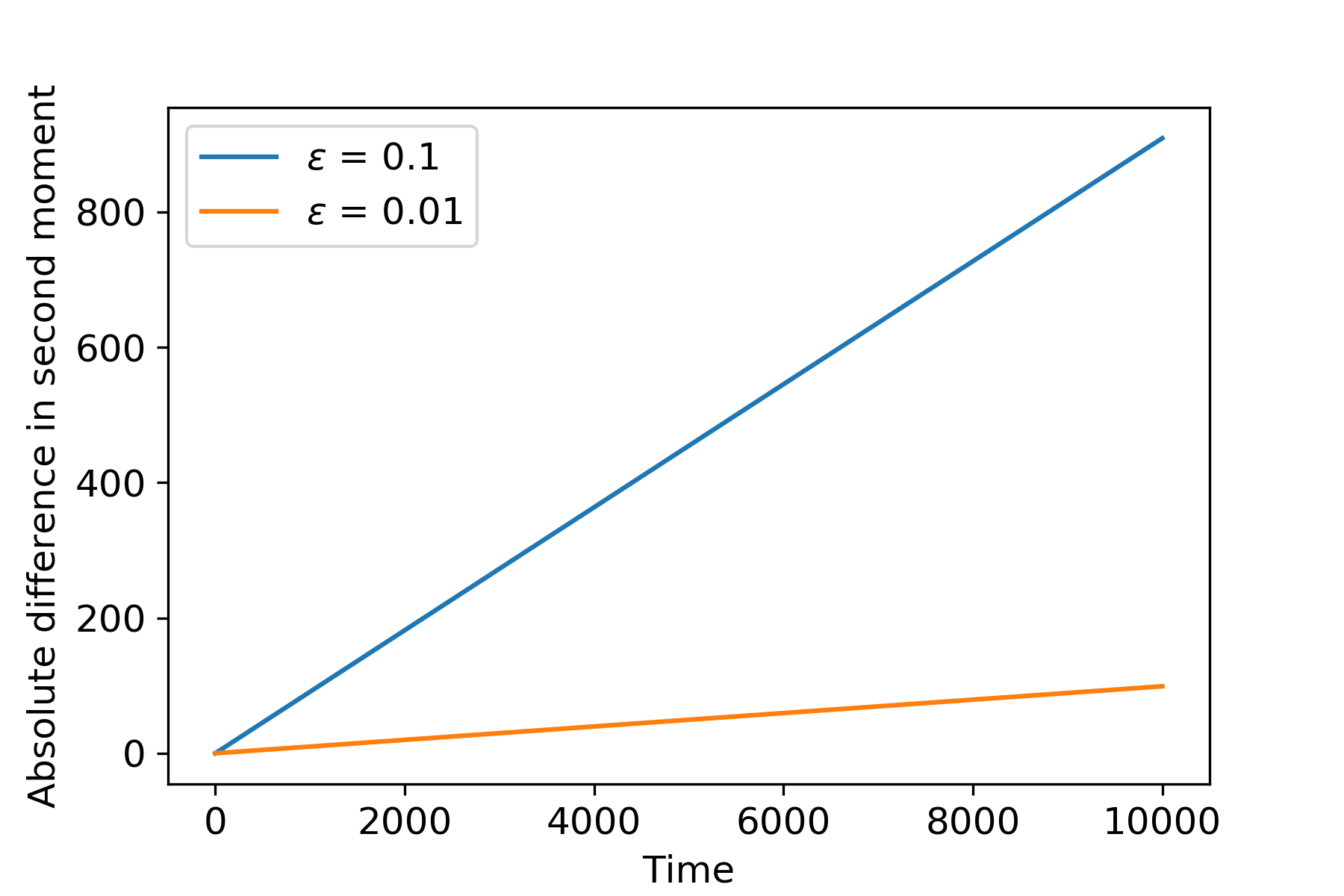}
		\caption{Plot of $\left|\mathbb{E}\left| X_{t}^{\epsilon,x,y} \right|^2-\mathbb{E}\left| \bar X_{t}^{x} \right|^2\right|$ for the system \eqref{slowDecDer}-\eqref{fastDecDer}, plotted up until time $t=10^5$, for both $\epsilon = 0.1$ and $\epsilon = 0.01$. Note that this is exact, since the system can be solved for analytically (see \eqref{eqn:exactdecder}).}
		\label{fig:decderplotvars}
	\end{figure}

Indeed, though simple, the system does not satisfy several of our conditions. Assumption \ref{ass:polGrowth} is not satisfied because the drift in \eqref{fastDecDer} is not bounded in the slow variable $x$. Where this becomes a problem in the proof is that the constants $C_k$ in Assumption \ref{DriftAssump} are no longer independent of $x$. Hence, the constants in the exponential ergodicity bound obtained in Lemma \ref{prop:momentBounds} are not independent of $x$. This causes difficulties throughout the proofs of the results for the Poisson equation. In addition, Assumption \ref{ass:avgSGcond} does not hold, and moreover the averaged semigroup is not SES in the sense of Assumption \ref{ass:averagederivativeestimate4}. In fact our method of proof would still work if the bound of the averaged semigroup derivatives was merely integrable in time over the positive real line, but again this would not hold, since the averaged equation \eqref{avgDecDer} is a brownian motion. Hence, this is an example in which both the `fast' and `slow' equation are well-behaved but the corresponding averaged equation is not so, and this rules out the possibility of a UiT bound under our mechanism of proof. Indeed, the authors of this paper believe that conditions on the averaged semigroup derivatives (or conditions that imply conditions on the averaged semigroup derivatives) are necessary to obtain UiT convergence since they appear in $f_t^1$, see \eqref{eqn:f_t^1}. This example would satisfy the conditions of prior, non uniform in time results (see e.g \cite[Theorem 2.3]{rockner2020diffusion}), and therefore is a case of a relatively simple system for which one can obtain convergence over finite time horizons, but where uniform in time convergence does not hold. 
 
		\end{example}
	
	Now we introduce an example of a system that does not satisfy the conditions under which we have proven the UiT bound, with numerical experiments similar to those in Example \ref{ex:conv}.
	\begin{example}[Double Well]\label{ex:DW}
		Consider the following coupled slow-fast system
		\begin{align}
			dX_{t}^{\epsilon,x,y} & =  -\left[ Y_{t}^{\epsilon,x,y} + X_{t}^{\epsilon,x,y} \right] dt +  \, dW_t \label{slowDW} \\ 
			dY_{t}^{\epsilon,x,y} & = -\frac{1}{\epsilon}\left[ Y_{t}^{\epsilon,x,y}(Y_{t}^{\epsilon,x,y}+2)(Y_{t}^{\epsilon,x,y}-2) \right] dt + \frac{\sqrt{2}}{\sqrt{\epsilon}}\, dB_t  \label{fastDW} 
		\end{align}
		and the associated averaged equation
		\begin{align}
			\label{avgDW}
			d \bar{X}_{t}^{x} = -\bar{X}_{t}^{x} dt + d\tilde{B}_t,
		\end{align}
		where $W_t,B_t,\tilde{B}_t$ are all Brownian motions independent to each other.
As in Example \ref{ex:conv}, we replace the drift of $X_{t}^{\epsilon,x,y}$ with $b(x,y) = -\left[ y\mathbb{I}_{\left|y\right|<M}+\varphi(y)\mathbb{I}_{\left|y\right|>M} + x\right]$, where $\varphi(y)$ is a smooth function with $\varphi(M)=M, \varphi(-M)=-M, \varphi'( M)=1$ and where $\varphi(y)=\varphi(-y)=0$ for all $y >M+1$ with the same observation that for large $M$, say $M=100$, the dynamics never reaches this point.
  Note that, while the above system satisfies \ref{ass:polGrowth}-\ref{DriftAssumpS}, it does not satisfy Assumption \ref{ass:SGcond}.
		
		We use Euler-Maruyama, as in the previous example, to solve this system numerically. Again, we can verify that $Y_{t}^{\epsilon,x,y}$  satisfies the conditions in \cite[Hypothesis $3.1$]{uitEuler}. This is not immediate, but one can utilise Proposition \ref{prop:derivativeest} to obtain the semigroup derivative decay. The resulting plots can be seen in Figure \ref{fig:DW}, and they were produced in the same way as Figure \ref{fig:Conv}. From the plots in Figure \ref{fig:DWnz}, one might conjecture that there can indeed be a similar UiT bound that holds when $g$ in (\ref{fast}) is non-convex. That is, Assumption \ref{ass:SGcond} may be sufficient rather than necessary (though, as discussed in the previous example, the authors expect that decay of the derivatives of the semigroup in some form is necessary). We note to the reader that we numerically solved the system for a shorter and shorter time frame (see the caption of Figure \ref{fig:DW} for details) as $\epsilon$ decreased to allow for the finer discretisation needed, but that all distributions were checked to be in or very close to equilibrium at the point of stopping.
		
	    Though not directly relevant to the existence of a UiT bound, we include Figure \ref{fig:DWz} as an example of a transition in the qualitative behaviour of the method of averaging as $\epsilon$ varies. It is simply a restriction to the middle $\epsilon$-regime of Figure \ref{fig:DWnz}. From the right hand side plot, we see that for large $\epsilon$ the settled distribution of the averaged equation and the fully coupled system are qualitatively distinct; that is, they are single-modal and bi-modal respectively. On the left hand side plot, it is clear that this is also a quantitative difference, though it does also portray a smooth transition as the convergence in $\epsilon$ is shallower in this large $\epsilon$ regime than in the small regime. While the order of convergence is $1$ as $\epsilon \rightarrow 0$ (as we have proved in Theorem \ref{mainthm}), the order of convergence may be far slower if $\epsilon$ is too large. Hence, it is important when using the method of averaging to ensure that $\epsilon$ is small enough that there is not such a qualitative difference between the behaviour of the averaged equation and the coupled system.
	\end{example}
	\begin{figure}[t!]
		\begin{subfigure}{\textwidth}
			\includegraphics[width = \textwidth, keepaspectratio]{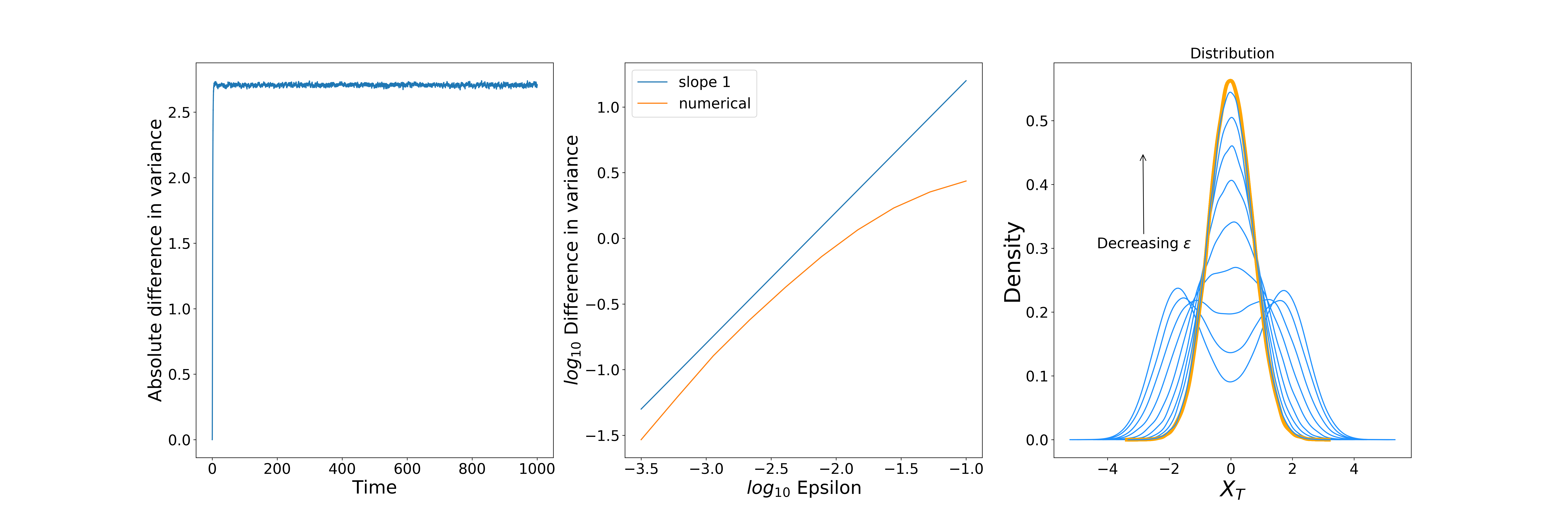}
			\caption{\emph{Left.}  Plot of $\left|\mathbb{E}\left| X_{t}^{\epsilon,x,y} \right|^2-\mathbb{E}\left| \bar X_{t}^{x} \right|^2\right|$, plotted until time $t=10^3$, with $\epsilon = 0.1$. \emph{Middle.} The maximum difference in second moment, plotted against $\epsilon$. Plotted alongside a line of slope $1$ for comparison. \emph{Right.} The corresponding distributions of $X_{t}^{\epsilon,x,y}$ for different $\epsilon$ values. Plotted alongside the distribution of $\bar{X}_{t}^{x}$, which is independent of $\epsilon$. These plots were produced using the behaviour of $10^5$ particles, with $\epsilon = 10^{-m}  $ and $\Delta t = 10^{-m-3}$ for  values $m \in \{1.        , 1.28, 1.56, 1.83, 2.11,
				2.39, 2.67, 2.94, 3.22, 3.5 \}$. All realisations were run for $N=10^6$ steps (hence run until $T=10^{-m+3}$) and were initialised at $X_{0}^{\epsilon,x,y}=\bar X_{0}^{x}=Y_{0}^{\epsilon,x,y}=0$. The authors checked that the distributions had reached equilibrium for all values of $\epsilon$.}
			\label{fig:DWnz}
		\end{subfigure}
		\begin{subfigure}{\textwidth}
			\includegraphics[width = \textwidth, keepaspectratio]{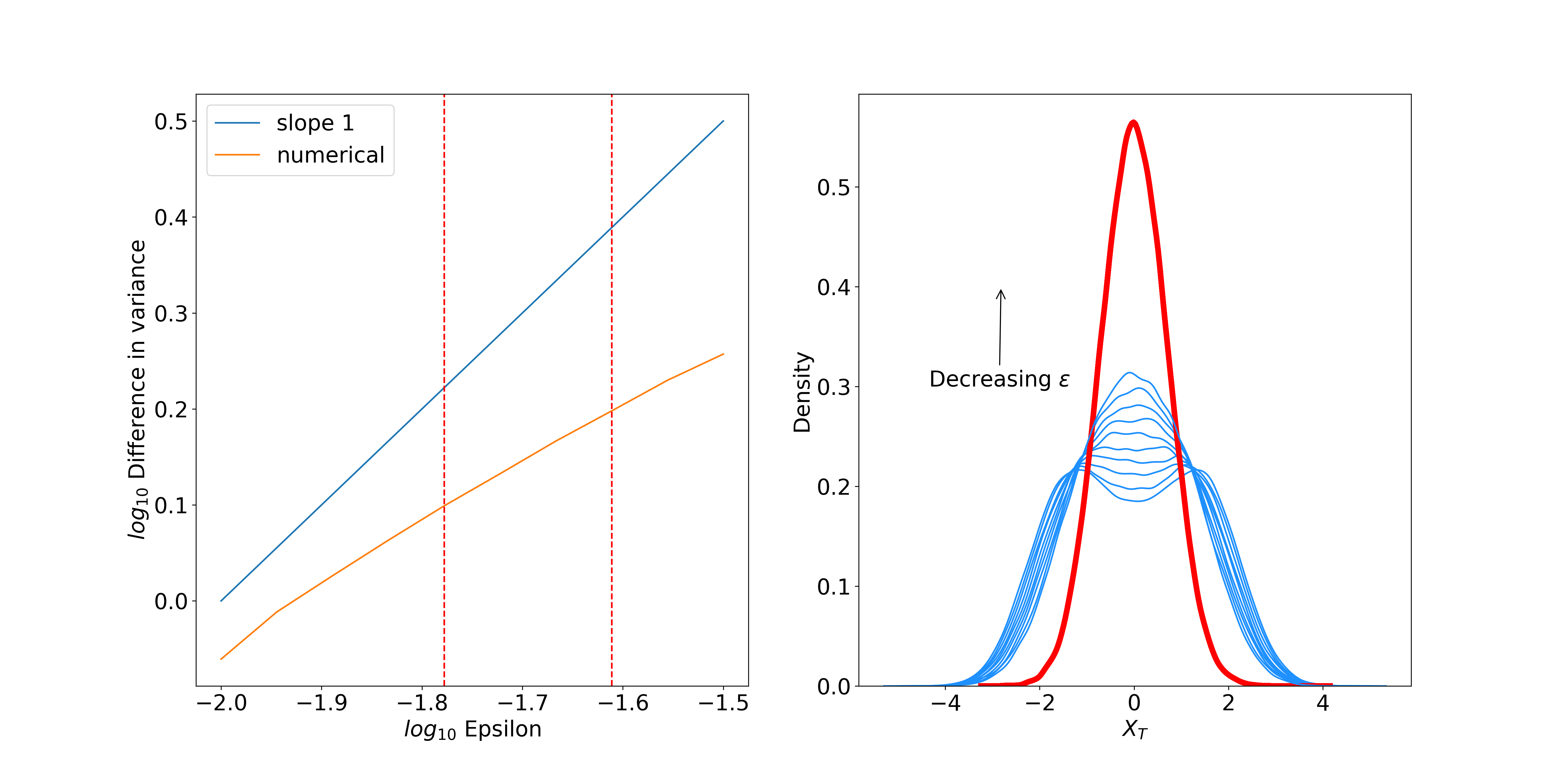}
			\caption{\emph{Left.} The maximum difference in second moment, plotted against $\epsilon$. Plotted alongside a line of slope $1$ for comparison. The red-dashed lines indicate points at which the coupled system is clearly bimodal (right) and clearly single-modal (left). \emph{Right.} The distributions of $X_{T}^{\epsilon,x,y}$ for different $\epsilon$ values. Plotted alongside the distribution of $\bar{X}_{T}^{x}$ (in orange), which is independent of $\epsilon$. These plots were produced with the same parameters as above, using $m \in \{1.5       , 1.56, 1.61, 1.67, 1.72,
				1.78, 1.83, 1.89, 1.94, 2. \}$.}
			\label{fig:DWz}
		\end{subfigure}
		\caption{}\label{fig:DW}
	\end{figure}

\begin{appendix}
	\section{Proofs}\label{Asec:proofLemma}
	\subsection{Proofs of Section \ref{sec:sec3}}\label{app:proofofprelims}
	
	Let us recall (see Note \ref{note:coupledlemmas}) that we need 
		 Lemma \ref{prop:momentBounds} and Proposition \ref{lemma:IntAvg} to hold if the diffusion coefficient $a$ is also a function of $y$, as is relevant in the proofs for the fully coupled case. So the proofs of Lemma \ref{prop:momentBounds} and Proposition \ref{lemma:IntAvg} below are done in this more general setting.

	\begin{proof}[Proof of Lemma \ref{prop:momentBounds}]
		The proof of this result can be found in \cite[Lemma $3.6$]{liu2020averaging}. We observe that Assumption \ref{DriftAssump} implies \cite[Assumption $A_k$]{liu2020averaging} for every $k\geq 2$ which is the main assumption of \cite[Lemma $3.6$]{liu2020averaging}. We include the calculation here to give an explicit form of the constants $C'_k$ and $r'_k$. 
		The result holds for $k=1$ by Jensen's inequality and the case $k=2$. Indeed, $$\fastExp{y} \left[|\fastFixedI_t|\right] \leq \left(\fastExp{y} \left[|\fastFixedI_t|^2\right]\right)^{1/2}\leq e^{-r_2't/2}|y| + \sqrt{\frac{ C_2'}{r_2'}}.$$
		
		Now let us consider the case {$k > 2$}. Assumption \ref{DriftAssump} implies that the function $V(y) = |y|^k$ is a Lyapunov function for the process $\fastFixedI_{t}$ (see \cite[Lemma $3.6$]{liu2020averaging} for details) in the sense that for all $x \in \R^n, y \in \R^d$
		\begin{align}\label{eqn:lyapunov}
			\cL^x V(y) &\leq -kr_k \lvert y\rvert^k+C_kk\lvert y\rvert^{k-2}\\
			&\leq -\frac{kr_k}{2} V(y) + 2C_k\left(\frac{2C_k(k-2)}{kr_k}\right)^{\frac{k-2}{2}}.
		\end{align}
		To obtain the last inequality we have used Young's inequality.

		Then, using the positivity of Markov semigroups, we have
		\begin{equation*}
			\frac{d}{dt}(e^{kr_k t/2}\SGfastsx{t}{x}V(y)) = (\frac{1}{2}kr_k\SGfastsx{t}{x}V(y)+\SGfastsx{t}{x}\cL^x V(y))e^{kr_kt/2} \leq 2C_k\left(\frac{2C_k(k-2)}{kr_k}\right)^{\frac{k-2}{2}}e^{kr_kt/2}.
		\end{equation*}
		We can then integrate and obtain
		\begin{equation*}
			\mathbb{E}\left[|\fastFixedI_t|^k \right] =	\SGfastsx{t}{x}V(y) \leq e^{-kr_k t/2} V(y) + \frac{4C_k}{kr_k}\left(\frac{2C_k(k-2)}{kr_k}\right)^{\frac{k-2}{2}}.
		\end{equation*}
        {By an analogous argument for $k=2$ we have
        \begin{equation*}
			\mathbb{E}\left[|\fastFixedI_t|^2 \right] =	\SGfastsx{t}{x}V(y) \leq e^{-2r_2 t} V(y) + \frac{C_2}{r_2}.
		\end{equation*}}
	\end{proof}
	\begin{proof}[Proof of Proposition \ref{lemma:IntAvg}.]
		Using \cite[Theorem 2.5]{MATTINGLYSTUARTergodicity}, which holds under our Assumptions \ref{UniformEllip} and \ref{DriftAssump}, we obtain \eqref{eqn:SGconv} for some $C=C(x), c=c(x)$.\footnote{Strictly speaking \cite[Theorem 2.5]{MATTINGLYSTUARTergodicity} refers to a hypoelliptic generator with constant diffusion matrix. The process to which we apply this theorem is elliptic, with non-constant diffusion matrix. Under our assumptions the conclusion of \cite[Theorem 2.5]{MATTINGLYSTUARTergodicity} still holds, as one can easily see by following the proof of that theorem. Nonetheless, because we need to control the dependence of the various constants on $x$, in what follows we anyway retrace a good part of that proof. } What we need within our scheme of proof, is that $C(x)$ and $c(x)$ are uniformly bounded above and below respectively. To prove this fact, we will show that the conditions of \cite[Theorem 2.5]{MATTINGLYSTUARTergodicity} hold uniformly in $x$. In particular, in \cite[Theorem 2.5]{MATTINGLYSTUARTergodicity}, the authors require two conditions to hold: a Minorisation condition, and a Lyapunov condition. For the readers convenience we gather these two conditions and then write a modified Minorisation condition, which we will use for our proof.
		\begin{assumption}[Minorisation condition] \label{ass:minor}
			Let $\fastFixedI_t$ be the process (\ref{fastF}) with semigroup $\{\SGfastsx{t}{x}\}_{t\geq 0}$ and transition probability
			denoted by $\SGfastsx{t}{x}(y,A) = \left(\SGfastsx{t}{x}\mathbb{I}_{A}\right)(y)$, where $A \in \mathcal{B}(\R^d)$ and $\mathcal{B}(\R^d)$ denotes the Borel sets of $\R^d$. For all fixed compact sets $\bar{U} \in \mathcal{B}(\R^d)$,  there exist $\eta,T>0$ independent of $x \in \R^n$ such that
			\begin{equation*}
				\SGfastsx{T}{x}(y,A) \geq \eta \lambda(A \cap \bar{U}) \quad \forall A\in \mathcal{B}(\R^d), y \in \bar{U}.
			\end{equation*}
			where $\lambda$ denotes the Lebesgue measure on $\R^d$.
		\end{assumption}

		\begin{assumption}[Lyapunov Condition]\label{ass:lyapunovMS}
			There is a function $V :\R^d \rightarrow [1,\infty)$, with $\lim_{y\rightarrow \infty}V(y) = \infty$, and real numbers $a,d\in (0,\infty)$, independent of $x$, such that
			\begin{equation*}
				\Lfastx{x} V(y) \leq -aV(y) + d,
			\end{equation*}
			where $\mathcal{L}^x$ is the generator for the process outlined in \eqref{eqn:gen}. 
			
		\end{assumption}
		Assumption \ref{ass:lyapunovMS} is exactly \cite[Assumption 2.4]{MATTINGLYSTUARTergodicity}. Assumption \ref{ass:minor} implies the conclusion of \cite[Lemma 2.3]{MATTINGLYSTUARTergodicity}, with $\nu(A)=\lambda(A \cap \bar{U})$.
		Now we write a modified Minorisation condition that implies Assumption \ref{ass:minor} in our setting.
		\begin{assumption}[Minorisation condition \RNum{1}] \label{ass:minor1}
			Let $\fastFixedI_t$ be the process (\ref{fastF}) with semigroup $\{\SGfastsx{t}{x}\}_{t\geq 0}$ and density $p_t^{x}(y,y')$
			such that $\SGfastsx{t}{x}(y,A) = \int_{A}p^x_t(y,y')dy' $. For all fixed compact sets $\bar{U} \in \mathcal{B}(\R^d)$, $y \in \bar{U}$ and $T>0$, there exists $\eta(y)>0$ independent of $x \in \R^n$ such that
			\begin{equation*}
				\inf_{x \in \R^n}\inf_{y' \in \bar{U}}p^x_T(y,y') \geq \eta(y).
			\end{equation*}
		\end{assumption}
		
		Indeed, if Assumption \ref{ass:minor1} holds, then we have, for each $A \in  \mathcal{B}(\R^d)$, $y \in \bar{U}$ and $x \in \R^n$
		\begin{equation*}
			\SGfastsx{T}{x}(y,A) = \int_{A}p^x_T(y,y')dy' \geq  \int_{A\cap \bar{U}}p^x_T(y,y')dy' \geq \eta(y) \lambda(A \cap \bar{U}).
		\end{equation*}
		Since $\bar{U}$ is a compact subset of $\R^d$, $\eta \coloneqq\inf_{y \in \bar{U}}\eta(y)>0$ and so we have that Assumption \ref{ass:minor1} implies Assumption \ref{ass:minor}.
		Hence, if we show that both Assumption \ref{ass:lyapunovMS} and Assumption \ref{ass:minor1} hold uniformly in $x\in \R^n$, then we are done. The Lyapunov condition stated in Assumption \ref{ass:lyapunovMS} is uniform in $x$ and follows directly from Assumption \ref{DriftAssump}, with $V(y) = |y|^k$ for any $k>0$, see the proof of Lemma \ref{prop:momentBounds}. The remaining problem is in verifying Assumption \ref{ass:minor1}. 
		
		We prove that Assumption \ref{ass:minor1} holds uniformly in $x$ by contradiction. Suppose that there exists a compact set $\bar{U} \in \mathcal{B}(\R^d)$, a point $y\in \bar{U}$, and a time $T>0$ (which from this point onwards we fix) such that $\inf_{x \in \R^n}\inf_{y' \in \bar{U}}p^x_T(y,y') = 0$.
		That is, there exists a sequence $(x_n)_{n=1}^\infty$ such that
		\begin{equation}\label{ptgoes0}
			\inf_{y' \in \bar{U}}p^{x_n}_T(y,y') \rightarrow 0, \quad \text{as $n\rightarrow\infty$}.
		\end{equation}
		We will arrive at a contradiction by first considering a process whose coefficients are the limit of a subsequence of the coefficients corresponding to the sequence of processes $\left(Y_{t}^{x_n,y}\right)^{\infty}_{n=1}$;  in particular we will show that the limiting process still satisfies uniform ellipticity and so has strictly positive density. This strict positivity will contrast the limit of the densities necessarily being not strictly positive, and we will argue that these two densities must agree; hence a contradiction.
		
		We first consider the sequence of drift coefficients $(g^{x_n}(z))_{n=1}^\infty$, as functions of $z \in \R^d$. This sequence is pointwise bounded due to Assumption \ref{ass:polGrowth}, \ref{ass:polGrowthDriftFast}, and equicontinuous when restricted to compact sets by local lipschitzianity. Hence, by Arzela-Ascoli (see e.g \cite{rudin1964principles}) we have a subsequence of drifts $(g^{x_{n_k}}(z))_{k=1}^\infty$ which converges locally uniformly to a function $g^*(z)$. Similarly, we take a further subsequence to obtain local uniform convergence of the sequence of diffusion coefficients $(a^{x_{n_k}}(z))_{k=1}^{\infty}$ to some $a^{*}(z)$. Note that $a^{*}(z)$ still satisfies Assumption \ref{UniformEllip}. Now, we consider the $\R^d$-valued process defined by
		\begin{equation*}
			d\mathcal{Y}^{*,y}_t = g^*(\mathcal{Y}^{*,y}_t)dt + \sqrt{2}a^*(\mathcal{Y}^{*,y}_t)d\tilde{W}_t
		\end{equation*}
		with corresponding semigroup $\{\mathcal{P}^*_t\}_{t\geq 0}$ and generator
		\begin{equation*}
			(\cL^*h)(y) = g^* (y) \cdot \nabla_y h(y) + a^*(y) a^*(y)^T : \text{Hess}_y h(y).
		\end{equation*}
		By the ellipticity of $a^{*}(z)$, the density $p^*_T(y, z) $ associated with the process $\mathcal{Y}^{*,y}_t $ at time $t = T$ is strictly positive everywhere; that is
		\begin{equation}\label{eqn:posEverywhere}
			p^*_T(y, z) >0,\quad \text{ for all } z \in \R^d.
		\end{equation}
		
		Now, we consider the sequence of the continuous densities $(p_T^{x_{n_k}}(y,z))_{k=1}^\infty$ (densities of the process $Y_T^{x_{n_k},y}$) as functions of $z \in \R^d$. This sequence is uniformly bounded and equicontinuous when restricted to compact sets (uniform boundedness follows from \cite[Theorem 7.3.3]{bogachev2015fokker} equicontinuity follows from the fact that the derivative is uniformly locally bounded using \cite[Theorem 6.4.5]{bogachev2015fokker}, where we use that the density is uniformly bounded, which itself is from \cite[Theorem 7.3.3]{bogachev2015fokker}). Hence, we take a subsequence that converges locally uniformly to some density $ \tilde{p}_T(y,z)$. By \eqref{ptgoes0}, and since $\bar{U}$ is compact, we have $\inf_{y' \in \bar{U}}\tilde{p}_T(y,y') =0$;  this infimum is attained at some $y_0 \in \bar{U}$, i.e $\tilde{p}_T(y,y_0) = 0$. By this and \eqref{eqn:posEverywhere}, we have a contradiction (and thus the proof is complete) if we can show $p^{*}_T(y,\cdot) = \tilde{p}_T(y,\cdot)$. To do this, we consider the semigroup associated with the density $\tilde{p}_T(y,\tilde{y})$. i.e $\tilde{\cP}_T f(y) = \int_{\R^d} f(\tilde{y})\tilde{p}_T(y,\tilde{y})d\tilde{y}$. For $f \in C^2_b(\R^d)$ we have
		
		\begin{equation}
			\label{eqn:SGconvergence}
			\begin{split}
				|\cP^*_T f(y) - P^{x_n}_T f(y)| &= \bigg \lvert\int_0^T \left(\cP^{*}_{T-s} \left( \cL^{x_n} - \cL^*\right)P^{x_n}_s f\right)(y) ds \bigg \rvert \\
				&\leq \sum_{i=1}^{d}\int_0^T \bigg \lvert\left(\cP^{*}_{T-s} \left( g^{x_n}_i - g^*_i\right)\partial_{y_i}P^{x_n}_s f\right)(y) \bigg \rvert ds \\&+ \sum_{i,j=1}^{d}\int_0^T\bigg \lvert \left(\cP^{*}_{T-s} \left( \diffusion^{x_n}_{ij} - \diffusion^*_{ij}\right)\partial_{y_i y_j}P^{x_n}_s f\right)(y)\bigg \lvert ds,
			\end{split}
		\end{equation}
  where $ \diffusion^*(y) = a^*(y) a^*(y)^T$.
		From  \cite[Theorem 6.1.7]{lorenzi2006analytical} (or also from our own results), we have that $\|\partial_{y_i}P^{x_n}_s f\|_\infty, \|\partial_{y_i y_j}P^{x_n}_s f\|_\infty < \infty$. Indeed, since the hypotheses of \cite[Theorem 6.1.7]{lorenzi2006analytical} hold for us uniformly in $x \in \R^n$ (for details of this verification, see the proof of Theorem \ref{thm:lorenziderest} above), we have  $$\|\partial_{y_i}P^{x_n}_s f\|_\infty+ \|\partial_{y_i y_j}P^{x_n}_s f\|_\infty < C,$$ for some $C>0$ independent of $n$. Then, by \eqref{eqn:SGconvergence}
		\begin{align*}
			|\cP^*_T f(y) - P^{x_n}_T f(y)| &\leq C\sum_{i=1}^{d}\int_0^T \left(\cP^{*}_{T-s} \big \lvert g^{x_n}_i - g^*_i\right)\big \rvert(y)  ds \\&+ C\sum_{i,j=1}^{d}\int_0^T \left(\cP^{*}_{T-s} \big \lvert \diffusion^{x_n}_{ij} - \diffusion^*_{ij}\big \lvert \right)(y)ds.
		\end{align*}
		By Assumption \ref{ass:polGrowth} \ref{ass:polGrowthDriftFast},\ref{ass:polGrowthDiffFast} and Lemma \ref{prop:momentBounds}, we can use the DCT to take the limit inside both the time integral and the integral implied by the semigroup. This yields
		\begin{equation*}
			|\cP^*_T f(y) - P^{x_n}_T f(y)| \rightarrow 0
		\end{equation*}
		as $n \rightarrow \infty$. Then $\cP^*_T f(y) = \tilde{\cP}_T f(y)$ (for every $f \in C_b^2$ and, by density, for every $f \in C_b$), and so in particular $0<p^{*}_T(y,y_0) = \tilde{p}_T(y,y_0)=0$. This is a contradiction. Thus, Assumption \ref{ass:minor1} holds uniformly in $x$.
		
		
		
		Hence, both Assumption \ref{ass:lyapunovMS} and Assumption \ref{ass:minor1} hold independent of $x$, and this entails that all constants in the proof of \cite[Theorem 2.5]{MATTINGLYSTUARTergodicity} are independent of $x$, and we have the result as required.
	\end{proof}
	
	\begin{proof}[Proof of Proposition \ref{prop:derivativeest}]
		We first prove \eqref{eqn:SGdecayConc}, before we explain how to adapt the proof to obtain \eqref{eqn:SGdecayConc2}.
		Fix $\psi\in \funcSpace{\pow{}{x}}{\pow{}{y}}$ and set $f_t^x(y) = \SGfastsx{t}{x}\psi^x(y)$. 
		Define the function 
		\begin{equation*}
			\Gamma(f_t^x) = \sum_{i=1}^d \lvert \partial_{y_i} f_t^x\rvert^2 + \gamma_1\sum_{i,j=1}^d \lvert \partial_{y_i}\partial_{y_j} f_t^x\rvert^2 +  \gamma_2\sum_{i,j,k=1}^d \lvert \partial_{y_i,y_j,y_k} f_t^x\rvert^2+ \gamma_3\sum_{i,j,k,\ell=1}^d \lvert \partial_{y_i,y_j,y_k,y_\ell} f_t^x\rvert^2. 
		\end{equation*}
		Here $\gamma_1,\gamma_2,\gamma_3$ are positive constants which will be chosen later in the proof. If we show that
		\begin{equation}\label{eq:Gronwallcond}
			\partial_s \SGfastsx{t-s}{x}\Gamma(f_s^x) \leq - \kappa \SGfastsx{t-s}{x}\Gamma(f_s^x) 
		\end{equation}
		then by Gronwall's inequality we have
		\begin{equation}\label{eq:aftergronwall}
			\SGfastsx{t-s}{x}\Gamma(f_s^x) \leq e^{-\kappa s}\SGfastsx{t}{x}\Gamma(f_0^x). 
		\end{equation}
		By Lemma \ref{prop:momentBounds}, we have
		\begin{equation*}
   			\SGfastsx{t}{x}(1+|x|^{\tilde{\pow{}{}}}+|y|^{\pow{}{}}) \leq 1+ |x|^{\tilde{\pow{}{}}}+|y|^{\pow{}{}}+ \frac{C'_{\pow{}{}}}{r'_{\pow{}{}}},
		\end{equation*}
for any $m, \tilde{m} >0$.
		
		Since $\Gamma(f^x_0)\leq 3(d+\gamma_1 d^2+\gamma_2 d^3+\gamma_3 d^4)\Vseminorm{\psi}{4,\pow{}{x},\pow{}{y}}^{2}\left(1+ |x|^{ 2\pow{}{x}}+|y|^{ 2\pow{}{y}}\right)$ and using the positivity of Markov semigroups we have
		\begin{equation*}
			\SGfastsx{t}{x}\Gamma(f^x_0)\leq 3 \Vseminorm{\psi}{4,\pow{}{x},\pow{}{y}}^{2}(d+\gamma_1 d^2+\gamma_2 d^3+\gamma_3 d^4)\left( 1+ |x|^{2\pow{}{x}}+|y|^{2\pow{}{y}} + \frac{C'_{2\pow{}{y}}}{r'_{ 2\pow{}{y}}}\right).
		\end{equation*}
		Therefore taking $s=t$ in \eqref{eq:aftergronwall} we have
		\begin{equation*}
			\Gamma(f_t^x) \leq  3 \Vseminorm{\psi}{4,\pow{}{x},\pow{}{y}}^{2}(d+\gamma_1 d^2+\gamma_2 d^3+\gamma_3 d^4)\left( 1+ |x|^{2\pow{}{x}}+|y|^{2\pow{}{y}} + \frac{C'_{ 2\pow{}{y}}}{r'_{ 2\pow{}{y}}}\right). 
		\end{equation*}
		Thus \eqref{eqn:SGdecayConc} holds for $K>0$.
		It remains to prove \eqref{eq:Gronwallcond}, note that \eqref{eq:Gronwallcond} follows (see \cite[Proposition 3.4]{crisanottobre} for details) provided there exists $\kappa>0$ with
		\begin{equation}\label{eq:gammanesest}
			(\partial_t-\cL^x)\Gamma(f_t^x)  \leq -\kappa \Gamma(f_t^x).
		\end{equation}
		We prove that \eqref{eq:gammanesest} holds by expanding each term in $\Gamma$ and using Young's inequality, ellipticity and Assumption \ref{ass:SGcond}. Since the left hand side of \eqref{eq:gammanesest} involves the generator $\cL^x$ applied to $\Gamma(f_t^x)$ we expect derivatives up to order $6$ of $f_t^x$ to appear. However the $6$-th order derivatives cancel and we see that the $5$-th order derivatives appear with a minus sign and hence can be bounded above by zero.
		
		Let us first consider the first order derivative terms
		\begin{align}\label{eq:expandinggammafirstorder}
			(\partial_t-\cL^x)(\lvert \partial_{y_i} f_t^x\rvert^2) & = 2 (\partial_{y_i} (\cL^xf_t^x))( \partial_{y_i} f_t^x) - 2 (\cL^x\partial_{y_i} f_t^x)( \partial_{y_i} f_t^x) -  2\sum_{j=1}^d (a^j(x),\nabla_y\partial_{y_i} f_t^x)^2.
		\end{align}
		Here $a^j(x)$ denotes the $j$-th column of $a(x)$.
		Observe that since the diffusion coefficient $a$ is independent of $y$ we can write
		\begin{equation}\label{eq:comgenwithonderivative}
			\partial_{y_i} \cL^xf_t^x -  \cL^x\partial_{y_i}f_t^x = (\partial_{y_i}g(x), \nabla_y f_t^x).
		\end{equation}
		Substituting \eqref{eq:comgenwithonderivative} into \eqref{eq:expandinggammafirstorder} we have
		\begin{align*}
			(\partial_t-\cL^x)(\lvert \partial_{y_i} f_t^x\rvert^2) & = 2 (\partial_{y_i}g(x), \nabla_y f_t^x)( \partial_{y_i} f_t^x) -   2\sum_{j=1}^d (a^j(x),\nabla_y\partial_{y_i} f_t^x)^2.
		\end{align*}
		Summing over $i$ we have
		\begin{align*}
			\sum_{i=1}^d(\partial_t-\cL^x)(\lvert \partial_{y_i} f_t^x\rvert^2) & = 2 \sum_{i=1}^d (\partial_{y_i}g(x), \nabla_y f_t^x)( \partial_{y_i} f_t^x) - 2\sum_{i=1}^d\sum_{j=1}^d (a^j(x),\nabla_y\partial_{y_i} f_t^x)^2.
		\end{align*}
		
		Next consider the second order derivatives 
		\begin{align*}
			(\partial_t-\cL^x)(\lvert \partial_{y_i}\partial_{y_j} f_t^x\rvert^2) & = 2 (\partial_{y_i}\partial_{y_j} \cL^xf_t^x)(\partial_{y_i}\partial_{y_j} f_t^x) - 2 (\cL^x\partial_{y_i}\partial_{y_j} f_t^x)( \partial_{y_i}\partial_{y_j} f_t^x)\\
			& - 2\sum_{k=1}^d(a^k,\nabla_y\partial_{y_i}\partial_{y_j} f_t^x)^2.
		\end{align*}
		Using that $\cL^x$ is given by \eqref{eqn:gen} we have
		\begin{equation}\label{eq:secondordercommutator}
			\partial_{y_i}\partial_{y_j} \cL^xf_t^x -  \cL^x\partial_{y_i}\partial_{y_j}f_t^x = (\partial_{y_i}\partial_{y_j}g(x), \nabla_y f_t^x) +(\partial_{y_i}g(x), \nabla_y \partial_{y_j}f_t^x)+(\partial_{y_j}g(x), \partial_{y_i}\nabla_y f_t^x) .
		\end{equation}
		Here $\partial_{y_i}\nabla_y$ is meant component wise, i.e. $\partial_{y_i}\nabla_y f_t^x$ is a vector whose $j$-th entry is $\partial_{y_i}\partial_{y_j}f_t^x$. 
		Note that
		\begin{equation*}
			\sum_{j=1}^d (\partial_{y_j}g(x), \partial_{y_i}\nabla_y f_t^x )( \partial_{y_i}\partial_{y_j} f_t^x) = (\nabla_y\partial_{y_i}f_t^x)^T(\nabla_yg(x))(\nabla_y\partial_{y_i}f_t^x).
		\end{equation*}
		Then using Young's inequality and the ellipticity assumption, Assumption \ref{UniformEllip}, we have for any $\nu_0>0$ which will be chosen later
		\begin{align*}
			\sum_{i,j=1}^d(\partial_t-\cL^x)(\lvert \partial_{y_i}\partial_{y_j} f_t^x\rvert^2)  	&\leq  \sum_{i,j=1}^d\nu_0(\partial_{y_i}\partial_{y_j}g(x),\nabla_y f_t^x)^2+\nu_0^{-1}( \partial_{y_i}\partial_{y_j} f_t^x)^2 \\
			&+4\sum_{i=1}^d(\nabla_y\partial_{y_i}f_t^x)^T(\nabla_yg(x))(\nabla_y\partial_{y_i}f_t^x)-2\lambda_-\sum_{i,j=1}^d(\nabla_y\partial_{y_i}\partial_{y_j} f_t^x)^2.
		\end{align*}
		Similarly for the third order derivatives,
		\begin{align*}\label{eq:expandinggammathirdorder}
			&\sum_{i,j,k=1}^d(\partial_t-\cL^x)(\lvert \partial_{y_i,y_j,y_k} f_t^x\rvert^2) \\
			& = \sum_{i,j,k=1}^d2 ([\partial_{y_i,y_j,y_k}, \cL^x]f_t^x)(\partial_{y_i,y_j,y_k} f_t^x) - 2\sum_{i,j,k=1}^d\sum_{\ell=1}^d(a^\ell,\nabla_y\partial_{y_i,y_j,y_k} f_t^x)^2\\
			&=2 \sum_{i,j,k=1}^d\left((\partial_{y_i,y_j,y_k}g(x),\nabla_y f_t^x)+3(\partial_{y_i,y_j}g(x), \nabla_y \partial_{y_k}f_t^x) + 3(\partial_{y_i}g(x), \nabla_y \partial_{y_j,y_k}f_t^x)\right)(\partial_{y_i,y_j,y_k} f_t^x) \\
			&-2\sum_{i,j,k,\ell=1}^d (a^\ell,\nabla_y\partial_{y_i,y_j,y_k} f_t^x)^2.
		\end{align*}
		Using Young's inequality and ellipticity we have for any $\nu_1,\nu_2>0$ to be chosen later
		\begin{align*}
			&\sum_{i,j,k=1}^d(\partial_t-\cL^x)(\lvert \partial_{y_i,y_j,y_k} f_t^x\rvert^2)  \\
			&\leq  \sum_{i,j,k=1}^d\nu_1(\partial_{y_i,y_j,y_k}g(x), \nabla_y f_t^x)^2+(\nu_1^{-1}+3\nu_2^{-1})(\partial_{y_i,y_j,y_k} f_t^x)^2+3\nu_2(\partial_{y_i,y_j}g(x), \nabla_y \partial_{y_k}f_t^x)^2 \\
			&+\sum_{j,k=1}^d 6(\nabla_y\partial_{y_j,y_k} f_t^x)^T(\nabla_yg(x))(\nabla_y \partial_{y_j,y_k}f_t^x) -2\lambda_-\sum_{i,j,k=1}^d (\nabla_y\partial_{y_i,y_j,y_k} f_t^x)^2.
		\end{align*}
		Similarly,
		\begin{align*}
			&\sum_{i,j,k,\ell=1}^d(\partial_t-\cL^x)(\lvert \partial_{y_i,y_j,y_k,y_\ell} f_t^x\rvert^2)  = \sum_{i,j,k,\ell=1}^d2 ([\partial_{y_i,y_j,y_k,y_\ell}, \cL^x]f_t^x)(\partial_{y_i,y_j,y_k,y_\ell} f_t^x)\\
			&-2 \sum_{i,j,k,\ell,p=1}^d (a^p,\nabla_y\partial_{y_i,y_j,y_k,y_\ell} f_t^x)^2
			\\
			&=2 \sum_{i,j,k,\ell=1}^d\Big((\partial_{y_i,y_j,y_k,y_\ell}g(x), \nabla_y f_t^x)+4(\partial_{y_i,y_j,y_k}g(x), \nabla_y \partial_{y_\ell}f_t^x) + 6(\partial_{y_i,y_j}g(x), \nabla_y \partial_{y_k,y_\ell}f_t^x)\\
			&+ 4\partial_{y_i}g(x), \nabla_y \partial_{y_j,y_k,y_\ell}f_t^x\Big)(\partial_{y_i,y_j,y_k,y_\ell} f_t^x) -2\sum_{i,j,k,\ell,p=1}^d (a^p,\nabla_y\partial_{y_i,y_j,y_k,y_\ell} f_t^x)^2
		\end{align*}
		Using Young's inequality and ellipticity once more we have
		\begin{align*}
			&\sum_{i,j,k,\ell=1}^d(\partial_t-\cL^x)(\lvert \partial_{y_i,y_j,y_k,y_\ell} f_t^x\rvert^2) \\
			&\leq  \sum_{i,j,k,\ell=1}^d\Big(\nu_3(\partial_{y_i,y_j,y_k,y_\ell}g(x)\cdot \nabla_y f_t^x)^2+(\nu_3^{-1}+4\nu_4^{-1}+6\nu_5^{-1})(\partial_{y_i,y_j,y_k,y_\ell} f_t^x)^2\\
			&+4\nu_4(\partial_{y_i,y_j,y_k}g(x)\cdot \nabla_y \partial_{y_\ell}f_t^x)^2+ 6\nu_5(\partial_{y_i,y_j}g(x)\cdot \nabla_y \partial_{y_k,y_\ell}f_t^x)^2\Big)\\
			&+ 8 \sum_{i,j,k,\ell=1}^d(\partial_{y_i}g(x)\cdot \nabla_y \partial_{y_j,y_k,y_\ell}f_t^x)(\partial_{y_i,y_j,y_k,y_\ell} f_t^x) -2\lambda_-\sum_{i,j,k,\ell=1}^d (\nabla_y\partial_{y_i,y_j,y_k,y_\ell} f_t^x)^2
		\end{align*}
	Combining these terms we have \useshortskip \begin{equation*}\begin{split}
	    &(\partial_t-\cL^x)\Gamma(f_t^x) \leq 2 ( \nabla_y f_t^x)^T\nabla_yg(x)( \nabla_y f_t^x) - 2\lambda_- \sum_{i=1}^d (\nabla_y\partial_{y_i} f_t^x)^2 \\
			&+\sum_{i,j=1}^d\gamma_1\nu_0(\partial_{y_i}\partial_{y_j}g(x), \nabla_y f_t^x)^2+\gamma_1\nu_0^{-1}( \partial_{y_i}\partial_{y_j} f_t^x)^2 \\
			&+4\sum_{i=1}^d\gamma_1(\nabla_y\partial_{y_i}f_t^x)^T(\nabla_yg(x))(\nabla_y\partial_{y_i}f_t^x)-2\gamma_1\lambda_-\sum_{i,j=1}^d(\nabla_y\partial_{y_i}\partial_{y_j} f_t^x)^2\\
			&+\sum_{i,j,k=1}^d\gamma_2\nu_1(\partial_{y_i,y_j,y_k}g(x), \nabla_y f_t^x)^2+\gamma_2(\nu_1^{-1}+3\nu_2^{-1})(\partial_{y_i,y_j,y_k} f_t^x)^2+3\gamma_2\nu_2(\partial_{y_i,y_j}g(x), \nabla_y \partial_{y_k}f_t^x)^2 \\
			&+\sum_{j,k=1}^d \gamma_26(\nabla_y\partial_{y_j,y_k} f_t^x)^T(\nabla_yg(x))(\nabla_y \partial_{y_j,y_k}f_t^x) -2\gamma_2\lambda_-\sum_{i,j,k=1}^d (\nabla_y\partial_{y_i,y_j,y_k} f_t^x)^2\\
			&+\sum_{i,j,k,\ell=1}^d\Big(\gamma_3\nu_3(\partial_{y_i,y_j,y_k,y_\ell}g(x), \nabla_y f_t^x)^2+\gamma_3(\nu_3^{-1}+4\nu_4^{-1}+6\nu_5^{-1})(\partial_{y_i,y_j,y_k,y_\ell} f_t^x)^2\\
			&+4\nu_4\gamma_3(\partial_{y_i,y_j,y_k}g(x), \nabla_y \partial_{y_\ell}f_t^x)^2+ 6\gamma_3\nu_5(\partial_{y_i,y_j}g(x), \nabla_y \partial_{y_k,y_\ell}f_t^x)^2\Big)\\
			&+ 8 \sum_{i,j,k,\ell=1}^d\gamma_3(\partial_{y_i}g(x),\nabla_y \partial_{y_j,y_k,y_\ell}f_t^x)(\partial_{y_i,y_j,y_k,y_\ell} f_t^x) -2\gamma_3\lambda_-\sum_{i,j,k,\ell=1}^d (\nabla_y\partial_{y_i,y_j,y_k,y_\ell} f_t^x)^2.
	\end{split}
		\end{equation*}
		Rearranging and bounding the fifth order derivative terms by zero we have
		\begin{align*}
			&(\partial_t-\cL^x)\Gamma(f_t^x) \leq 2 ( \nabla_y f_t^x)^T\nabla_yg(x)( \nabla_y f_t^x)+\sum_{i,j=1}^d\gamma_1\nu_0(\partial_{y_i}\partial_{y_j}g(x), \nabla_y f_t^x)^2 \\
			&+\sum_{i,j,k=1}^d\gamma_2\nu_1(\partial_{y_i,y_j,y_k}g(x), \nabla_y f_t^x)^2+\sum_{i,j,k,\ell=1}^d\gamma_3\nu_3(\partial_{y_i,y_j,y_k,y_\ell}g(x), \nabla_y f_t^x)^2 \\
			&- 2\lambda_- \sum_{i=1}^d (\nabla_y\partial_{y_i} f_t^x)^2+\sum_{i,j=1}^d\gamma_1\nu_0^{-1}( \partial_{y_i}\partial_{y_j} f_t^x)^2 +4\sum_{i=1}^d\gamma_1(\nabla_y\partial_{y_i}f_t^x)^T(\nabla_yg(x))(\nabla_y\partial_{y_i}f_t^x)\\
			&+3\sum_{i,j,k=1}^d\gamma_2\nu_2(\partial_{y_i,y_j}g(x), \nabla_y \partial_{y_k}f_t^x)^2+\sum_{i,j,k,\ell=1}^d4\nu_4\gamma_3(\partial_{y_i,y_j,y_k}g(x), \nabla_y \partial_{y_\ell}f_t^x)^2\\
			&-2\gamma_1\lambda_-\sum_{i,j=1}^d(\nabla_y\partial_{y_i}\partial_{y_j} f_t^x)^2+\sum_{i,j,k=1}^d\gamma_2(\nu_1^{-1}+3\nu_2^{-1})(\partial_{y_i,y_j,y_k} f_t^x)^2\\
			&+\sum_{j,k=1}^d \gamma_26(\nabla_y\partial_{y_j,y_k} f_t^x)^T(\nabla_yg(x))(\nabla_y \partial_{y_j,y_k}f_t^x)+ \sum_{i,j,k,\ell=1}^d6\gamma_3\nu_5(\partial_{y_i,y_j}g(x), \nabla_y \partial_{y_k,y_\ell}f_t^x)^2 \\
			&-2\gamma_2\lambda_-\sum_{i,j,k=1}^d (\nabla_y\partial_{y_i,y_j,y_k} f_t^x)^2+\gamma_3(\nu_3^{-1}+4\nu_4^{-1}+6\nu_5^{-1})(\partial_{y_i,y_j,y_k,y_\ell} f_t^x)^2\\
			&+ 8 \sum_{i,j,k,\ell=1}^d\gamma_3(\nabla_y\partial_{y_j,y_k,y_\ell} f_t^x)^T(\nabla_{y}g(x))( \nabla_y \partial_{y_j,y_k,y_\ell}f_t^x) .
		\end{align*}
		By making the choice
		\begin{align*}
			&\nu_0=\zeta_1/\gamma_1,  \nu_1=\zeta_2/\gamma_2, \nu_3=\zeta_3/\gamma_3\\
			&\nu_2=\frac{2\gamma_1\zeta_1}{3\gamma_2}, \nu_4=\frac{\gamma_1}{2\gamma_3}\zeta_2, \nu_5=\frac{\gamma_2}{2\gamma_3}\zeta_1,\\
			&\gamma_1=\sqrt{2\zeta_1\lambda_{-}},\quad  \gamma_2=\sqrt{\frac{2\gamma_1\lambda_{-}}{\zeta_2^{-1}+\frac{9}{2}\gamma_1^{-1}\zeta_1^{-1}}} ,\quad \gamma_3=\sqrt{\frac{2\gamma_2\lambda_{-}}{\zeta_3^{-1}+8\gamma_1^{-1}\zeta_2^{-1}+12\gamma_2^{-1}\zeta_1^{-1}}}.
		\end{align*}
		we obtain
		\begin{equation}\label{eqn:finalderivativeestimate}\begin{split}
				&(\partial_t-\cL^x)\Gamma(f_t^x) \leq 2 ( \nabla_y f_t^x)^T\nabla_yg(x)( \nabla_y f_t^x)+\sum_{i,j=1}^d\zeta_1(\partial_{y_i}\partial_{y_j}g(x), \nabla_y f_t^x)^2 \\
				&+\sum_{i,j,k=1}^d\zeta_2(\partial_{y_i,y_j,y_k}g(x), \nabla_y f_t^x)^2+\sum_{i,j,k,\ell=1}^d\zeta_3(\partial_{y_i,y_j,y_k,y_\ell}g(x), \nabla_y f_t^x)^2 \\
				&+4\sum_{i=1}^d\gamma_1(\nabla_y\partial_{y_i}f_t^x)^T(\nabla_yg(x))(\nabla_y\partial_{y_i}f_t^x)\\
				&+\sum_{i,j,k=1}^d2\gamma_1\zeta_1(\partial_{y_i,y_j}g(x), \nabla_y \partial_{y_k}f_t^x)^2+\sum_{i,j,k,\ell=1}^d2\zeta_2\gamma_1(\partial_{y_i,y_j,y_k}g(x), \nabla_y \partial_{y_\ell}f_t^x)^2\\
				&+\sum_{j,k=1}^d \gamma_26(\nabla_y\partial_{y_j,y_k} f_t^x)^T(\nabla_yg(x))(\nabla_y \partial_{y_j,y_k}f_t^x)+ \sum_{i,j,k,\ell=1}^d3\gamma_2\zeta_1(\partial_{y_i,y_j}g(x), \nabla_y \partial_{y_k,y_\ell}f_t^x)^2 \\
				&+ 8 \sum_{i,j,k,\ell=1}^d\gamma_3(\nabla_y\partial_{y_j,y_k,y_\ell} f_t^x)^T(\nabla_{y}g(x))( \nabla_y \partial_{y_j,y_k,y_\ell}f_t^x) .
			\end{split}
		\end{equation}
		Now by applying \eqref{eq:driftcondition} we have
		\begin{align*}
			(\partial_t-\cL^x)\Gamma(f_t^x) &\leq -\cb\lvert\nabla_y f_t^x\rvert^2-2\cb \sum_{i=1}^d\gamma_1\lvert\nabla_y\partial_{y_i}f_t^x\rvert^2-3\cb\sum_{j,k=1}^d \gamma_2\lvert\nabla_y\partial_{y_j,y_k} f_t^x\rvert^2\\
			&-4\cb \sum_{i,j,k,\ell=1}^d\gamma_3\lvert\nabla_y\partial_{y_j,y_k,y_\ell} f_t^x\rvert^2 \\
			&\leq -\cb\Gamma(f_t^x) .
		\end{align*}
		Therefore \eqref{eq:Gronwallcond} holds.
		Let us now explain how to prove \eqref{eqn:SGdecayConc2}. {The proof follows the same argument as above with the difference that we set $\gamma_2=\gamma_3=0$.  in the above argument. That is we replace $\Gamma(f_t^x)$ by
  \begin{equation*}
    \Gamma(f_t^x) = \sum_{i=1}^d \lvert \partial_{y_i} f_t^x\rvert^2 + \gamma_1\sum_{i,j=1}^d \lvert \partial_{y_i}\partial_{y_j} f_t^x\rvert^2 .
  \end{equation*}
  Following the previous argument with $\zeta_2$, and $\zeta_3$ set to zero leads to \eqref{eqn:finalderivativeestimate} }becoming
		\begin{equation*}
			\begin{split}
				&(\partial_t-\cL^x)\Gamma(f_t^x) \leq 2 ( \nabla_y f_t^x)^T\nabla_yg(x)( \nabla_y f_t^x)+\sum_{i,j=1}^d\zeta_1(\partial_{y_i}\partial_{y_j}g(x), \nabla_y f_t^x)^2 \\
				&+4\sum_{i=1}^d\gamma_1(\nabla_y\partial_{y_i}f_t^x)^T(\nabla_yg(x))(\nabla_y\partial_{y_i}f_t^x)\\
				&+\sum_{i,j,k=1}^d2\gamma_1\zeta_1(\partial_{y_i,y_j}g(x), \nabla_y \partial_{y_k}f_t^x)^2.
			\end{split}
		\end{equation*}
		Therefore, applying \eqref{eq:driftcondition2},
  we have that $(\partial_t-\cL^x)\Gamma(f_t^x)\leq -\kappa \Gamma(f_t^x)$ and we obtain the $\Gamma(f_t^x)$ decays exponentially by setting $s=t$ in \eqref{eq:aftergronwall}. We can express this in terms of the derivatives of $f_t^x$ by using that 
		for any $\pow{}{x},\pow{}{y}\geq 0$ and $\psi\in C^{0,2}(\R^n\times\R^d)$ with $\Vseminorm{\psi}{2,\pow{}{x},\pow{}{y}}<\infty$ and for all $x \in \R^n, y \in \R^d$
		we have
		\begin{equation}\label{eqn:SGderfor6.2}
			\begin{split}
				&\sum_{i=1}^d \lvert \partial_{y_i} \SGfastsx{t}{x}\psi^x(y)\rvert^2 + \sum_{i,j=1}^d \lvert \partial_{y_i}\partial_{y_j} \SGfastsx{t}{x}\psi^x(y)\rvert^2\\ &\leq 3\max\{(\zeta_1\lambda_-)^{-1/2} ,1\}\left( d+\sqrt{\zeta_1\lambda_-}d^2\right)\Vseminorm{\psi}{2,\pow{}{x},\pow{}{y}}^{2}e^{-\kappa t}\left(1+ |x|^{2\pow{}{x}}+|y|^{2\pow{}{y}} +\frac{C'_{2\pow{}{y}}}{r'_{2\pow{}{y}}}\right),
			\end{split}
		\end{equation}
		and so \eqref{eqn:SGdecayConc2} holds. {Here the term $\max \{(\zeta_1\lambda_{-})^{-\frac{1}{2}},1\}$ arises inverse of the smallest coefficient in $\Gamma$, i.e. $(\min\{1,\gamma_1\})^{-1}$ and recalling that $\gamma_1=\sqrt{\kappa \lambda_{-}}$.} In particular, we have
\begin{equation}\label{eqn:SGderfor6.2 1Norm}
			\begin{split}
				&\sum_{i=1}^d \lvert \partial_{y_i} \SGfastsx{t}{x}\psi^x(y)\rvert , \sum_{i,j=1}^d \lvert \partial_{y_i}\partial_{y_j} \SGfastsx{t}{x}\psi^x(y)\rvert\\ &\leq \sqrt{3}d\max\{(\zeta_1\lambda_-)^{-1/4} ,1\}\left( d+\sqrt{\zeta_1\lambda_-}d^2\right)^{1/2}\Vseminorm{\psi}{2,\pow{}{x},\pow{}{y}}e^{-\kappa t/2}\left(1+ |x|^{\pow{}{x}}+|y|^{\pow{}{y}} +\sqrt{\frac{C'_{2\pow{}{y}}}{r'_{2\pow{}{y}}}}\right),
			\end{split}
		\end{equation}
  and so \eqref{eqn:SGdecayConc2} holds.

		
	\end{proof}
	
	\subsection{Proofs of Section \ref{sec:psneqn}}\label{app:proofspsneqn} 	
	Before we prove Lemma \ref{lemma:probRep} we give two preliminary lemmas. The lemma below is a specific case of \cite[Note $3.1$]{michelaLongTimeAsymptotics}. 
	\begin{lemma}\label{Aineq}
		Let $([T,\infty), \lambda)$ be the truncated real line (for some $T \in \R$) equipped with the Lebesgue measure, $(\mathbb{R}^d, \mu)$ be Euclidean space equipped with some measure $\mu$ and $f: [0,\infty)\times \mathbb{R}^d \rightarrow \R$ a positive function. Suppose
		\begin{equation*}
			F(y) \coloneqq \int_T^\infty f(s,y) ds  < \infty \quad \mu\text{-a.e}
		\end{equation*}
		and \begin{equation*}
			\int  \left(\int_T^\infty  f(s,y)ds \right)^m d\mu(y)< \infty \text{ for some } m>1.
		\end{equation*}
		Then
		\begin{equation*}
			\int_{\R^d}  \left(\int_T^\infty f(s,y)ds\right)^m d\mu(y)\leq \left( \int_T^\infty  \left( \int \left(f(s,y)\right)^m d\mu(y) \right)^\frac{1}{m} \right)^m ds
		\end{equation*}
	\end{lemma}
	\begin{proof}
		We refer the reader to \cite[Note $3.1$]{michelaLongTimeAsymptotics}.
	\end{proof}
	In Lemma \ref{lemma:psneqn} we will prove that under certain conditions the solution to the Poisson equation, \eqref{poisprobgeneral}, is given by \eqref{ft1rep}.
	The following is a slight generalisation of \cite[Corollary $3.10$]{cattiaux2011central}, but less general than Proposition $3.21$ from the same paper.
	
	\begin{lemma}\label{lemma:psneqn}
		Let $\phi:\R^n\times\R^d \to \R$ be such that for each $x\in\R^n$ $y\mapsto\phi^x(y)\in L^2(\mu)$. Let Assumption \ref{ass:polGrowth} \ref{ass:polGrowthDriftFast}, \ref{ass:polGrowthDiffFast} and Assumption \ref{UniformEllip} hold. If
		\begin{equation*}
			\intPos \|\SGfastsx{s}{x} \phi-\bar{\phi}\|_{L^2(\mu)}ds < \infty,
		\end{equation*}
		then the function $u$ defined in \eqref{ft1rep} is a classical solution to the Poisson equation \eqref{poisprobgeneral}.
	\end{lemma}
	\begin{proof}
		Fix $x\in \R^n$. As argued in \cite[Section $3$]{cattiaux2011central}, it is sufficient to show that for $u^{x,T}$ defined by \eqref{eqn:truncPsn} we have $u^{x,T} \to u^x$ as $T\to \infty$ in $L^2(\mu^x)$, as this implies that $u^x$ is a solution of $\cL^x u^x(y) = \phi^x-\bar{\phi}(x)$. Using Lemma \ref{Aineq},
		\begin{align*}
			\|u^{x,T} - u^x\|^2_{L^2(\mu^x)} &= \left[ \int  \left(\int^\infty_T (\SGfastsx{s}{x} \phi^x)(y)-\bar{\phi}(x) ds\right)^2 \mu^x(dy)\right]^{\frac{1}{2}} \\
			&\leq \int_T^\infty \left(\int |\SGfastsx{s}{x} \phi^x(y)-\bar{\phi}(x)|^2  \mu^x(dy)\right)^\frac{1}{2}ds \\
			&= \int_T^\infty \|\SGfastsx{s}{x} \phi^x-\bar{\phi}(x)\|_{L^2(\mu^x)}ds.
		\end{align*}
		Taking $T \rightarrow \infty$ concludes the proof. Smoothness follows by ellipticity. Indeed, by considering the PDE $\cL u^x = \phi^x-\bar{\phi}(x)$ restricted to a ball, one can use standard results to obtain smoothness. This is done in \cite[Theorem 1]{pardoux2001}, where we make the observation that since we are restricting to a ball, our coefficients satisfy their conditions.
		Assume, conversely, that $\cL^x u^x = \phi^x-\bar{\phi}(x)$ with $\int u^xd\mu^x=0$. Then
		\begin{equation*}
			\SGfastsx{t}{x}u^x - u^x = \int^t_0 \partial_s\SGfastsx{s}{x}u^x ds = \int^t_0 \cL^x \SGfastsx{s}{x} u^x ds = \int_0^t \SGfastsx{s}{x} \cL^x  u^x ds = \int_0^t \SGfastsx{s}{x}\phi^x-\bar{\phi}(x) ds.
		\end{equation*}
		From Proposition \ref{lemma:IntAvg}, we have that $\lim_{t \rightarrow \infty}\SGfastsx{t}{x}u^x = \int u^x d\mu^x =0$ then we have
		\begin{equation*}
			u^x = -\int_0^\infty \left(\SGfastsx{s}{x} \phi^x-\bar{\phi}(x)\right) ds,
		\end{equation*}
		so that the solution is unique in the class $L^2(\mu^x)$ when restricted to mean-zero functions.
	\end{proof}
	An alternative proof to Lemma \ref{lemma:psneqn}, can be obtained by using the methods in  \cite{cattiaux2011central}, in particular by combining  \cite[Section 5.5]{cattiaux2011central}, \cite[Remark 5.7]{cattiaux2011central} and  \cite[Proposition 3.21]{cattiaux2011central}. 
	
	\begin{proof}[Proof of Lemma \ref{lemma:probRep}]
		It is sufficient to show that the conditions of Lemma \ref{lemma:psneqn} hold. To see this, we use (\ref{eqn:SGconv}) to write
		\begin{align*}
			\intPos \|\SGfastsx{s}{x} \phi^x-\bar{\phi}(x)\|_{L^2(\mu^x)}ds &=
			\int^\infty_0\left[ \int   |\left(\SGfastsx{s}{x} \phi^x\right)(y)-\bar{\phi}(x)|^2 \mu^x(dy)\right]^{\frac{1}{2}}ds \\
			&\leq \int^\infty_0\left[ \int  Ce^{-2cs}(1+ |x|^{2\pow{}{x}} + |y|^{2\pow{}{y}})  \mu^x(dy)\right]^{\frac{1}{2}}ds.
		\end{align*}
		Then we use (\ref{eqn:momBoundsYinv}) from Lemma \ref{prop:momentBounds}
		\begin{equation*}
			\intPos \|\SGfastsx{s}{x} \phi^x-\bar{\phi}^x\|_{L^2(\mu^x)}ds \leq C(1+ |x|^{\pow{}{x}}) \int^\infty_0 e^{-cs} ds < \infty.
		\end{equation*}
		Hence Lemma \ref{lemma:psneqn} holds and we conclude.
		
	\end{proof}

	\begin{proof}[Proof of Lemma \ref{lemma:IntAvgUB}]
		This is a method of proof seen in the literature elsewhere (see e.g \cite{dragoni2010ergodicity}) so we only sketch it. We write
		\begin{align*}
			|\left(\SGfastsx{t}{x}\phi^x\right)(y) -\left(\SGfastsx{s}{x}\phi^x\right)(y)| =\left\lvert \int_s^t \partial_u \left(\SGfastsx{u}{x}\phi^x\right)(y)du\right\rvert
			= \left\lvert\int_s^t \left(\Lfastx{x}\SGfastsx{u}{x}\phi^x\right)(y)du\right\rvert.
		\end{align*}
		Now we use Assumption \ref{ass:polGrowth} (\ref{ass:polGrowthDriftFast},\ref{ass:polGrowthDiffFast}) and Proposition \ref{prop:derivativeest} to have
		\begin{align*}
			|\left(\SGfastsx{t}{x}\phi^x\right)(y) -\left(\SGfastsx{s}{x}\phi^x\right)(y)| &\leq C \Vseminorm{\phi}{2,\pow{}{x},\pow{}{y}}\int_s^t \left( 1 + |y|^{\pow{\drift}{y}}\right)\left(1 + |x|^{\pow{}{x}} + |y|^{\pow{}{y}}\right)e^{-cu} du \\
			&+ C \Vseminorm{\phi}{2,\pow{}{x},\pow{}{y}}\int_s^t \left(1 + |x|^{\pow{}{x}} + |y|^{\pow{}{y}}\right)e^{-cu} du\\
			&\leq C\Vseminorm{\phi}{2,\pow{}{x},\pow{}{y}}(1 + |y|^{\pow{\drift}{y}+\pow{}{y}}+|y|^{2\pow{g}{y}} + |x|^{2\pow{}{x}})(e^{-cs}-e^{-ct}).
		\end{align*}Taking $t \rightarrow \infty$ on both sides concludes \eqref{eqn:SGconvUB}. Here we use Assumption \ref{DriftAssump} to apply Lemma \ref{lemma:IntAvg} so that $\left(\SGfastsx{t}{x}\phi^x\right)(y) \rightarrow \mu^x(\phi^x)$. Integrating over $s$ yields \eqref{eqn:IntAvg1}.
	\end{proof}


	\begin{proof}[Proof of Proposition \ref{lemma:IntAvgDerDer}]
		This proof will follow a similar outline to the proof of Proposition \ref{lemma:IntAvgDerRepresentation}, and in the same way we will consider increments $he_j$ where $e_j$ is the $j^{th}$ element of the standard basis of $\R^n$.
		We begin by showing (\ref{eqn:intAvgDerFast}) and some preliminary bounds, before outlining the steps to proving the continuity and the second $x$ derivative representations \eqref{eqn:repSGderder} and \eqref{eqn:invMeasDerDer}.
		We use \eqref{eqn:repSG} to write
		\begin{equation}\label{eqn:term0}
			\fastDer{j} \left(\partial_{x_i} \SGfastsx{t}{x} \phi^{x}(y)\right) =  \fastDer{j}\SGfastsx{t}{x}\partial_{x_i}  \phi^{x}(y) + \int_0^{t}\fastDer{j} \left[\SGfastsx{t-s}{x} \LfastxDeri \SGfastsx{s}{x}\phi^{x}\right](y) ds.
		\end{equation}
		Similarly,
		\begin{equation}\label{eqn:termdyy}
			\fastDer{k}\fastDer{j} \left(\partial_{x_i} \SGfastsx{t}{x} \phi^{x}(y)\right) =  \fastDer{k}\fastDer{j}\SGfastsx{t}{x}\partial_{x_i}  \phi^{x}(y) + \int_0^{t}\fastDer{k}\fastDer{j} \left[\SGfastsx{t-s}{x} \LfastxDeri \SGfastsx{s}{x}\phi^{x}\right](y) ds.
		\end{equation}
		Addressing the first terms of \eqref{eqn:term0} and \eqref{eqn:termdyy} we have that, from Proposition \ref{prop:derivativeest}, since $\Vseminorm{\partial_{x_i} \phi}{2,\pow{}{x},\pow{}{y}}\leq \Vnorm{ \phi}{4,\pow{}{x},\pow{}{y}}$,
		\begin{equation}\label{eqn:term1}
			\left| \fastDer{j}\SGfastsx{t}{x}\partial_{x_i}  \phi^{x}(y) \right| + \left| \fastDerDer{j}{k}\SGfastsx{t}{x}\partial_{x_i}  \phi^{x}(y) \right| \leq C\Vnorm{\phi}{4,\pow{}{x},\pow{}{y}}e^{-ct}(1+ |y|^{\pow{}{y}}+|x|^{\pow{}{x}}).
		\end{equation}
		Here $C$ and $c$ are constants which may change line by line.
		Addressing the integral terms in \eqref{eqn:term0} and \eqref{eqn:termdyy} we have, using Proposition \ref{prop:derivativeest} (we bound $ \fastDer{j}
		\LfastxDeri \SGfastsx{s}{x}\phi^{x}$ and $\fastDerDer{j}{k}
		\LfastxDeri \SGfastsx{s}{x}\phi^{x} $ by taking the limit of \eqref{eqn:elemBound3} and \eqref{eqn:elemBound4} as $h \rightarrow 0$) again, that
		\begin{equation}\label{eqn:term2}
			\begin{split}
				&\left| \int_0^{t}\fastDer{j} \left[\SGfastsx{t-s}{x} \LfastxDeri \SGfastsx{s}{x}\phi^{x}\right](y) ds \right|\\ &\leq \int_0^{t}C\Vseminorm{\phi}{4,\pow{}{x},\pow{}{y}}e^{-c(t-s)}e^{-cs}(1+ |y|^{\max\{2\pow{\drift}{y},\pow{\drift}{y}
					+\pow{}{y}\}}+|x|^{2\pow{}{x}})ds \\
				&\leq C\Vseminorm{\phi}{4,\pow{}{x},\pow{}{y}} te^{-ct}(1+ |y|^{M^g_y}+|x|^{2\pow{}{x}}).
			\end{split}
		\end{equation}
		as well as
		\begin{equation}\label{eqn:term3}
			\left| \int_0^{t}\fastDerDer{j}{k} \left[\SGfastsx{t-s}{x} \LfastxDeri \SGfastsx{s}{x}\phi^{x}\right](y) ds \right| \leq C\Vseminorm{\phi}{4,\pow{}{x},\pow{}{y}} te^{-ct} (1+ |y|^{M^g_y}+|x|^{2\pow{}{x}}).
		\end{equation}
		Now we put \eqref{eqn:term1}, \eqref{eqn:term2} and \eqref{eqn:term3} together with \eqref{eqn:term0} and \eqref{eqn:termdyy} to write 
		\begin{equation*}
			\left|\fastDer{j} \left(\partial_{x_i} \SGfastsx{t}{x} \phi^{x}(y)\right) \right| + \left|\fastDerDer{j}{k} \left(\partial_{x_i} \SGfastsx{t}{x} \phi^{x}(y)\right) \right|\leq C (1+t)\Vnorm{\phi}{4,\pow{}{x},\pow{}{y}}e^{-ct} (1+ |y|^{M^g_y}+|x|^{2\pow{}{x}}),
		\end{equation*}
		which implies the required inequality \eqref{eqn:intAvgDerFast}. Now we show that $\fastDer{j} \left(\partial_{x_i} \SGfastsx{t}{x} \phi^{x}(y)\right)$ and $\fastDer{k}\fastDer{j} \left(\partial_{x_i} \SGfastsx{t}{x} \phi^{x}(y)\right)$ are continuous in $x$. In order to see this, notice that we can extend the result $(i)$ from Proposition \ref{lemma:IntAvgDerRepresentation} to all $\psi \in C^2(\R^n\times \R^d)$ with $\lvert \psi \rvert_{2,\pow{}{x}, \pow{}{y}}< \infty$ by taking a sequence $\phi_k \in \funcSpace{\pow{}{x}}{\pow{}{y}}$ that converges pointwise to $\psi$ and such that $\lvert \psi -  \phi_k \rvert_{2,\pow{\psi}{x}, \pow{\psi}{y}} \rightarrow 0$ as $k\rightarrow \infty$. Indeed, using the DCT (justified by the uniform bound  obtained from Proposition \ref{prop:derivativeest}) concludes that we can extend $(i)$ from Proposition \ref{lemma:IntAvgDerRepresentation}. Next, use this along with the DCT (justified by the uniform bounds  \eqref{eqn:term2} and \eqref{eqn:term3}) to obtain that \eqref{eqn:term0} and \eqref{eqn:termdyy} are continuous in $x$ as required. This concludes the proof of \emph{(i)}. 
		We move on to proving some preliminary bounds that will be useful for proving \emph{(ii)}.
		First, similarly to \eqref{eqn:SGdecay},
		\begin{equation}
			\begin{split}\label{eqn:SGdecay1}
				\left|\LfastxDeriDerj \SGfastsx{s}{x}\phi^{x}(y)\right| &=
				\sum_{k=1}^d\left|\slowDerDer{i}{j}\drift_k(x,y)\fastDer{k}\SGfastsx{s}{x}\phi^{x}(y)\right| + \sum_{k,l=1}^d\left| \slowDerDer{i}{j} \diffusion(x)_{kl}\fastDerDer{k}{l}\SGfastsx{s}{x}\phi^{x}(y)\right|\\ 
				&\leq C\Vseminorm{\phi}{4,\pow{}{x},\pow{}{y}}e^{-cs} (1+ |y|^{M^g_y}+|x|^{2\pow{}{x}}).
			\end{split}
		\end{equation}
		We also have, by direct calculation and \eqref{eqn:intAvgDerFast}
		\begin{equation}
			\begin{split}\label{eqn:SGdecay2}
				\left|\LfastxDeri \partial_{x_i} \SGfastsx{t}{x} \phi^{x}(y) \right| &= \sum_{j=1}^d\left|\slowDer{i}\drift_j(x,y)\fastDer{j}\partial_{x_i} \SGfastsx{s}{x} \phi^{x}(y)\right| + \sum_{j,k = 1}^d\left| \slowDer{i} \diffusion(x)_{jk}\fastDerDer{j}{k} \partial_{x_i} \SGfastsx{s}{x} \phi^{x}(y)\right|\\ 
				&\leq C\Vnorm{\phi}{4,\pow{}{x},\pow{}{y}}e^{-cs} (1+ |y|^{M^g_y + \pow{g}{y}}+|x|^{4\pow{}{x}}).
			\end{split}
		\end{equation}
		Finally, with (\ref{eqn:SGdecay1}) and (\ref{eqn:SGdecay2}) we have
		\begin{equation}\label{eqn:SGdecay3}
			\begin{split}
				\bigg|\LfastxDeriDerj \SGfastsx{s}{x}\phi^{x}(y) &+ \LfastDerj{x} \partial_{x_j} \SGfastsx{s}{x} \phi^{x}(y)  + \LfastDeri{x} \partial_{x_i} \SGfastsx{s}{x} \phi^{x}(y) \bigg| \\
				&\leq \left|\LfastxDeriDerj \SGfastsx{s}{x}\phi^{x}(y)\right| + \left| \LfastDerj{x} \partial_{x_i} \SGfastsx{s}{x} \phi^{x}(y)\right|  +\left| \LfastDeri{x} \partial_{x_j} \SGfastsx{s}{x} \phi^{x}(y) \right| \\
				&\leq C\Vnorm{\phi}{4,\pow{}{x},\pow{}{y}} e^{-cs} (1+ |y|^{M^g_y + \pow{g}{y}}+|x|^{4\pow{}{x}}).
			\end{split}
		\end{equation}
		Now we outline the strategy of proof for \emph{(ii)}, making the note similarly to the start of the proof of Proposition \ref{lemma:IntAvgDerRepresentation} that here and throughout $h$ is in fact $he_j$.
		\begin{itemize}
			\item Step $1$: Define the difference quotient
			\begin{equation}\label{eqn:qbar}
				\widebar{q}_{t}^{h,\phi^x}(y) \coloneqq \frac{\slowDer{i} \left(\SGfastsx{t}{x+h}\phi^{x+h}\right)(y) - \slowDer{i} \left(\SGfastsx{t}{x}\phi^x\right)(y)}{h}
			\end{equation}
			and obtain the representation formula
			\begin{equation}\label{eqn:quotientrepDerDer}
				\begin{split}
					\widebar{q}_{t}^{h,\phi^x}(y) = \SGfastsx{t}{x}\widebar{q}_{0}^{h,\phi^x}(y) + \int_{0}^{t} \SGfastsx{t-s}{x}&\bigg[\LfastDeri{x+h} q_{s}^{h,\phi^x}(y) + \frac{1}{h}\left( \LfastDeri{x+h} -\LfastDeri{x}\right)\SGfastsx{s}{x}\phi^{x}(y) \\
					&+ \left(\frac{\Lfastx{x+h} - \Lfastx{x}}{h}\right)\slowDer{i} \left(\SGfastsx{s}{x+h}\phi^{x+h}\right)(y) \bigg]ds.
				\end{split}
			\end{equation}
			\item Step $2$: Let $h\rightarrow 0$ in \eqref{eqn:quotientrepDerDer} to obtain \eqref{eqn:repSGderder}, and prove the continuity of the LHS of \eqref{eqn:repSGderder}.
			\item Step $3$: Let $t \rightarrow \infty$ in \eqref{eqn:repSGderder} to obtain \eqref{eqn:invMeasDerDer}. The bound \eqref{intAvgDerDerInv} is then a consequence of \eqref{eqn:invMeasDerDer}.
			\item Step $4$: Integrate $	\partial_{x_i x_j} \left(\SGfastsx{t}{x}\phi^x\right)(y) -	\partial_{x_i x_j}\mu^x\left(\phi^x\right)$ with respect to $t$ using the representation formulas \eqref{eqn:repSGderder}-\eqref{eqn:invMeasDerDer} to obtain \eqref{eqn:intAvgDerDer}.
		\end{itemize}

		Step 1: We differentiate \eqref{eqn:qbar} in time using the formulas from Proposition \ref{lemma:IntAvgDerRepresentation} to get
		\begin{align*}
			\partial_{t} \widebar{q}_{t}^{h,\phi^x}(y) &= \frac{1}{h}\left( \LfastDeri{x+h}\SGfastsx{t}{x+h}\phi^{x+h}(y) + \Lfastx{x+h}\slowDer{i} \left(\SGfastsx{t}{x+h}\phi^{x+h}\right)(y) \right) \\
			&- \frac{1}{h}\left(\LfastDeri{x}\SGfastsx{t}{x}\phi^{x}(y) +  \Lfastx{x}\slowDer{i} \left(\SGfastsx{t}{x}\phi^x\right)(y) \right) \\
			&= \LfastDeri{x+h} \left( \frac{\SGfastsx{t}{x+h}\phi^{x+h}- \SGfastsx{t}{x}\phi^x}{h} \right)(y) + \frac{1}{h}\left( \LfastDeri{x+h} -\LfastDeri{x}\right)\SGfastsx{t}{x}\phi^{x}(y) \\
			&+ \left(\frac{\Lfastx{x+h} - \Lfastx{x}}{h}\right)\slowDer{i} \left(\SGfastsx{t}{x+h}\phi^{x+h}\right)(y) \\
			&+ \Lfastx{x}\left(\frac{\slowDer{i} \left(\SGfastsx{t}{x+h}\phi^{x+h}\right)(y) - \slowDer{i} \left(\SGfastsx{t}{x}\phi^x\right)(y)}{h}\right),
		\end{align*}
		which, using \eqref{eqn:qbar} and recalling the definition of $q_{t}^{h,\phi^x}(y)$ from \eqref{eqn:quotient}, gives us
		\begin{align*}
			\partial_{t} \widebar{q}_{t}^{h,\phi^x}(y) &=\LfastDeri{x+h} q_{t}^{h,\phi^x}(y) + \frac{1}{h}\left( \LfastDeri{x+h} -\LfastDeri{x}\right)\SGfastsx{t}{x}\phi^{x}(y) \\
			&+ \left(\frac{\Lfastx{x+h} - \Lfastx{x}}{h}\right)\slowDer{i} \left(\SGfastsx{t}{x+h}\phi^{x+h}\right)(y) + \Lfastx{x}\widebar{q}_{t}^{h,\phi^x}(y).
		\end{align*}
		Using variation of constants we get \eqref{eqn:quotientrepDerDer}.
		
		Step $2$: We let $h \rightarrow 0$ in a very similar way to the proof of Proposition \ref{lemma:IntAvgDerRepresentation}:
		\begin{align*}
			\lim_{h\rightarrow 0}\widebar{q}_{t}^{h,\phi^x}(y) &
			=  (\SGfastsx{t}{x}\slowDerDer{i}{j}\phi^x)(y) \\
			& + \int_{0}^{t} \SGfastsx{t-s}{x}\Big[\LfastxDeriDerj \SGfastsx{s}{x}\phi^{x}(y) + \LfastDerj{x}\slowDer{j} \left(\SGfastsx{s}{x}\phi^{x}\right)(y) +  \LfastDeri{x}\slowDer{i} \left(\SGfastsx{s}{x}\phi^{x}\right)(y) \Big]ds
		\end{align*}
		to conclude (\ref{eqn:repSGderder}). Indeed, to take $h \rightarrow 0$ we get pointwise convergence under the integral by the smoothness of our coefficients, and the continuity in $x$ of $\fastDer{i}\SGfastsx{s}{x}\phi^x(y)$, $\fastDerDer{i}{j}\SGfastsx{s}{x}\phi^x(y)$, $\fastDer{i}\slowDer{j}\SGfastsx{s}{x}\phi^x(y)$ and $\fastDerDer{i}{j}\slowDer{k}\SGfastsx{s}{x}\phi^x(y)$, and then use the DCT. 
		
		We use \eqref{eqn:SGdecay3} with \eqref{eqn:repSGderder} to prove continuity in $x$ of $\slowDerDer{i}{j} \left(\SGfastsx{t}{x}\phi^x\right)(y)$. Indeed, this is the uniform bound needed to apply the DCT along with the smoothness of our coefficients as well as the continuity in $x$ of $\fastDer{i}\SGfastsx{s}{x}\phi^x(y),\fastDerDer{i}{j}\SGfastsx{s}{x}\phi^x(y),\fastDer{i}\slowDer{j}\SGfastsx{s}{x}\phi^x(y)$ and $\fastDerDer{i}{j}\slowDer{k}\SGfastsx{s}{x}\phi^x(y)$. Hence, we have continuity. 
		
		Step $3$: To let $t \rightarrow \infty$ in \eqref{eqn:repSGderder}, we must justify the use of the DCT as we had to in the proof of Proposition \ref{lemma:IntAvgDerRepresentation}. Letting $t \rightarrow \infty$ in the first term of \eqref{eqn:repSGderder} is straightforward from Proposition \ref{lemma:IntAvg}. With respect to the second addend, it is very similar to the same limit for the second addend of \eqref{eqn:repSG} in the proof of Proposition \ref{lemma:IntAvgDerRepresentation}. This time, we use \eqref{eqn:SGdecay3} to justify the use of the DCT to pass the limit inside the time integral. Therefore, we have proved that the RHS of \eqref{eqn:repSG} tends to the RHS of \eqref{eqn:repInvMeas} as $t \rightarrow \infty$. As in Step $3$ in the proof of Proposition \ref{lemma:IntAvgDerRepresentation}, we now wish to do the same for the LHS which involves exchanging the limit $t\rightarrow\infty$ with the derivative $\slowDerDer{i}{j}$. We have
		\begin{align*}
			&\lim_{t\rightarrow\infty}[\slowDerDer{i}{j}\left(\SGfastsx{t}{x}\phi^x\right)(y) ]-\slowDerDer{i}{j}{ \left(\SGfastsx{t}{x}\phi^x\right)(y)}   \\
			&= \int_{\mathbb{R}^d}\slowDerDer{i}{j}\phi^x(y)\mu^x(dy)  -(\SGfastsx{t}{x}\slowDerDer{i}{j}\phi^x)(y) \\
			&+ \int_{0}^{\infty}\!\! \int_{\mathbb{R}^d}\!\!\left[\LfastxDeriDerj \SGfastsx{s}{x}\phi^{x}(y) + \LfastDerj{x} \partial_{x_j} \SGfastsx{s}{x} \phi^{x}(y)  + \LfastDeri{x} \partial_{x_i} \SGfastsx{s}{x} \phi^{x}(y) \right]\mu^x(dy) ds\\
			&- \int_{0}^{t} \SGfastsx{t-s}{x}\left[\LfastxDeriDerj \SGfastsx{s}{x}\phi^{x}(y) + \LfastDerj{x} \partial_{x_j} \SGfastsx{s}{x} \phi^{x}(y)  + \LfastDeri{x} \partial_{x_i} \SGfastsx{s}{x} \phi^{x}(y) \right]ds \\
			&= \fastExp{y} \left[  \mu^x\left(\slowDerDer{i}{j}\phi^x\right) -\slowDerDer{i}{j}\phi^x(y) \right] \\
			&+ \underbrace{\int_0^{t} \left(\mu^x - \SGfastsx{t-s}{x}\right)\left[\LfastxDeriDerj \SGfastsx{s}{x}\phi^{x} + \LfastDerj{x} \partial_{x_j} \SGfastsx{s}{x} \phi^{x}  + \LfastDeri{x} \partial_{x_i} \SGfastsx{s}{x} \phi^{x} \right] ds}_{\eqqcolon \RNum{1}} \\
			&+\underbrace{\int_{t}^{\infty} \mu^x\left(\LfastxDeriDerj \SGfastsx{s}{x}\phi^{x} + \LfastDerj{x} \partial_{x_j} \SGfastsx{s}{x} \phi^{x}  + \LfastDeri{x} \partial_{x_i} \SGfastsx{s}{x} \phi^{x} \right) ds.}_{\eqqcolon \RNum{2}}
		\end{align*}
		We now show that each of the addends in the above converge to zero locally uniformly in $x$ and $y$. Similarly to what we have done in the analogous stage of the proof of Proposition \ref{lemma:IntAvgDerRepresentation}, the claim is trivial for the first term, since by Proposition \ref{lemma:IntAvg},
		\begin{equation}\label{eqn:unifConvDerDer}
			\left(\SGfastsx{s}{x} \slowDerDer{i}{j}\phi^x\right) (y) - \mu^x(\slowDerDer{i}{j}\phi^x) \leq C\Vnorm{\phi}{4,\pow{}{x},\pow{}{y}}e^{-cs}(1+ |y|^{\pow{}{y}}+|x|^{\pow{}{x}})
		\end{equation}
		If we show the following bounds, the proof is concluded by integrating in $t$:
		\begin{align}\label{eqn:int3}
			|\RNum{1}| &\leq  C\Vnorm{\phi}{4,\pow{}{x},\pow{}{y}}te^{-ct} (1 + |y|^{M^g_y + \pow{g}{y}}+|x|^{4\pow{}{x}}) , \\
			\label{eqn:int4}
			|\RNum{2}| &\leq C\Vnorm{\phi}{4,\pow{}{x},\pow{}{y}}e^{-ct}(1+|x|^{4\pow{}{x}}),
		\end{align}
		
		Now we proceed with proving \eqref{eqn:int3}. With \eqref{eqn:SGdecay3} and Proposition \ref{lemma:IntAvg} we write (with $\tilde{c}$ denoting the constant in \eqref{lemma:IntAvg})
		\begin{align*}
			|\RNum{1}| 
			&\leq e^{-ct}\int_0^{t} Ce^{(c-\tilde{c})s}\Vnorm{\phi}{4,\pow{}{x},\pow{}{y}} (1 + |y|^{M^g_y + \pow{g}{y}}+|x|^{4\pow{}{x}}) ds \\
			&\leq Ce^{-ct}\Vnorm{\phi}{4,\pow{}{x},\pow{}{y}} (1 + |y|^{M^g_y + \pow{g}{y}}+|x|^{4\pow{}{x}}).
		\end{align*}
		and so \eqref{eqn:int3} is shown. Now we show \eqref{eqn:int4}.
		By (\ref{eqn:SGdecay3}) and subsequently (\ref{eqn:momBoundsYinv}) from Lemma \ref{prop:momentBounds}, we have
		\begin{align*}
			|\RNum{2}| &\leq \Vnorm{\phi}{4,\pow{}{x},\pow{}{y}}\int_{t}^{\infty} \int_{\mathbb{R}^d}\left[ Ce^{-cs} (1 + |y|^{M^g_y + \pow{g}{y}}+|x|^{4\pow{}{x}})\right]\mu^x(dy)ds \\
			&\leq \Vnorm{\phi}{4,\pow{}{x},\pow{}{y}}\int_{t}^{\infty} Ce^{-cs}(1+|x|^{4\pow{}{x}}) ds \\
			&\leq  C\Vnorm{\phi}{4,\pow{}{x},\pow{}{y}}e^{-ct}(1+|x|^{4\pow{}{x}}),
		\end{align*}
		so that \eqref{eqn:int4} is shown.
		Hence, with \eqref{eqn:unifConvDerDer} (\ref{eqn:int3}) and (\ref{eqn:int4}), and subsequently Proposition \ref{lemma:IntAvg} we have
		\begin{equation}
			\label{eqn:justificationderder}\begin{split}
				&\big| \slowDerDer{i}{j}{ \left(\SGfastsx{t}{x}\phi^x\right)(y)} -\lim_{t\rightarrow\infty}\slowDerDer{i}{j}\left(\SGfastsx{t}{x}\phi^x\right)(y) \big| \\ &\leq C\Vnorm{\phi}{4,\pow{}{x},\pow{}{y}}e^{-ct}(1+ |y|^{\pow{}{y}}+|x|^{\pow{}{x}}) \\
				&+ C\Vnorm{\phi}{4,\pow{}{x},\pow{}{y}}e^{-ct}(1 + |y|^{M^g_y + \pow{g}{y}}+|x|^{4\pow{}{x}}) \\
				&\leq C\Vnorm{\phi}{4,\pow{}{x},\pow{}{y}}e^{-ct}(1 + |y|^{M^g_y + \pow{g}{y}}+|x|^{4\pow{}{x}})
			\end{split}
		\end{equation}
		for some constants $C, c$. Hence, the convergence is locally uniform and we have that $\lim_{t\rightarrow\infty}\slowDerDer{i}{j}\left(\SGfastsx{t}{x}\phi^x\right)(y) = \slowDerDer{i}{j}\mu^x(\phi^x)$. Hence \eqref{eqn:invMeasDerDer} holds. Now we show \eqref{intAvgDerDerInv}. Using (\ref{eqn:momBoundsYinv}) of Lemma \ref{prop:momentBounds}, (\ref{eqn:invMeasDerDer}), (\ref{eqn:SGdecay3}), and the assumption that $\phi \in \funcSpace{m_1}{m_2}$ we have
		\begin{align*}
			|\slowDerDer{i}{j}\mu^x(\phi^x)| &\leq  \left|\int_{\mathbb{R}^d}\slowDerDer{i}{j}\phi^x(y)\mu^x(dy)\right| \\
			&+ \int_{0}^{\infty} \int_{\mathbb{R}^d}\left|\LfastxDeriDerj \SGfastsx{s}{x}\phi^{x}(y) +\LfastDerj{x} \partial_{x_j} \SGfastsx{s}{x} \phi^{x}(y)  + \LfastDeri{x} \partial_{x_i} \SGfastsx{s}{x} \phi^{x}(y) \right|\mu^x(dy) ds \\
			&\leq C\Vnorm{\phi}{4,\pow{}{x},\pow{}{y}}(1+|x|^{\pow{}{x}})\\
			&+ C\Vnorm{\phi}{4,\pow{}{x},\pow{}{y}}	\int_{0}^{\infty}
			\int_{\mathbb{R}^d}e^{-cs} (1+ |y|^{M^g_y + \pow{g}{y}}+|x|^{4\pow{}{x}})\mu^x(dy) ds
		\end{align*}
		Now we use \eqref{eqn:momBoundsYinv} from Lemma \ref{prop:momentBounds} again to obtain 
		\begin{align*}
			\int_0^{\infty} \int_{\mathbb{R}^d} Ce^{-cs} (1+ |y|^{M^g_y + \pow{g}{y}}+|x|^{4\pow{}{x}})\mu^x(dy) ds &\leq \int_0^{\infty} Ce^{-cs} (1+|x|^{4\pow{}{x}})ds \\
			&\leq C(1+|x|^{4\pow{}{x}})
		\end{align*}
		so that we have shown
		\begin{equation*}
			\left|\slowDerDer{i}{j}\int_{\mathbb{R}^d} \phi^{x}(y) \mu^x(dy)\right| \leq C\Vnorm{\phi}{4,\pow{}{x},\pow{}{y}} (1+|x|^{4\pow{}{x}}),
		\end{equation*}
		which proves \eqref{intAvgDerDerInv}.
		Since the convergence as $t \rightarrow \infty$ is locally uniform,  we also have continuity of $\slowDerDer{i}{j}\mu^x (\phi^x)$.
		
		Step $4$: We proceed in a similar vein to the proof of Proposition \ref{lemma:IntAvgDerRepresentation}.
		We note that the expression we wish to control is
		\begin{equation*}
			\left| \slowDerDer{i}{j}\int_0^\infty \fastExp{y} \left[ \mu^x(\phi^x) - \phi^{x}(\fastFixedI_{t})\right]dt \right| = \left|\int_0^\infty \slowDerDer{i}{j}\mu^x(\phi^x) -\slowDerDer{i}{j}{ \left(\SGfastsx{t}{x}\phi^x\right)(y)}  dt\right|.
		\end{equation*}
		This means that we have, using \eqref{eqn:justificationderder},
		\begin{align*}
			\sum^n_{i,j=0}\bigg| \slowDerDer{i}{j}\int_0^\infty \fastExp{y} &\left[\mu^x(\phi^x_j) - \phi_j(x,\fastFixedI_{t})\right]dt \bigg| \\
			&\leq \sum^n_{i,j=0} \int_0^\infty |\slowDerDer{i}{j}\mu^x(\phi^x) -\slowDerDer{i}{j}{ \left(\SGfastsx{t}{x}\phi^x\right)(y)} | dt \\
			&\leq  C\Vnorm{\phi}{4,\pow{}{x},\pow{}{y}}(1 + |y|^{M^g_y + \pow{g}{y}}+|x|^{4\pow{}{x}}),
		\end{align*}
		and so the proof is complete.
	\end{proof}

	\subsection{Proofs of Section \ref{subsec:detailedproof}}\label{app:proofsaveraging}
	
	\begin{proof}[Proof of Lemma \ref{prop:fullMombound}]
		Let $V_1(x) = |x|^{4\pow{b}{x}}$ and $V_2(y) = |y|^{k}$ for some $k>0$. Using Assumption \ref{DriftAssump}, \ref{DriftAssumpS} and \eqref{decompLepsilon}, there exist $\tilde{r}', \tilde{C}', \tilde{r}_k',\tilde{C}'_k>0$ such that
		\begin{equation*}
			(\mathfrak{L}_{\epsilon} V_1)(x,y) \leq -\tilde{r}'V_1(x)+\tilde{C}', \quad (\mathfrak{L}_{\epsilon} V_2)(y) \leq -\frac{r_{k}'}{\epsilon}V_2(y)+\frac{C_k'}{\epsilon}.
		\end{equation*}
		By the same argument as in the proof of Lemma \ref{prop:momentBounds} this gives
		\begin{equation}\label{momBoundx}
			\cP_t^\epsilon V_1(x,y) \leq e^{-\tilde{r}' t}V_1(x) +\frac{\tilde{C}'}{\tilde{r}'}, \quad \cP_t^\epsilon V_2(y) \leq e^{-\frac{r_k'}{\epsilon} t}V_2(y) +\frac{C_k'}{r_k'}.
		\end{equation}
		Hence, taking $V(x,y)=V_1(x)+V_2(y)$, using \eqref{momBoundx} with $\epsilon\leq 1$ we can conclude \eqref{momBoundFull}.
	\end{proof}
	
	\subsection{Proofs of Section \ref{sec:boundsonft1}}\label{app:proofsstrongergodicity}
	
	\begin{proof}[Proof of Proposition \ref{prop:avgderivativeest}]
		Observe that Proposition \ref{prop:avgderivativeest} follows from Theorem \ref{thm:lorenziderest}, provided the conditions of Theorem \ref{thm:lorenziderest} hold. Therefore, it is sufficient to verify \eqref{eqn:lorenziderestAsumpHess}-\eqref{eqn:lorenziderestAsump2}. First, we will find a polynomial $\tilde{R}(x)$ that satisfies \eqref{eqn:lorenziderestAsump1}. Then we will find the polynomial $R(x)$ that satisfies \eqref{eqn:lorenziderestAsumpHess}. Finally, we will verify that these polynomials satisfy \eqref{eqn:lorenziderestAsump2}.
		
		
		We write the LHS of \eqref{eqn:lorenziderestAsump1}, using \eqref{eqn:repInvMeas} with $\phi = b_j$, as
		\begin{align}\label{eqn:firstDerivbarb}
			\sum_{i,j=1}^n \partial_{x_i}\overline{b}_j(x) \xi_i\xi_j =	\mu^x &\left(\sum_{i,j=1}^n \partial_{x_i}b^x_j(y) \xi_i\xi_j\right) + \underbrace{\sum_{i,j=1}^{n}\left(\int_0^{\infty} \mu^x\left( \LfastxDeri \left(\SGfastsx{s}{x}b^{x}_j\right)\right) ds\right) \xi_i \xi_j}_{\RNum{1}}.
		\end{align}
		Now we bound \RNum{1} from \eqref{eqn:firstDerivbarb}. We have that \begin{equation}\label{eqn:boundonI}
			|\RNum{1}| \leq \sum_{i,j=1}^{n}\left|\int_0^{\infty} \mu^x\left( \LfastxDeri \left(\SGfastsx{s}{x}b^{x}_j\right)\right) ds\right| |\xi_i| |\xi_j|
		\end{equation}
		and \begin{align*}
			\left|\int_0^{\infty} \mu^x \left( \LfastxDeri \left(\SGfastsx{s}{x}b^{x}_j\right)\right) ds\right| \leq \int_0^{\infty} \mu^x\left(\sum_{k = 1}^{d}  \left|\slowDer{i}g_k(x,y)\fastDer{k}\SGfastsx{s}{x}b^{x}_j(y)\right| +  \sum_{k,l = 1}^{d}\left| \slowDer{i} A(x)_{kl}\fastDerDer{k}{l}\SGfastsx{s}{x}b^{x}_k(y)\right|\right) ds
		\end{align*}
		Using Assumption \ref{ass:polGrowth} \ref{ass:polGrowthDriftFast} and \eqref{eqn:SGderfor6.2} from Proposition \ref{prop:derivativeest} (with $\psi=b_j$, since $b_j \in C^{0,2}(\R^n\times\R^d)$ with $\Vseminorm{b_j}{2,\pow{b}{x},\pow{b}{y}}< \infty$  by Assumption \ref{ass:polGrowth} \ref{ass:polGrowthDrift}) we write
		\begin{equation}\label{eqn:driftPart}
			\begin{split}
				\sum_{k = 1}^{d}  &\left|\slowDer{i}g_k(x,y)\fastDer{k}\SGfastsx{s}{x}b^{x}_j(y)\right| \\
				&\leq D_0e^{-\cb s/2} \Vseminorm{b}{2,\ytwodery,\ytwoderyy} K_g\left( 1+|y|^{\onederxg} \right)\left( 1+|x|^{\ytwodery} + |y|^{\ytwoderyy}+\sqrt{\frac{C_{2\ytwoderyy}'}{r_{2\ytwoderyy}'}}\right) \\
				&\leq D_0e^{-\cb s/2} \Vseminorm{b}{2,\ytwodery,\ytwoderyy} K_g\Bigg[ 1+|x|^{\ytwodery} + |y|^{\ytwoderyy}+\sqrt{\frac{C_{2\ytwoderyy}'}{r_{2\ytwoderyy}'}} \\
				&+ |y|^{\onederxg}+|x|^{\ytwodery}|y|^{\onederxg} + |y|^{\ytwoderyy+\onederxg}+\sqrt{\frac{C_{2\ytwoderyy}'}{r_{2\ytwoderyy}'}}|y|^{\onederxg}\Bigg]
			\end{split}
		\end{equation}
		and similarly
		\begin{equation}\label{eqn:diffPart}
			\begin{split}
				\sum_{k,l = 1}^{d}&\left| \slowDer{i} A_{kl}(x)\fastDerDer{k}{l}\SGfastsx{s}{x}b^{x}_k(y)\right| \leq D_0e^{-\cb s/2}K_{A}\left(1+|x|^{\ytwodery}+|y|^{\ytwoderyy} + \sqrt{\frac{C_{2\ytwoderyy}'}{r_{2\ytwoderyy}'}}\right).
			\end{split}
		\end{equation}
		Using \eqref{eqn:driftPart}, \eqref{eqn:diffPart} and Proposition \ref{prop:momentBounds},
		\begin{align*}
			\left|\int_0^{\infty} \mu^x \left( \LfastxDeri \left(\SGfastsx{s}{x}b^{x}_j\right)\right) ds\right| &\leq \frac{2D_0}{\cb} \Vseminorm{b}{2,\ytwodery,\ytwoderyy}  \Bigg[\left( K_g+K_{A}\right)\left(1+2\sqrt{\frac{C_{2\ytwoderyy}'}{r_{2\ytwoderyy}'}}\right) \\
			&+ K_g\left(\frac{C'_{\ytwoderyy+\onederxg}}{r'_{\ytwoderyy+\onederxg}}+\frac{C'_{\onederxg}}{r'_{\onederxg}}+ \sqrt{\frac{C_{2\ytwoderyy}'}{r_{2\ytwoderyy}'}}\frac{C'_{\onederxg}}{r'_{\onederxg}}\right) \\
			&+\left( K_g+K_{A} + K_g\frac{C'_{\onederxg}}{r'_{\onederxg}} \right)|x|^{\ytwodery}\Bigg]
		\end{align*}
		so that, using \eqref{eqn:boundonI},
		\begin{equation}\label{eqn:boundonIfinal}\begin{split}
				|\RNum{1}| \leq n|\xi|^2\frac{2D_0}{\cb} \Vseminorm{b}{2,\ytwodery,\ytwoderyy}  &\Bigg[\left( K_g+K_{A}\right)\left(1+2\sqrt{\frac{C_{2\ytwoderyy}'}{r_{2\ytwoderyy}'}}\right) \\
				&+ K_g\left(\frac{C'_{\ytwoderyy+\onederxg}}{r'_{\ytwoderyy+\onederxg}}+\frac{C'_{\onederxg}}{r'_{\onederxg}}+ \sqrt{\frac{C_{2\ytwoderyy}'}{r_{2\ytwoderyy}'}}\frac{C'_{\onederxg}}{r'_{\onederxg}}\right) \\
				&+\left( K_g+K_{A} + K_g\frac{C'_{\onederxg}}{r'_{\onederxg}} \right)|x|^{\ytwodery}\Bigg].
			\end{split}
		\end{equation}
		By \eqref{eq:avgdriftcondition}, \eqref{eqn:firstDerivbarb} and \eqref{eqn:boundonIfinal}, \eqref{eqn:lorenziderestAsump1} holds with \begin{equation}\label{eqn:tildeR}
			\tilde{R}(x) = -\left(\zeta(1+|x|^{\yfourder} +|x|^{ \ytwoderyonex} +|x|^{  \ytwoderx}) + \frac{K_\Sigma^2n^3}{4\lambda_-}\right).\end{equation}
		The cases where at least one of $\ytwodery$, $\ytwoderyy$, and $\onederxg$ are equal to $0$ follows with the same $\tilde{R}(x)$ from the subsequent simplified version of \eqref{eqn:boundonIfinal}, where we recall that we define $\frac{C'_0}{r'_0} = 0$.
		Now we move onto showing \eqref{eqn:lorenziderestAsumpHess}.
		From \eqref{eqn:invMeasDerDer}, we have that \begin{align*}
			&|\slowDerDer{i}{j}\bar b_k(x)| \\
			&\leq \left|\mu^x\left(\slowDer{i} \slowDer{j}b_k^x\right)\right| + \underbrace{\int_0^{\infty} \mu^x\left( \left|\LfastxDeriDerj \SGfastsx{s}{x}b^{x}_k\right| +  \left|\LfastxDeri \slowDer{j}\SGfastsx{s}{x}b^{x}_k\right| + \left|\LfastDerj{x} \slowDer{i}\SGfastsx{s}{x}b^{x}_k\right| \right) ds }_{\RNum{2}}
		\end{align*}
		We consider the first addend of the right hand side of the above, using the bounds on $b$ outlined in Assumption \ref{ass:polGrowth} (from which \eqref{starstarAss1} follows) and Lemma \ref{prop:momentBounds}
		\begin{equation}\label{addend1}
			\left|\mu^x\left(\slowDerDer{i}{j}b_k^x\right)\right| \leq \Vnorm{\slowDerDer{i}{j}b_k^x}{0,\ytwoderx,\ytwoderxy}\left( 1+|x|^{\ytwoderx} +\frac{C_{\ytwoderxy}'}{r_{\ytwoderxy}'}\right)\leq \hat{C} \left( 1+|x|^{\ytwoderx}\right).
		\end{equation}
		Here $\hat{C}$ is a generic constant which may change line by line and is independent of $x\in\R^n,y\in\R^d$ but may depend on norms of the coefficients. 
		Considering the first term of \RNum{2}, similarly to the term \RNum{1},
		\begin{equation}\label{addend2}
			\begin{split}
				\int_0^{\infty} \mu^x \left( \left|\LfastxDeriDerj \SGfastsx{s}{x}b^{x}_k \right|\right) ds  \leq \hat{C} \left(1+|x|^{\ytwodery}\right)
			\end{split}
		\end{equation}
		Now, we consider the second term of \RNum{2}. Making the observation that $\LfastDeri{x}$ is a second order differential operator in the $y$ variable, we need a bound on the terms $\fastDer{j}\slowDer{i}\SGfastsx{t}{x}b_k^x$ and $\fastDerDer{j}{k}\slowDer{i}\SGfastsx{t}{x}b_k^x$. For these term, we recall the formulae \eqref{eqn:term0}-\eqref{eqn:termdyy}, setting $\phi^x(y) = b_k(x,y)$. Notice first that by Assumption \ref{ass:polGrowth} \ref{ass:polGrowthDrift}, $\Vseminorm{\partial_{x_i} b_k}{2,\ytwoderyonex,\ytwoderyoney} < \infty$ for all $i\in \{1,...,d\}$, and hence applying the semigroup derivative estimates to this term is immediate. For what comes next, we wish to apply \eqref{eqn:SGdecayConc2} to $\LfastxDeri \SGfastsx{s}{x}b^{x}_k$. To this end, we write the following bound
		\begin{align*}
			\max\{\left|\partial_{y_j}\LfastxDeri \SGfastsx{s}{x}b^{x}_k \right|&, \left|\partial_{y_j y_k}\LfastxDeri \SGfastsx{s}{x}b^{x}_k \right|\} \\ &\leq \hat{C} e^{-\cb s}(1+|y|^{  \yfourdery+  \onederxg}|x|^{  \yfourder}+|y|^{\ytwodery+  \onederxg}+|x|^{  \yfourder}) \\
			&\leq \hat{C} e^{-\cb s}\left(1+|x|^{\yfourder}\right)\left(1+|y|^{\ytwodery+  \onederxg} + |y|^{\yfourdery+  \onederxg}\right)
		\end{align*}
		meaning we can use Proposition \ref{prop:derivativeest} with $\psi = \LfastxDeri \SGfastsx{s}{x} b^{x}_k  $ to obtain \begin{equation}\label{eqn:sgder}\begin{split}
				\max\{\left|\partial_{y_j}\SGfastsx{t-s}{x}\LfastxDeri \SGfastsx{s}{x}b^{x}_k \right|&, \left|\partial_{y_j y_k}\SGfastsx{t-s}{x}\LfastxDeri \SGfastsx{s}{x}b^{x}_k \right|\} \\
				&\leq \hat{C} e^{-\cb t}\left(1+|x|^{\yfourder}\right)\left(1+|y|^{\ytwodery+  \onederxg} + |y|^{\yfourdery+  \onederxg}\right)\end{split}
		\end{equation} By differentiating (\ref{eqn:repSG}) twice, using Proposition \ref{prop:derivativeest} with $\psi = \slowDerDer{i}{j}b^x_k$ for the first addend and integrating \eqref{eqn:sgder} for the second, we have
		\begin{align*}
			\begin{split}
				\max\{\left|\partial_{y_i}\slowDer{i} \SGfastsx{s}{x}b^{x}_k\right|, \left|\partial_{y_j y_l}\slowDer{i} \SGfastsx{s}{x}b^{x}_k\right| \}
				&\leq Ke^{-\cb s}\left(1+|y|^{\ytwoderyoney}+|x|^{\ytwoderyonex}\right) \\&+ \hat{C} se^{-\cb s}\left(1+|x|^{\yfourder}\right)\left(1+|y|^{\ytwodery+  \onederxg} + |y|^{\yfourdery+  \onederxg}\right)
			\end{split}
		\end{align*}
		so, using Lemma \ref{prop:momentBounds},
		
		\begin{align}\label{addend3}
			\int_0^{\infty} \mu^x \left( \left|\LfastxDeri \slowDer{j} \SGfastsx{s}{x}b^{x}_k \right|\right) ds \leq \hat{C} (1+|x|^{  \yfourder}+|x|^{\ytwoderyonex}).
		\end{align}
		Now we can use (\ref{addend1}),(\ref{addend2}) and \eqref{addend3} to write
		\begin{align*}
			|\slowDerDer{i}{j}\bar b_k(x)| \leq \hat{C}(1+|x|^{  \yfourder}+|x|^{  \ytwoderyonex} +|x|^{  \ytwoderx} )
		\end{align*}
		so that \eqref{eqn:lorenziderestAsumpHess} holds, with
		\begin{equation}\label{eqn:hessbarb1}
			R(x) = \hat{C}(1+|x|^{  \yfourder}+|x|^{  \ytwoderyonex} +|x|^{  \ytwoderx} ).
		\end{equation} 
		We conclude \eqref{eqn:lorenziderestAsump2} with $L = \zeta /\hat{C}$ by using \eqref{eqn:tildeR}, \eqref{eqn:hessbarb1} and Assumption \ref{ass:avgSGcond}.
	\end{proof}

\end{appendix}

\bibliographystyle{plain} 
\bibliography{references.bib}       


\end{document}